\documentclass[11pt]{amsart}
\usepackage{amsfonts, mathtools, nccmath, amssymb, amsthm,  commath, bbm}
\usepackage{algorithm, algorithmicx, algpseudocode, float}
\usepackage{tikz, graphicx, subfig, booktabs}
\usepackage[titletoc, title]{appendix}
\usepackage{indentfirst}
\usepackage[english]{babel}
\usepackage[top=1in, bottom=1in, left=1in, right=1in]{geometry}
\graphicspath{ {./images/} }
\usepackage{color}
\usepackage{times}
\usepackage[foot]{amsaddr}
\usepackage{hyperref}

\hypersetup{colorlinks,breaklinks,
             linkcolor=blue,urlcolor=blue,
             anchorcolor=blue,citecolor=blue}
\usepackage{dsfont}
\usepackage{cite}

\newcommand{\one}{\mathds{1}}
\renewcommand{\phi}{\varphi}
\renewcommand{\L}{\mathcal{L}}
\renewcommand{\div}{\text{div}}
\renewcommand{\P}{\mathbb{P}}
\renewcommand{\S}{\mathbb{S}}
\newcommand{\E}{\mathbb{E}}
\newcommand{\Var}{\text{Var}}

\newcommand{\bnu}{\bar{\nu}_{\eps}(x^0)}
\renewcommand{\O}{\mathcal{O}}

\numberwithin{equation}{section}

\newcommand{\Rn}[1][n]{\ensuremath{\mathbb{R}^{#1}}}
\newcommand{\mux}[1][]{\ensuremath{E_{1}^{#1}}}
\newcommand{\falseneg}{\ensuremath{\mathbb{P}(\widehat{T}^1_{\eps,\rr}(x^{0})=1\;|\,  d_{\Omega}(x^{0})\geq 2\eps)}}
\newcommand{\falsepos}{\ensuremath{\mathbb{P}(\widehat{T}^1_{\eps,\rr}(x^{0})=0\;|\, d_{\Omega}(x^{0})\leq\eps)}}

\newcommand{\Sx}[1][]{\ensuremath{S_{n}^{#1}}}
\newcommand{\Sy}{\ensuremath{S_{n}^{d}}}
\newcommand{\Sxsum}{\ensuremath{\left(\sum_{j=1}^{d-1}\left(\Sx[j]\right)^2\right)^{1/2}}}
\newcommand{\wod}[1][d]{\ensuremath{\omega_{#1}}}
\newcommand{\rr}{\mathit{r}}

\DeclareMathOperator{\supp}{supp}

\definecolor{byzantine}{rgb}{0.74, 0.2, 0.64}
\definecolor{darkgreen}{rgb}{0.1,0.6,0.1}
\definecolor{darkred}{rgb}{0.6,0,0}
\definecolor{lightgray}{rgb}{0.5,0.5,0.5}
\newcommand{\Ds}[1]{{\color{darkgreen}#1}}

\newcommand{\Sa}[1]{{\color{blue}#1}}

\newcommand{\jc}{\color{orange}}
\DeclareMathOperator{\dist}{dist}

\newcommand{\nc}{\normalcolor}

\newcommand{\R}{\mathbb{R}}

\newcommand{\M}{\mathcal{M}}
\newcommand{\x}{\textbf{x}}
\newcommand{\X}{\mathcal{X}}
\newcommand{\te}{\textrm}

\newcommand{\eps}{\varepsilon}
\renewcommand{\epsilon}{\varepsilon}

\newtheorem{theorem}{Theorem}[section]
\newtheorem{lemma}[theorem]{Lemma}
\newtheorem{proposition}[theorem]{Proposition}
\newtheorem{corollary}[theorem]{Corollary}

\newtheorem{assumption}[theorem]{Assumption}
\theoremstyle{definition}
\newtheorem{remarkx}[theorem]{Remark}
\newenvironment{remark}{\pushQED{\qed}\remarkx}{\popQED\endremarkx} 

\newenvironment{proofsketch}{\pushQED{\qed}
  \proof}{\endproof}

\title{Boundary Estimation from Point Clouds: Algorithms, Guarantees and Applications }

\author{Jeff Calder} 
\address{J. Calder: School of Mathematics, University of Minnesota, 127 Vincent Hall, 206 Church St. S.E., Minneapolis, MN 55455}
\email{jwcalder@umn.edu}

\author{Sangmin Park} 
\author{Dejan Slep\v{c}ev}
\address{S. Park, D. Slep\v{c}ev: Department of Mathematical Sciences, Carnegie Mellon University, 5000 Forbes ave., Pittsburgh, PA 15213}
\email{sangminp@andrew.cmu.edu, slepcev@math.cmu.edu}
\thanks{\textbf{Acknowledgments.} JC was supported by NSF grant DMS 1944925, the Alfred P.~Sloan Foundation, and a McKnight Presidential Fellowship. SP and DS were supported by NSF grant DMS 1814991. The authors would like to thank Eddie Aamari for valuable comments. 
The authors are grateful to CNA of CMU, IMA of Univ. of Minnesota, and Simons Institute at UC Berkeley for hospitality.}
\date{\today}

\begin{document}
\maketitle



\begin{abstract}
We investigate identifying  the boundary of a domain from sample points in the domain. 
We introduce new estimators for the normal vector to the boundary, distance of a point to the boundary, and a test for whether a point lies within a boundary strip. 
The estimators can be efficiently computed and are more accurate than the ones present in the literature. 
We provide rigorous error estimates for the estimators. 
Furthermore we use the  detected boundary points to solve boundary-value problems for PDE on point clouds. We prove error estimates for the Laplace and eikonal equations on point clouds. Finally we provide a range of numerical experiments illustrating the performance of our boundary estimators, applications to PDE on point clouds, and tests on image data sets. 
\end{abstract}
\medskip

\medskip

\subsection*{Notation}
	\begin{itemize} \addtolength{\itemsep}{2pt}
	\addtolength{\itemindent}{40pt}
		\item[$\Omega$:] bounded domain in $\R^d$. We denote the volume of $\Omega$ by $|\Omega|$.
		\item[$R$:]  lower bound for the reach of $\partial \Omega$.
		\item[$d_{\Omega}:$] the distance function $d_{\Omega}=\dist(x,\partial\Omega):\Omega\rightarrow \R_+$.
		\item[$\partial_{a}\Omega$:] boundary region $\partial_{a}\Omega: =\{x\in\Omega: \dist(x,\partial\Omega)\leq a\}$ for $a>0$.
		\item[$\wod$:]  volume of the unit ball in $\Rn[d]$.
		\item[$\rho$:]  probability density function $\rho:\Omega\rightarrow [\rho_{\min},\rho_{\max}]$ where $0<\rho_{\min}\leq \rho_{\max} < \infty$.
		\item[$L$:]  upper bound for the Lipschitz constant of $\rho$. 
		\item[$\X$:] $=\{x^1,\cdots,x^n\}$: set of \emph{i.i.d.}~sample points drawn from density $\rho$. 
		\item[$n$:]  total number of sample points considered.
		\item[$\rr$:]  neighborhood radius.
		\item[$\eps$:]  thickness of the boundary region we seek to identify.
        \item[$\nu$:] inward unit normal vector to $\partial\Omega$, extended to $\partial_R \Omega$ by \eqref{def:nu_star}.
        \item[$\bar{v}_{\rr},\,\bar{\nu}_{\rr}$:]  population-based estimator of the normal vector, and its unit normalization, \eqref{eq:v_bar}.
		\item[$\hat{v}_{\rr},\,\hat{\nu}_{\rr}$:]  first-order empirical estimator of the normal vector, and its unit normalization, \eqref{eq:v_hat}.
		\item[$\hat{v}^{2}_{\rr},\,\hat{\nu}^{2}_{\rr}$:]  second-order empirical estimator of the normal vector, and its unit normalization, \eqref{eq:hatvn}. \item[$\hat{d}_\rr^1(x^0), \hat{d}_\rr^2(x^0)$] first and second-order estimators of the distance to boundary of $\Omega$,  \eqref{eq:D_hat} and \eqref{eq:Deps2}.
        \item[$C_x, C_y, C_r$:] 	dimensionless constants explicitly stated  in Appendix \ref{appendixConstants}.
\end{itemize}

\section{Introduction}

We focus on determining the boundary of a domain given sample points in the domain. By determining the boundary we mean identifying the points which lie within an $\eps>0$ neighborhood of the boundary; see Figure \ref{fig:test_illustration} for illustration. Our aim is develop an algorithm that is efficient to compute, accurate (so that the boundary strip can be identified even for $\eps>0$ which is smaller than the typical distance between neighboring sample points), and guarantees that we identify a high percentage of points that are within distance $\eps$, while misidentifying as few points as possible that are at distance greater than $2 \eps$ as boundary points. Having such a set is sufficient for imposing boundary values for computing solutions of PDE on point clouds. 
\begin{figure}[ht]
    \centering
    \includegraphics[trim={20pt 10pt 20pt  30pt}, height=7cm]{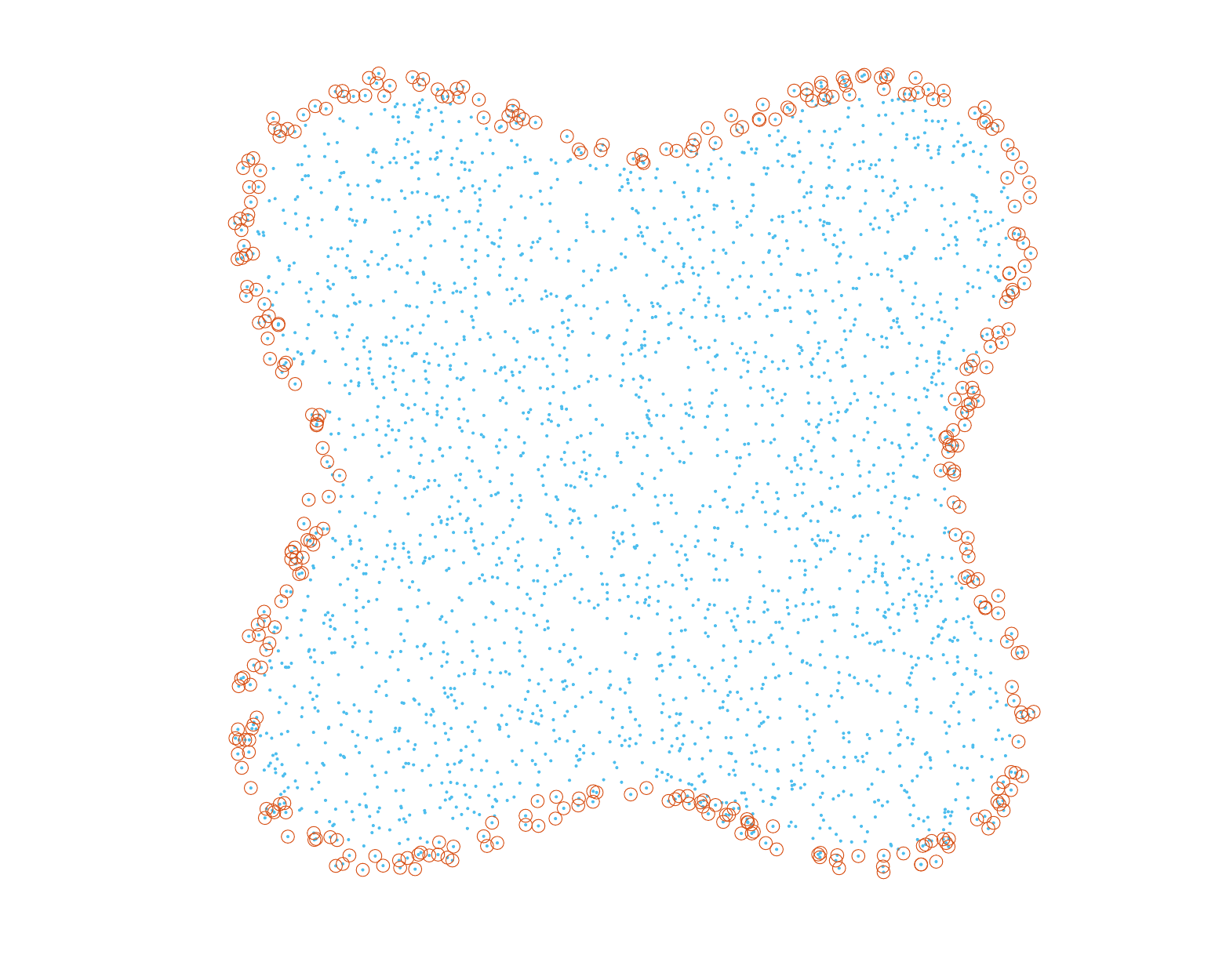}
    \vspace*{-12pt}
    \caption{Boundary points identified using the proposed  test \eqref{def:test}. } 
    \label{fig:test_illustration}
\end{figure}	

Estimating the boundary of the support of an unknown distribution and the normal vector to the boundary  are important and basic tasks with many applications.  Identification of boundary points are crucial to solving partial differential equations (PDEs) on data clouds \cite{CST20, liang2021solving, Shi17, vaughn2019diffusion}, and have applications such as detecting anomalies in a point cloud \cite{DevWis80} or assigning a notion of depth to each point (Section \ref{sec:realdata}). Estimation of the distance of each point to the boundary is also used to improve the accuracy of kernel distance estimators near the boundary \cite{BerrySauer17}. When the distribution is supported on a lower dimensional manifold, identifying points close to the boundary is important for estimation of the manifold itself. See \cite{AAL21} and references therein. 
While identifying the boundary of a point cloud is a basic problem, there are relatively few works that investigate the question in depth, see Section
\ref{sec:related-works}, and none satisfied the desired criteria above. In this work we introduce an approach that is simple, efficient, accurate and has the desired guarantees. 

Our approach is to first estimate the approximate normal vector to the boundary using a kernel average. 
In fact, in Section \ref{ssec:normalvec} we develop two such estimators: a first-order estimator,  given in \eqref{eq:v_hat}, which estimates the normal vector to first-order with respect to the kernel bandwidth, and a second-order estimator, given in \eqref{eq:hatvn}.  We use these normal vector estimators in Section \ref{ssec:dist} to define estimators for the distance to the boundary, \eqref{eq:D_hat} and \eqref{eq:Deps2},  which are, respectively, first and second-order accurate for points near the boundary. 
This allows us to define in Section \ref{ssec:bdrytest} the statistical test for the boundary strip in \eqref{def:test}. We implement our boundary test using MATLAB and Python, and make our code available on Github \footnote{\url{https://github.com/sangmin-park0/BoundaryTest}}.

In this work we provide rigorous non-asymptotic error bounds of the first-order estimators and only asymptotic estimates for the second-order estimators. We focus on the first-order estimators in this paper, since nonasymptotic bounds for the second-order versions would be highly complicated, involving nontrivial dependence on a large number of parameters, including higher order derivatives of the density $\rho$ and the boundary of $\Omega$, which the first-order estimators do not require.

In Sections \ref{ssec:normalvec} and \ref{ssec:dist} we motivate and define the normal vector and distance-to-boundary estimators. The estimates on the normal vector estimators are provided in Section \ref{sec:prelim}. Section \ref{sec:main} then establishes nonasymptotic estimates for the first-order test. In particular, the nonasymptotic error bounds on the distance estimator are provided in Theorem \ref{thm:main}, and Corollary \ref{corol:boundary_test} establishes the nonasymptotic estimates for the first-order test. Asymptotic error estimates for the second-order distance test are given in Section \ref{sec:2ndOrder}.

In Section \ref{sec:algorithm_experiments} we state our boundary tests in the form of a practical procedure, see Algorithm \ref{alg:1st} and Algorithm \ref{alg:2nd}. We conduct a number of experiments that illustrate the qualitative and quantitative  performance of the algorithms. We also discuss the optimal selection of parameters, in particular the bandwidth of the kernel. 

In Section \ref{sec:PDE_graph} we turn to applications of the boundary test towards solving PDE boundary value problems  using graph-based approximations, which is one of the problems that motivated our work. Since we estimate both the boundary points and the normal vector to the boundary, we are able to assign Dirichlet, Neumann, and Robin boundary conditions. In particular, we study the eikonal equation with Dirichlet boundary conditions and Poisson equations with Robin conditions on point clouds, and prove quantitative convergence rates to the solutions of the continuum PDEs. It is important to point out that not all methods for detecting boundary points will lead to convergent numerical approximations of PDEs. If too few points are identified, the boundary conditions may not be attained continuously as the mesh is refined \cite{CST20}. Similar problems can occur if points far inside the interior of the domain are falsely identified as boundary points. The purpose of this section is to illustrate that our boundary detection method is compatible with setting boundary conditions for PDEs on point clouds. Our results cover only some preliminary examples, with much investigation left to future work. 

Finally, in Sections \ref{sssec:1storderPDE} and \ref{sssection:2ndorderPDE} we implement numerical schemes for solving the eikonal and Robin equations on point clouds and conducted a number of experiments to both illustrate the solutions and numerically investigate the rate of convergence. Solving the eikonal equation enables us to estimate the distance to the boundary of any point in the dataset, which gives a notion of data depth on a point cloud. While our boundary test is not designed for working with manifolds in high dimensional spaces, Section \ref{sec:realdata} include experiments with notions of data depth based on the eikonal equation and Dirichlet eigenfunctions of the graph Laplacian on MNIST and FashionMNIST, using our boundary detection method to set the Dirichlet boundary conditions. The results are intriguing and agree with intuition; the boundary images are clearly outliers while the deepest images are good representatives of their class.

\subsection{ Setting }\label{ssec:setting}
Consider a domain $\Omega\subset\Rn[d]$ such that both $\Omega$ and $\R^{d}\setminus\Omega$ has reach at least $R>0$ , where reach is the maximal distance such that for all $x$ with $\dist(x,\Omega)\leq R$ there exists a \emph{unique} point $y \in \overline \Omega$ such that $|x-y| = \dist(x, \Omega)$. Denote by $\rho: \Rn[d] \rightarrow [0, \infty)$ a probability density function, which we assume satisfies $\rho_{\min}\leq\rho \leq \rho_{\max}$ on $\Omega$ for some positive numbers $\rho_{\min}\leq \rho_{\max}$ and $\rho =0$ outside of $\Omega$. 
We assume that on $\Omega$, the function $\rho$ is Lipschitz continuous with Lipschitz constant $L$. Given a set of \emph{i.i.d.}~points $\X$ distributed according to $\rho$, our goal is to identify the points that are close to the boundary $\partial \Omega$ with high probability; namely, we aim to approximate the set
\[\partial_\eps\Omega\cap\X=\{x\in \X: d_{\Omega}(x)\leq \eps\}\]
of $\eps$-boundary points, where $d_{\Omega}:\Omega\rightarrow\mathbb{R}_{+}$ is the distance function
\[d_{\Omega}(x):=\dist(x,\partial\Omega).\]

Our approach is as follows: we approximate inward normal vectors, use these to estimate the distance of each point to the boundary, and threshold the distance to obtain a boundary test.  For $x \in\partial\Omega$ we denote by $\nu(x)$ the unit inward normal to $\partial\Omega$ at $x$. We extend the unit normal to a vector field on the set $\partial_R \Omega$ by setting 
\begin{equation} \label{def:nu_star}
\nu(x)=\nu(x^*),
\end{equation}
where $x^*\in \partial\Omega$ is the closest point to $x$ on $\partial\Omega$. Note that $x^*$ is uniquely defined on $\partial_R \Omega$. We can also equivalently set $\nu(x) = \nabla  d_\Omega(x)$.

\subsection{Estimation of the inward normal vector} \label{ssec:normalvec}

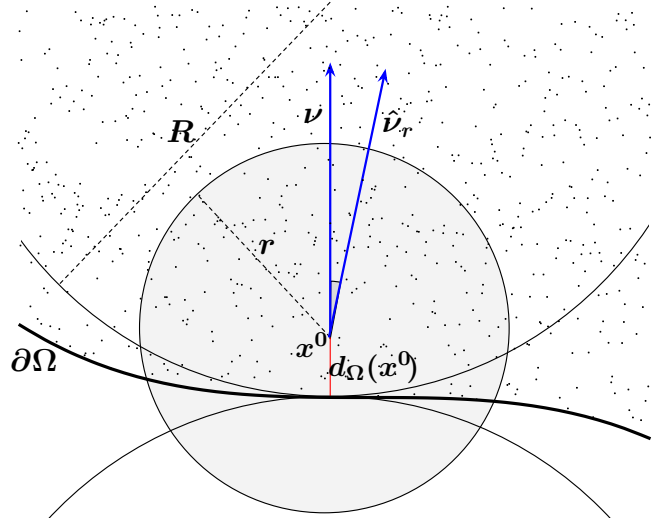
\begin{figure}[ht]  
	\centering
	\begin{tikzpicture}
	    \node at (0,0){\includegraphics[height=7cm]{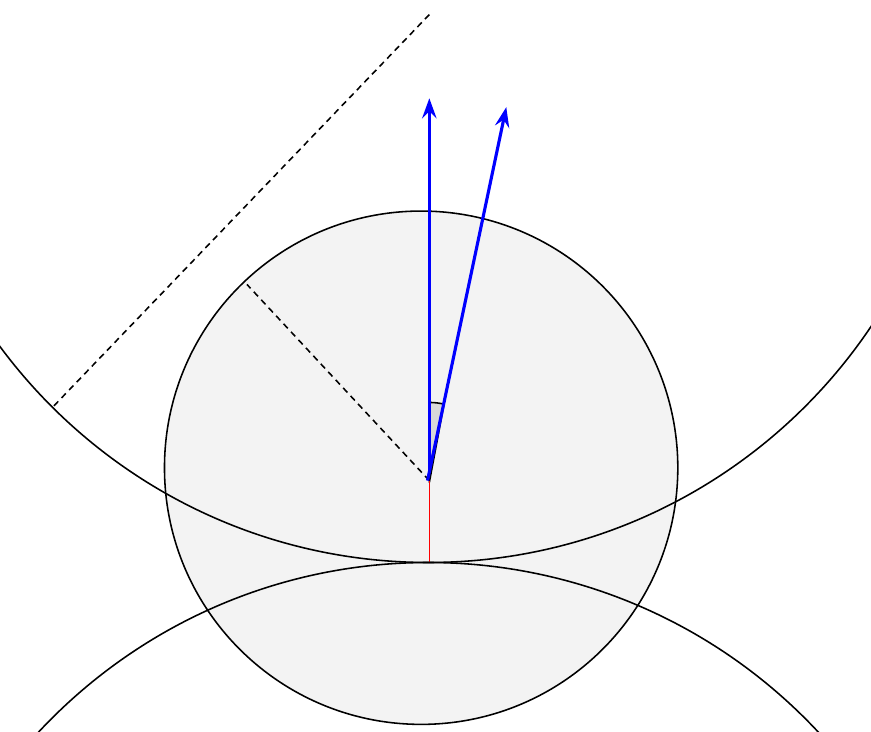}};
	    \node at (-0.3,-1.2){$\boldsymbol{x^0}$};
	    \node  at (0.85,1.75){\large $\boldsymbol{\hat{\nu}_{\rr}}$};
	    \node at (-0.28,1.85){\large $\boldsymbol{\nu}$};
	    \node at (-2.05,1.65){\large $\boldsymbol{R}$};
	    \node at (0.52,-1.5){$\boldsymbol{d_{\Omega}(x^0\!)}$};
	    \node at (-0.9,0.1){\large $\boldsymbol{r}$};
	    \node at (-4,-1.4){\large $\boldsymbol{\partial\Omega}$};
	    \draw[scale=1, domain=-4.2:4.2, smooth, variable=\x, very thick, black] plot ({\x}, {-0.40*(\x-0.4)*(\x-0.4)*(\x-0.4)/40-1.9});
	    \draw plot [only marks, mark=*, mark size=0.15, domain=-4.2:4.2, samples=700] (\x,{mod(5.7*rnd-0.40*(\x-0.4)*(\x-0.4)*(\x-0.4)/40-1.9,3.3)});
	\end{tikzpicture}
	    \caption{Illustration of the test setup: $x^0$ is the point tested.  }\label{fig_test_setup}
\end{figure}

We now introduce the first and second-order estimator of $\nu(x^0)$. These estimators are accurate when $x^0$ is near the boundary. This is sufficient as our test does not require any accuracy of the estimated normal vectors in the interior. In fact, even in the continuum case the normal vectors are not necessarily well-defined for points outside of $\partial_R \Omega$.
\medskip

\textbf{First-order normal vector estimator.} 
Let $\rr>0$ and $\X=\{x^1,x^2,\cdots,x^n\}$ be the set of \emph{i.i.d.}~points distributed according to $\rho$. For each $x^0\in \X$ we define the first-order normal vector estimator
\begin{equation}\label{eq:v_hat}
 \boxed{    \hat{v}_{\rr}(x^0)=\frac{1}{n}\sum_{i=1}^n\one_{B(x^0,\rr)}(x^i)(x^i-x^0),\qquad \hat{\nu}_{\rr}(x^0)=\frac{\hat{v}_{\rr}(x^0)}{|\hat{v}_{\rr}(x^0)|}}.
\end{equation}
 If $\hat v_r(x^0)= 0$ then we set $\hat\nu_r(x^0) = 0$. In this case, our test will identify $x^0$ as a boundary point. Note that this can happen with nonzero probability only when $x^0$ is an isolated point.
We also define the corresponding population level estimator
\begin{equation}\label{eq:v_bar}
    \bar{v}_{\rr}(x^0)=\int_{\Omega\cap B(x^0,\rr)}(x-x^0)\rho(x)\, dx,\qquad \bar{\nu}_{\rr}(x^0)=\frac{\bar{v}_{\rr}(x^0)}{|\bar{v}_{\rr}(x^0)|}.
\end{equation}
Theorem \ref{thm:normal_vector} establishes precise error bounds on the normal estimator, which in particular imply that 
\begin{equation}\label{eq:intro_normalVecRate}
        \mathbb{P}\left(|\hat{\nu}_{\rr}(x^0)-\nu(x^0)|>C\left(\frac{\log n}{n}\right)^{\frac{1}{d+2}}\right)\leq \frac{2d}{n^3}
\end{equation}
for $\rr\sim (\log n/n)^{1/(d+2)}$, where $C>0$ is a constant independent of $n$, with scaling $C\sim d^2$. 
\medskip

\textbf{Second-order normal vector estimator.}  
In addition to the assumptions for the first-order test, we now assume that $\rho$ is a $C^2$ function and that the boundary of $\Omega$ is a $C^3$ manifold. To reduce the bias that arises from the fact that $\rho$ is not  constant near $x^0$ we weight the points by the inverse of a kernel density estimate of $\rho$. For each $x^0 \in \X$  we define the second-order normal vector estimator
\begin{equation}\label{eq:hatvn}
\boxed{\hat{v}^{2}_r(x^0) = \frac{1}{n}\sum_{i=1}^n \frac{\one_{B(x^0,\rr)}(x^i)}{\hat \theta(x^i)}(x^i - x^0), \qquad \hat{\nu}^{2}_r(x^0) = \frac{\hat{v}^{2}_r(x^0)}{|\hat{v}^{2}_r(x^0)|},}
\end{equation}
where
\begin{equation}\label{def:theta;sample} \boxed{\hat \theta(x) = \frac{1}{\omega_d  n} \left(\frac{2}{\rr}\right)^{\!d}  \sum_{j=1}^n \one_{B(x,\rr/2)}(x^j).} \end{equation}
Similarly, we set $\hat\nu_r^2(x^0) = 0$ if $\hat{v}_\rr^2(x^0)= 0$.
We note that the radius for estimating $\theta$, namely $\frac{r}{2}$ is somewhat arbitrary. Using $r$
instead of $\frac{r}{2}$ results in the error of the same order, however in practice using  $r/2$ resulted in  smaller error than using $r$.

At the population level our estimator takes the form
\begin{equation}\label{eq:batvn}
\bar{v}^{2}_r(x^0) = \int_{B(x^0,r) \cap \Omega}  \frac{ \rho(x) }{\theta(x)}(x - x^0) dx, \qquad \bar{\nu}^{2}_r(x^0) = \frac{\bar{v}^{2}_r(x^0)}{|\bar{v}^{2}_r(x^0)|},
\end{equation}
where 
\begin{equation}\label{def:theta;population} \theta(x) = \frac{2^d}{\omega_d r^d} \, \int_{B(x,r/2) \cap \Omega}  \rho(z) dz. \end{equation}

In Section \ref{sec:asym2} we provide a proof that the error is indeed of size $r^2$ when $r \gtrsim (\log n/n)^{1/(d+4)}$, for $n$ large enough. In contrast to our results for the first-order test (Theorem \ref{thm:normal_vector}) we did not carry out a careful analysis of the second-order estimator to determine the exact constants appearing in the error bounds, and only determined the asymptotic scaling law. A more careful analysis of the second-order estimator is a nontrivial undertaking that we leave to future work.

We note that in addition to its use for distance estimation and the boundary test, the estimation of normal vectors is itself important to PDEs on graphs. It allows for the solution of PDEs on point cloud with not only Dirichlet boundary conditions but also Neumann, oblique, and Robin boundary conditions, which we study in  Section \ref{sec:PDE_graph}.

\subsection{Estimation of the distance to the boundary} \label{ssec:dist}

The distance to $\partial \Omega$, $\;d_\Omega : \Omega \to \R$, is differentiable in $ \partial_R \Omega$; see for example Lemma 2.21 in
\cite{BD-medial}. Furthermore, the gradient of the distance function conicides with the extension of the inward normal vector, that is, for $x\in \partial_R \Omega$ we have
\begin{equation}\label{eq:normal_graddist}
\nabla d_\Omega(x)=\nu(x).
\end{equation}
We exploit this relationship to approximate the distance function using the normal vectors near the boundary. First, we observe that $d_\Omega$ satisfies
\begin{equation}\label{eq:rep}
d_\Omega(x) = \max_{y\in B(x,\rr)\cap \Omega}\left\{ d_\Omega(x) - d_\Omega(y) \right\}
\end{equation}
provided $B(x,\rr)\cap \partial \Omega$ is not empty. Indeed, the maximum is attained at $y\in \partial \Omega$ where $d_\Omega(y)=0$. Suppose $d_\Omega\in C^2$ near the boundary. Then we can use the Taylor expansion 
\[d_\Omega(y) = d_\Omega(x) + \nabla d_\Omega(x)\cdot(y-x) + O(\rr^2)\]
in \eqref{eq:rep}, along with \eqref{eq:normal_graddist}, to obtain
\begin{equation}\label{eq:test1}
d_\Omega(x) = \max_{y\in B(x,\rr)\cap \Omega}\left\{\nu(x)\cdot  (x-y)\right\} + O(\rr^2).
\end{equation}
Replacing the true normal $\nu(x)$ in \eqref{eq:test1} with our first-order normal estimator $\hat{\nu}_{\rr}(x^0)$, and restricting the maximum to the point cloud, leads to our first-order estimator of the distance to the boundary.  
\medskip

\textbf{First-order estimator for the distance to the boundary of $\Omega$.}
Let $\rr>0$ and $\X=\{x^1,x^2,\cdots,x^n\}\subset\Omega$. We define the first-order distance function estimator  $\hat{d}_\rr^1: \X\rightarrow \mathbb{R}$ by
\begin{equation}\label{eq:D_hat}
\boxed{\hat{d}_\rr^1(x^0) = \max_{x^i\in B(x^0,\rr)\cap \X}(x^0-x^i)\cdot \hat{\nu}_{\rr}(x^0).}
\end{equation}
In Sections \ref{sec:prelim} and \ref{sec:main}, we show that the assumption that $\partial\Omega$ has positive reach guarantees the error rate $O(\rr^2)$ of the first-order distance estimator near the boundary.

The associated population based estimator $\bar{d}_\rr^1$ defined by
\begin{equation}\label{eq:D_bar}
\bar{d}_\rr^1(x^0) = \max_{x\in \overline{B(x^0,\rr)\cap\Omega}}(x^0-x)\cdot \bar{\nu}_{\rr}(x^0).
\end{equation}
Note that the population based estimator has a positive bias, meaning $d_{\Omega}(x^0)\leq \bar{d}_\rr(x^0)$. 
 In Lemma \ref{lem:bias} we obtain explicit bounds on the bias which establish that  $\bar{d}_\rr(x^0)- d_{\Omega}(x^0) = O(r^2)$ as $\rr\rightarrow 0$.
We combine this with variance bounds on $\bar \nu_r - \hat \nu_r$
established in Lemma \ref{lem:bernstein} to show, in Theorem \ref{thm:main} that  when  $r \gtrsim (\log n/n)^{1/(d+2)}$ we have $|\hat{d}^1_\rr(x^0) -  d_{\Omega}(x^0)| = O(r^2)$, with high probability, for $x^0$ sufficiently close to the boundary. The dependence of the error bounds on the parameters is explicitly stated.

\medskip

\textbf{Second-order estimator for the distance to the boundary of $\Omega$.}
If the boundary of $\Omega$ is $C^3$, and thus $d_\Omega$ is $C^3$ within the a sufficiently small tubular neighborhood of the boundary \cite{foote1984regularity}, then we can use the second-order estimator $\hat \nu^n_r$ of the unit  normal vector to obtain a second-order accurate estimator for the distance. 

To derive a second-order distance function estimation near the boundary, we proceed from \eqref{eq:rep}, as before, except now we use the higher order Taylor expansion 
\begin{equation}\label{eq:first_exp}
d_\Omega(y) = d_\Omega(x) + \nabla d_\Omega(x)\cdot(y-x) + \frac{1}{2}(y-x)\cdot \nabla^2 d_\Omega(x)(y-x)+ O(\rr^3).
\end{equation}
To handle the second-order terms, which cannot be easily estimated from the point cloud, we use the Taylor expansion
\[\nabla d_\Omega(y)  = \nabla d_\Omega(x) + \nabla^2 d_\Omega(x) (y-x) +O(\rr^2).\]
Taking dot products of both sides with $y-x$ yields
\[(y-x)\cdot \nabla^2 d_\Omega(x) (y-x) = (\nabla d_\Omega(y)- \nabla d_\Omega(x))\cdot(y-x) + O(\rr^3).\]
Combining this with the first expansion \eqref{eq:first_exp}    yields
\[d_\Omega(y) = d_\Omega(x) + \frac{1}{2}(\nabla d_\Omega(x) + \nabla d_\Omega(y))\cdot(y-x) + O(\rr^3).\]
Inserting this into \eqref{eq:rep} and using that $\nabla d_\Omega(x) = \nu(x)$ we obtain
\begin{equation}\label{eq:test2}
d_\Omega(x) = \max_{y\in B(x,\rr)\cap \Omega}\left\{ (x-y)\cdot \frac{1}{2}(\nu(x)+ \nu(y)) \right\} + O(\rr^3).
\end{equation}
Hence, the second-order distance estimator simply involves averaging the normals at $x$ and $y$. When discretizing to the point cloud, this yields the distance function estimation
\begin{equation} \label{eq:secd-pop}
\max_{x^i\in B(x^0,\rr)\cap X_n}(x^0-x^i)\cdot\frac12 (\hat{\nu}_\rr(x^0)+\hat{\nu}_\rr(x^i))
\end{equation}
The above test is second-order accurate when applied to points that are closer to boundary than $\tfrac r2$, however at far away points, in particular those further than $r$, 
$\hat{\nu}^{2}_\rr(x^0)$ and  $\hat{\nu}^{2}_\rr(x^i)$ are to large extent random and can be almost opposite to each other. This can lead to the distance being severely underestimated by the test above. 

 To avoid this problem, we define the second-order estimator with cutoff
\begin{equation}\label{eq:Deps2}
\boxed{\hat{d}_\rr^2(x^0) = \max_{x^i\in B(x^0,\rr)\cap \X}(x^0-x^i)\cdot\left[\hat{\nu}^{2}_\rr(x^0)+\frac{\hat{\nu}^{2}_\rr(x^i)-\hat{\nu}^{2}_\rr(x^0)}{2}\one_{\R_+}(\hat{\nu}^{2}_\rr(x^i)\cdot\hat{\nu}^{2}_\rr(x^0))\right]. }
\end{equation}
The rationale for the particular cutoff function is as follows. We need a highly accurate estimate of the distance, for example, to determine the points in a boundary strip,  only when $d_\Omega(x^0) <  \frac12 r \ll R$. The point where the right-hand side of \eqref{eq:secd-pop} is maximized is on the boundary. Thus the point where \eqref{eq:Deps2} is maximized, provided the normals are accurate, are close to the boundary. 
Points far away from the boundary can only maximize the right hand side if there is cancellation between the normal vector estimates.  So we just need to discard the points where the normal is very poorly estimated, or rather, where the normal estimation is irrelevant as $B(x^0,r)\cap \partial\Omega = \varnothing$.  Selecting the points where $\hat{\nu}_\rr(x^i)\cdot\hat{\nu}_\rr(x^0)>0$ provides a convenient way to do so. We note that instead of discarding such points, we simply resort back to the first-order test, which provides another layer of robustness, in the case that the assumptions under which the second-order test was derived do not hold. 
 
Henceforth, by the second-order estimator we refer to the estimator with cutoff \eqref{eq:Deps2}, unless stated otherwise. In practice, we recommend the use of the second-order estimator.  The estimates of Section \ref{sec:asym2} imply that for 
 $ r \gtrsim (\log n/n)^{1/(d+4)}$, the test \eqref{eq:Deps2}
 provides a   second-order estimator of the normal vector. 
We note that unlike for the first-order test, our analysis for the second-order test is in the asymptotic regime, without precise estimates in the non-asymptotic regime.  
Developing the full error analysis of the second-order estimators remains a future task. 
 \medskip

\subsubsection{Extension to manifolds} We can generalize both the first and the second-order distance estimators to the case where $\rho$ is supported on an $m$-dimensional manifold $\M$ with $m<d$. We simply replace the normal vectors by their projection onto the relevant tangent spaces approximated using PCA locally. Using such projections in boundary estimation for manifolds has been exploited in \cite{AAL21}.
Let us denote by $\hat T^j$ the $m$-dimensional subspace spanned by the largest $m$ eigenvectors of the sample covariance matrix from the observations $x^i-x^j$ for $x^i\in B(x^j,\rr)$, and $\Pi^j$ the projection onto such a subspace. Thus we may define the first-order distance estimator in the manifold case as
\begin{equation}\label{def:Deps1_manifold}
\boxed{\hat{d}_\rr^1(x^0) = \max_{x^i\in B(x^0,\rr)\cap \X}\Pi^0((x^0-x^i))\cdot \hat{\nu}_{\rr}(x^0),}
\end{equation}
and the corresponding second-order estimator as
\begin{equation}\label{def:Deps2_manifold}
    \boxed{\hat{d}_{\rr,\M}^2 (x^0) = \max_{x^i\in B(x^0,\rr)\cap \X}\left(\Pi^0(x^0-x^i)\right)\cdot\left[\hat{\nu}^{2}_\rr(x^0)+\frac{\hat{\nu}^{2}_\rr(x^i)-\hat{\nu}^{2}_\rr(x^0)}{2}\one_{\R_+}(\Pi^0(\hat{\nu}^{2}_\rr(x^i))\cdot\Pi^0(\hat{\nu}^{2}_\rr(x^0)))\right]. }
\end{equation}
Note we have the equivalent distance estimators when we replace every vector $w$ that appear in the above definitions with $\Pi^0 w$, which we avoid to keep notation simple. When $\M$ itself has positive reach, $\Pi^j$ approximates the projection onto the true tangent plane at $x^j$ with an error of $O(\rr)$ in the operator norm with high probability; when $\M$ is a $C^3$ manifold, the error is of order $O(\rr^2)$ (see Theorem 2 of \cite{AL19}). In fact, this is also true in the presence of small additive noise. Further, Aamri and Levrard \cite{AL19} suggest the same order of accuracy in the presence of small additive, possibly non-random noise of order $O(\rr^2)$. This means that the error rates for the estimated normal vector carry  over, hence we can expect similar bounds on the distance estimators. Figure \ref{fig:manifold_illustration} shows experiments for 2 dimensional surfaces. However, the analysis required in this case is more intricate. One would need to bound the additional errors due to curvature and empirical estimation of the tangent plane. Thus we do not include the analysis in the current paper, and instead leave it to future work.

 \medskip

 \subsection{The new boundary test}\label{ssec:bdrytest}
 Now we are ready to present our boundary test.
 Our aim is to create a test such that given $\eps>0$ small the test would 
 recognize as boundary points all of the points within the distance $\eps$ from the true boundary of $\Omega$ and none of the points which are further than $2\eps$ from $\partial \Omega$. 
 
 The boundary test we introduce depends on the empirical estimator of the distance to the boundary. 
\medskip 
 
 \textbf{Boundary region test.}
Let 
$\X=\{x^1,x^2,\cdots,x^n\}\subset\Omega$ be an \emph{i.i.d.}~random sample of the density $\rho$. Let $\eps,\rr>0$ and $x^{0} \in \X$. 
Given an empirical estimator of the distance to the boundary $\hat d_r$ we define the test $\widehat{T}_{\eps,\rr}:\X\rightarrow\{0,1\}$ by
\begin{equation}  \label{def:test}
\boxed{
\widehat{T}_{\eps,\rr}(x^{0})=
		\begin{cases}
			1 \;\;\;\text{  if  } \hat{d}_\rr(x^0) < \frac{3\eps}{2}  \\
			0 \;\;\; \text{  otherwise}.
 		\end{cases}}
\end{equation}
We denote by $\widehat{T}^1_{\eps,\rr}$ the estimator that uses the first-order estimator for the distance $\hat{d}^1_\rr(x^0)$ defined in \eqref{eq:D_hat} and by $\widehat{T}^2_{\eps,\rr}$
the estimator that uses the second-order estimator for the distance $\hat{d}^2_\rr(x^0)$ defined in \eqref{eq:Deps2}. 
\medskip

Our theoretical guarantees focus on $\widehat{T}^1_{\eps,\rr}$. In particular we show that $\widehat{T}^1_{\eps,\rr}$ identifies the \emph{$\eps$-boundary points} with high probability, even when $\eps$ is much smaller than the typical distance between nearby points.  In particular Theorem \ref{thm:main} shows that, for $\eps \gtrsim(\log n/n)^{2/(d+2)}$, under appropriate assumptions, 
\begin{equation}\label{eq:Intro_prob_decay}
 	\falsepos+\falseneg\leq (2d+1) n^{-3}.
\end{equation}


The assumptions we make on the geometric parameters are as follows.

\begin{assumption}\label{ass:eps/r}
 $\frac{\eps}{\rr}\leq\frac{1}{3\sqrt{d}}$.
\end{assumption}

\begin{assumption}\label{ass:r/R}
$\rr^2\leq R\eps$.
\end{assumption}

Assumption \ref{ass:eps/r} assures that $\rr$ is sufficiently large
so that distances to boundary of size $\eps$ can be detected. In particular it ensures that there are points  $x \in B(x^0,\rr)$ for which $\left(x-x^{0}\right)\cdot\hat{\nu}_{\rr}(x^0)<-\frac{3\eps}{2}$. Assumptions \ref{ass:eps/r} and \ref{ass:r/R} together imply
\begin{equation}\label{eq:assumption_consequence1}
\left(\frac{\eps}{\rr}-\frac{\rr}{R}\right)^{2}\leq\frac{1}{d+1},
\end{equation}
which bounds the rate of growth of constant $C$ in Lemma \ref{lem:v0} in $d$. Assumption \ref{ass:r/R} is needed in  Lemma \ref{lem:hatd1_lowbd}
to ensure that  $\hat{d}^1_\rr(x^0)$  does not  underestimate the distance for 
positively curved domains. 
Assumptions \ref{ass:eps/r} and \ref{ass:r/R} imply
\begin{equation}\label{eq:assumption_consequence2}
\rr\leq R\frac{\eps}{\rr}\leq\frac{R}{3\sqrt{d}}.
\end{equation}
This guarantees that at least one third of $B(x^0,\rr)$ is in $\Omega$, which is crucial for establishing the lower bound in Lemma \ref{lem:p_bounds}. Finally, $\rr\leq\frac{R}{2}$ follows easily from the assumptions. This implies the estimate 
\[R-\sqrt{R^2-x^2}\leq\frac{x^2}{R} \quad \text{ for } \quad |x|\leq\rr,\]
which is used in the proof of Lemmas \ref{lem:p_bounds} and \ref{lem:v0}.

  Now we summarize our result on the accuracy of the boundary test. Corollary \ref{corol:BorelCantelli} states that under suitable conditions $\partial_{\eps,\rr}\X=\{x\in \X: \widehat T^1_{\eps,\rr}(x)=1\}$ satisfies
	   \begin{align*}
        \partial_{\eps}\Omega\subset
        \partial_{\eps,\rr}\X\subset
        \partial_{2\eps}\Omega
	   \end{align*}
        with probability at least $1-2dn^{-3}$, 
        if
        \begin{equation}\label{eq:intro_deltaRate}
        	       \eps\geq C\left(\frac{\log n}{n}\right)^{\tfrac{2}{d+2}}
        	   \end{equation}
        for some constant $C=C\left(d,R,L,\rho_{\min},\rho_{\max}\right)$. 
        For our second-order boundary test, our analysis in the asymptotic regime suggest that we can identify $\eps$-boundary points with $\eps\gtrsim (\log n/n)^{3/(d+4)}$ with high probability. Please see Sections \ref{sec:asym2} and \ref{sec:second-asymp} for precise statements. 

        We can compare the above result with that from Cuevas and Rodr\'{\i}guez-Casal \cite{CueRod04}, which gives the best available theroetical guarantee the authors are aware of. Theorem 4 of \cite{CueRod04} states that with probability one, the estimated set of boundary points $\partial\Omega_n$ based on the Devroye-Wise estimator \cite{DevWis80} satisfies
	   \begin{equation}\label{eq:intro_Ineqcuevas}
	       d_{H}\left(\partial\Omega_{n},\partial\Omega\right)
	       \leq(2s^{-1}\wod^{-1})^{\tfrac{1}{d}}\left(\frac{\log n}{n}\right)^{\tfrac{1}{d}}  
	       \;\;\;\text{eventually}.
	   \end{equation}
	   Here, $s$ denotes the \emph{standardness constant}, which in our case is at least $\frac{1}{3}$. Further, Theorem 5 of \cite{CueRod04} states that the rate in $n$ in \eqref{eq:intro_Ineqcuevas} is optimal for the Devroye-Wise estimator. Let us temporarily denote the right hand side of \eqref{eq:intro_Ineqcuevas}  by $\eps_n$. Note that this allows identifying all points within $\eps_n$ of the boundary and none farther than $2\eps$ via taking the points within $\eps_n$ of $\partial\Omega_{n}$.
	   
	   Note that our test satisfies, under suitable choices of $\eps,\rr$,
	   \[d_H(\partial_{\eps,\rr}\X,\partial\Omega)\leq 2\eps=O\left(\frac{\log n}{n}\right)^{\tfrac{2}{d+2}} \text{ with probability at least } 1-2dn^{-3},\]
       provided we choose $\epsilon$ at the lower bound in \eqref{eq:intro_deltaRate}. 
	   Thus for $d\geq 3$ our rate in $n$ compares favorably to the optimal rate of the Devroye-Wise estimator \eqref{eq:intro_Ineqcuevas}. However, the constant in \eqref{eq:intro_deltaRate} is of order $C\sim O(d^{5/2})$, while the constant $(2s^{-1}\wod^{-1})^{1/d}$ in  \eqref{eq:intro_Ineqcuevas} is of order $O(d^{1/2})$. Details on the dependence of the constants on $d$ can be found in Remark \ref{rmk:main}.
	   
	   Another notable difference is that identifying the boundary points through \cite{CueRod04} does not seem computationally tractable in higher dimensions. The points corresponding $x^i$ whose balls $B(x^i,\rr)$ contribute to the boundary correspond exactly to points on the boundary of the $\alpha$-shape \cite{EdelKirkSeidel83} of $\X$ . However, computing this involves Delaunay triangulation and may be difficult in dimensions higher than $3$. See Section \ref{sec:related-works} for more details.
	   
	   In contrast, our proposed boundary test is easy to implement and computationally efficient, as can be seen in Algorithms \ref{alg:1st} and \ref{alg:2nd}. The range search task of identifying $B(x^0,r)\cap \X$ for each $x^0\in \X$ is the computational bottleneck of our test. This is computationally equivalent to performing a $k$-nearest neighbor search for each point in $\X$ (all-kNN) for suitable $k$. Empirically, k-nearest neighbor search (kNN) can be done in almost linear time with high accuracy \cite{DongMosLi11,Github:annoy}. For further details, we refer the reader to the discussions in Section \ref{sec:algorithm_experiments}.
	   
	   Finally, our test does not require the knowledge of the intrinsic dimension of $\supp\rho$. For instance, if $\Omega$ is an $m$-dimensional disc, the proposed boundary test will perform exactly the same when $\Omega$ is embedded in $\R^d$ for any $d\geq m$, besides the slightly higher computational cost of performing range search or kNN in higher dimensions. This is because our test is based on estimation of the distance $d_\Omega$, which is intrinsic.

\subsection{Related works} \label{sec:related-works}



One of most studied  approaches to boundary and support estimation is via the Devroye-Wise estimator, which approximates the support of $\rho$ by a union of balls: 
\begin{equation}\label{def:DevWise}
    \Omega_n:=\bigcup_{i=1}^{n}B\left(x^i,\rr_{n}\right).
\end{equation}
 Devroye and Wise \cite{DevWis80} establish the convergence of $\Omega_n$ to $\Omega:=\supp\rho$ as $n\to \infty$ and  $r_n\rightarrow 0$, at a suitable rate, in the following sense: $\rho(\Omega \Delta \Omega_{n})\rightarrow 0$ in probability if $\rr_n\gg n^{-1/d}$, while $\rr_n \gg (\log n/n)^{1/d}$ implies almost sure convergence.
 
Cuevas and Rodriguez-Casal,
 \cite{CueRod04}, established that, under certain smoothness assumptions, the Hausdorff distances $d_{H}\left(\Omega_n,\Omega\right),d_{H}\left(\partial{\Omega}_{n},\partial \Omega\right)\sim (\log n/n)^{1/d}$, and that the rate is \emph{optimal}. Furthermore, it is possible to compute the points $x^i$ contributing to the boundary $\partial\Omega_n$ using $\alpha$-shapes, introduced in \cite{EdelKirkSeidel83}. However, $\alpha$-shapes are a union of a certain subset of simplicies of the Delaunay triangulation. This poses challenges as the Delaunay triangulation in $d>3$ dimensions is itself not an easy computational problem, as the number of simplices can be large, up to $O(n^{\lceil d/2 \rceil}$) \cite{mcmullen_1970}. Thus, while efficient $O(n^2)$ algorithms are established for $d\leq 3$ \cite{EdelsMucke94}, less is known for higher dimensions.
 
 We also note that the Devroye-Wise boundary estimators have been used to estimate the Minkowski content of the boundary of $S$, which for sufficiently regular sets approximates the surface area ($\left(d-1\right)$-dimensional Hausdorff measure). This  is shown to be $L_2$-consistent for general dimensions in \cite{CuFrGy13} and convergent at $O(n^{-1/(2d)})$ for $d=2,3$ in \cite{CuFrRo07}.
 
 Casal \cite{Casal07} defines an estimator called $r$-convex hull, based on the Minkowski sum and differences of sets and closely related to $\alpha$-shapes, to approximate the support $\Omega$ with improved rate of $(\log n/n)^{2/(d+1)}$ in the Hausdorff distance with high probability. 
 
 We note that the while the works of Devroye-Wise and Casal propose different estimators for the boundary of the set, the data points $x^i$ which are identified as being near the boundary are the same for both estimators, see Section \ref{sec:comparison} for explanation and  Figure \ref{fig:2nd-comparison} for illustration. 
 
Another family of approaches are associated with the kernel density estimators (KDE).
 Estimating the density level set via the kernel density estimator is well-studied \cite{ChGeWa17} \cite{QiaPol19}. Cuevas and Fraiman \cite{cuevas1997plug} approximate the support by the super-level sets $\{\hat f>\alpha_n\}$ of the KDE $\hat f$, where 
 tuning parameter $\alpha_n \to 0$ as $n \to \infty$, and establish $d_H$ almost at the aforementioned optimal rate.

On the other hand, Berry and Sauer \cite{BerrySauer17} approximates the distance $d_\Omega$ of points to the boundary of the manifold to improve accuracy of KDE near the boundary. To do so, they use the graph Laplacian to estimate the normal vectors, and compute $d_\Omega$ by solving an expression it satisfies in relation to the expectation of the said graph Laplacian.
 
For self-similar but possibly non-smooth $\partial\Omega$, such as the von Koch snowflake, 
 Lachi\`eze-Rey and Vega \cite{LacVeg17} use Voronoi cells to define an estimator that converges to $\Omega$ at the optimal rate in $d_H$ when $\rho$ is uniform.
 
Several further works, \cite{AAL21,AarCho20,WuWu19,BORDER06,BRIM07}, have focused on identifying the boundary when $\rho$ is supported on a lower dimensional manifold $\M$. Aamari, Aaron, and Levrard \cite{AAL21} generalize the result of Casal \cite{Casal07} to the manifold setting. They project the relevant geometric quantities onto the approximate tangent space estimated using principal component analysis (PCA) to identify the set $\mathcal Y\subset\X$ of points such that with high probability, for all $y^i\in\mathcal Y$ we have $d_H(y^i,\partial\M)\lesssim (\log n/n)^{2/(d+1)}$. Based on $\mathcal Y$, they use the weighted Tangential Delaunay Complex to provide an estimator approximating $\partial\M$ with rate $(\log n/n)^{2/(d+1)}$ in the Hausdorff distance with high probability. Further, they establish that this rate is minimax over the class of convex submanifolds (i.e. those diffeomorphic to a convex subset of $\R^d$), thus showing not only that their upper bound is tight, but also that estimation of boundary under the assumption of positive reach is not more difficult than that in the convex case.

Our first-order test identifies the set of boundary points such that with high probability each point is at most $(\log n /n)^{2/(d+2)}$. While our theoretical results are established for flat domains, we believe the same rate would apply to the generalized first-order estimator \eqref{def:Deps1_manifold} in the manifold case. Through the same boundary reconstruction process as stated in \cite{AAL21}, we may construct boundary estimators with the same rate, which is slightly slower than the minimax rate proven by \cite{AAL21}. However, we note that our test identifies w.h.p. \emph{all points} within such tubular neighborhood of the boundary, which is stronger than obtaining the same bound in the Hausdorff distance, and is important for application to PDEs on graphs.

It is also interesting to note that the asymptotic error rate for our second-order test \eqref{def:test} based on distance estimator 
 \eqref{eq:Deps2}
in the Euclidean case is $(\log n/n)^{3/(d+4)}$, see Sections \ref{sec:asym2} and  \ref{sec:second-asymp}. This estimator however requires that manifolds are of class $C^3$ and that $\rho$ is $C^2$, while the rates in  \cite{AAL21} hold for manifolds which are merely $C^2$ and bounded densities. Determining minimax rates for estimators for $C^3$, and more regular manifolds and densities, remains an open problem.



  Aaron and Cholaquidis \cite{AarCho20} devise a statistical test to determine whether a random sample supported on a manifold has a boundary, along with heuristics to identify some of the points closer to the boundary. While their test uses k-nearest neighbor search instead of range search, the suggested test statistic for each point $x^0$ is similar to the size of the projection of $\hat{v}_{\rr}(x^0)$ onto the approximate tangent space at $x^0$.
Thus, loosely speaking, this statistic exploits that the normal vector is of order $O(\rr)$ near the boundary, while $O(\rr^2)$ in the interior. We note that this approaches only use the size of the estimated normal, while we utilize the normal vector itself.
 
 Wu and Wu\cite{WuWu19} use the behavior of the locally-linear embedding (LLE) near the boundary to identify boundary points. Interestingly, their test statistic is a quadratic function of a kNN-analogue of our normal vector $\hat v_\rr$, where the coefficients take into account the curvature of $\partial\Omega$ and density fluctuations. Further, they provide theoretical guarantees for their test statistic (see Proposition 5.1 of \cite{WuWu19}).

A couple other methods try to use the normal vectors, but approximated in a different way. BORDER algorithm \cite{BORDER06} uses that, given a fixed $k\in\mathbb N$ and sufficiently many points, the number of points of which $x^0$ is a $k$-neighbor of will be roughly half when $x^0$ is near the boundary, compared to that when $x^0$ is in the interior. BRIM algorithm introduced in \cite{BRIM07}, exploits the fact that given a suitable approximation of the inward normal at $x^0$, say $\nu(x^0)$, the number of points $x^i$ such that $(x^i-x^0)\cdot \nu(x^0)$ is positive is greater than the number of points  for which the inner product is negative, when $x^0$ is near the boundary. BRIM approximates the inward normal by identifying the point $y\in B(x^0,\rr)\cap\mathcal \X$ such that $|B(y,\rr)\cap\mathcal \X|$ is largest, then using $y-x^0$ as the estimator. However, for both approaches, such difference is of the same order as the statistic, which is weaker than the dichotomy used in \cite{WuWu19}. Moreover, none of the approaches above use the normal vector to measure the distance to the boundary, which is one of the key elements for the improved accuracy. 

Our convergence proofs for the solutions of PDEs on point clouds in Section \ref{sec:PDE_graph} utilize the maximum principle, building upon previous related works in the field \cite{calder2018game,calder2019lip,garcia2020maximum,yuan2020continuum,flores2019algorithms}. We also expect that recent advances in the studies of PDEs on point clouds \cite{calder2020Lip, calder2019improved, GGHS2020} can also be applied in this setting, to obtain, for example, spectral convergence for the Dirichlet graph Laplacian.  There are many methods in the numerical analysis literature for solving PDEs on unstructured meshes or point clouds.  Methods with rigorous convergence results include the wide stencil schemes for Hamilton-Jacobi equations and elliptic PDEs \cite{oberman2008wide}, which were originally defined on regular grids and have subsequently been extended to unstructured point clouds \cite{froese2018meshfree,finlay2019improved}, and the point integral method \cite{li2017point}. Other works without convergence guarantees include upwind schemes for Hamilton-Jacobi equations on unstructured meshes \cite{sethian2000fast}, mesh-free generalized finite difference methods \cite{suchde2019fully,suchde2019meshfree}, least squares manifold approximation methods \cite{liang2013solving,wang2018modified,trask2020compatible}, the local mesh method \cite{lai2013local}, radial basis function methods \cite{flyer2009radial,fuselier2012scattered,piret2012orthogonal,piret2016fast}, and a recent approach using graph Laplacians and deep learning \cite{liang2021solving}. A general survey of meshfree methods in PDEs is given in \cite{chen2017meshfree}.

Regarding data depth, the ordering of multivariate data is an old problem in statistics \cite{barnett1976ordering,liu1999multivariate}. The goal is generally to extend robust statistical notions, like quantiles and the median, to multivariate data.  For point clouds, there are notions of depth like the Tukey halfspace depth \cite{tukey1975mathematics}, which has been extended to graphs \cite{small1997multidimensional} and metric spaces \cite{carrizosa1996characterization}, and the Monge-Kantorovich depth \cite{chernozhukov2017monge}. There are also notions of depth for curves \cite{de2020depth} It was recently shown in \cite{molina2021tukey} that the Tukey depth satisfies a non-standard eikonal equation in the viscosity sense, at the population level. To the best of our knowledge, the eikonal equation on a graph has not been used for data depth previously. Two forthcoming papers will study the graph eikonal depth in more detail \cite{molina2021eikonal,calder2021hamilton}. Other examples of connections between data depth and PDEs include convex hull peeling \cite{calder2020convex}, non-dominated sorting \cite{calder2014}, and Pareto envelope peeling \cite{bou2021hamilton}.

	   \medskip
	   
	   \textbf{Outline.} 
	   The remainder of this paper is organized as follows. In Section \ref{sec:prelim} we establish preliminary estimates and error estimates on normal vectors estimators that will be useful in proving the main results, which are presented in Sections \ref{sec:main} and \ref{sec:second-asymp}. Section \ref{sec:main} rigorously establishes nonasymptotic error bounds for the first-order test, which is the theoretical basis for applications to PDEs on graphs presented later in the paper. Section \ref{sec:second-asymp}, under some additional regularity assumptions, establishes asymptotic error bounds for the second-order test, which we recommend for practical use. Then we present the algorithm and discuss the computational aspects of the boundary test in Section \ref{sec:algorithm_experiments}. Turning to applications, in Section \ref{sec:PDE_graph} we will apply the boundary test to solving PDEs on graphs with various boundary conditions. Particular attention is paid to computing data-depth using PDEs in two ways: by solving the graph eikonal equation, and considering the first eigenfunction of the graph Laplacian. We also demonstrate these to MNIST and FashionMNIST data sets; see Section \ref{sec:realdata}.

	   
	   

\section{Preliminary results and error bounds for normal vector estimators}\label{sec:prelim}

    In this section we establish several results on the geometry of the empirical estimates we use, most importantly the error bounds for the normal vector estimators. Nonasymptotic $O(\rr)$ error bound for the first-order normal vector estimator is given in Theorem \ref{thm:normal_vector}, and Section \ref{sec:asym2} establishes asymptotic $O(\rr^2)$ error bound for the second-order normal vector estimator. All the constants introduced in this and the following sections can also be found in Appendix \ref{appendixConstants}, and are non-dimensional. That is, they are invariant under the change of length-scale.

	First we derive useful bounds on $\int_{B(x^0,\rr)}\rho(x)\,dx$ from the assumptions. We note that the following lemma is closely related to the `standardness constant' in \cite{CueRod04}, which denotes the constant $s>0$ in 
    such that for all $x^{0} \in \Omega$
	\begin{equation}\label{StandardnessIneq}
	    \frac{|B(x^0,\rr) \cap \Omega|}{|B\left(x^{0},\rr\right)|}\geq s.
	\end{equation}
	This constant is of importance as it gives a lower bound on the number of points in $B(x^0,\rr)\cap\Omega$ with high probability. Our first lemma asserts that the Assumptions \ref{ass:eps/r}, \ref{ass:r/R} imply that $s\geq\frac{1}{3}$. 

    \begin{lemma}\label{lem:p_bounds}
    Let $\rr>0$. Then
     \begin{equation}\label{eq:p_bounds}
        \frac{\rho_{\min}\wod\rr^{d}}{3}\leq \int_{B(x^0,\rr)}\rho(x)\,dx\leq \rho_{\max}\wod\rr^{d}
     \end{equation}
    \end{lemma}

    \begin{proof} 
	As the upper bound is obvious, we focus on the lower bound, which easily follows from $s\geq 1/3$. We claim that \eqref{StandardnessIneq} holds for  $s = \frac{1}{2}\left(1-\frac{\sqrt{d}\rr}{R}\right)$.
	Note that $B(x^0,\rr)\cap\Omega$ at least consists of the hemisphere minus the area between the tangent hyperplane at $x^{0}$. As the assumption $\rr\leq\frac{R}{3}$ implies that the height of the region between the tangent hyperplane and $\Omega$ with reach $R$ is bounded above by $\frac{\rr^2}{R}$. Therefore, we may upper bound the area of the region by considering the cylinder with base $\left(d-1\right)$-dimensional hypersphere of radius $\rr$ and height $\frac{\rr^2}{R}$. Thus its area is $\wod[d-1]\rr^{d-1}\frac{\rr^2}{R}=\frac{\wod[d-1]\rr^{d+1}}{R}$. Therefore
	    \begin{equation}\label{eq:pbounds;s_ineq1}
	        s\geq \frac{1}{2}-\frac{\wod[d-1]\rr^{d+1}R^{-1}}{\wod[d]\rr^{d}}=\frac{\wod[d-1]}{\wod[d]}\frac{\rr}{R}
	    \end{equation}
	 We introduce the notation
	 \begin{equation}\label{eq:gammad}
	 \kappa_{d} = \frac{\omega_{d-1}}{\omega_d}
	 \end{equation}
    and claim that $\kappa_{d}\leq\sqrt{d}$. Note that since $\Gamma$ is a logarithmically convex function
        \[ {\Gamma\left(\frac{d}{2}+1\right)}^2 \leq {\Gamma\left(\frac{d-1}{2}+1\right)} \, {\Gamma\left(\frac{d+1}{2}+1\right)}. \]
    Therefore, $\wod[d]^2 \geq \wod[d-1] \wod[d+1]$, and $\kappa_{d+1} \geq \kappa_{d}$.
    On the other hand, 
    \[\kappa_{d} \kappa_{d+1} = \frac{\wod[d-1]}{\wod[d+1]} = \frac{\Gamma\left(\frac{d+1}{2}+1\right)}{\pi \Gamma\left(\frac{d-1}{2}+1\right)} = \frac{d+3}{2 \pi}.\]
    Combining with $\kappa_{d+1} \geq \kappa_{d}$, we get $\kappa_{d}\leq\frac{\sqrt{d+3}}{2\pi}\leq\sqrt{d}$ as $d+3\leq 4\pi d$. Similarly, we have a lower bound  $\kappa_{d+1} \geq \sqrt{ \frac{d+3}{2 \pi}} \geq \frac{1}{3} \sqrt{d+1}$, which will be of use later. Hence
    \begin{equation}\label{eq:sphereVol_ratio}
        \frac{\sqrt{d}}{3}\leq\frac{\wod[d-1]}{\wod}\leq \sqrt{d}.
    \end{equation}
    Combining the upper bound of \eqref{eq:sphereVol_ratio} with \eqref{eq:pbounds;s_ineq1}, we have $s\geq\frac{1}{2}\left(1-\frac{\sqrt{d}\rr}{R}\right)$. This, along with \eqref{eq:assumption_consequence2}, implies that $s\geq\frac{1}{3}$.
    \end{proof}

In the following two lemmas we examine the bias of the population-based estimators.

\begin{lemma}[Bias of the estimated normal]\label{lem:v0}
For every $x^0\in \Omega$ with $d_{\Omega}(x^0)\leq \rr/2$ 
we have
\begin{equation}\label{eq:v_expansion}
\left|\bar{v}_{\rr}(x^0)  -C_{y}(x^0)\rho(x^0)\rr^{d+1}\nu(x^0)\right|  \leq \frac{C_{x}\rho(x^0)}{R} \rr^{d+2},
\end{equation}
provided $|\tfrac{\alpha}{\rr}-\tfrac{\rr}{R}|\leq 1$, where 
\begin{equation}\label{eq:v_const}
\begin{split}
&C_{x} = 2\wod[d-1]+\frac{LR\wod}{\rho_{min}}    \\
&C_{y}(x^0) = \frac{\omega_{d-1}\left(1-\left(\tfrac{ d_{\Omega}(x^0)}{\rr}-\tfrac{\rr}{R}\right)^2\right)^{\frac{d+1}{2}}}{(d+1)}.
\end{split}
\end{equation}
In particular, whenever $d_\Omega(x^0)\leq 2/(3\sqrt{d})$, we have $C_y(x^0)\geq\frac{\omega_{d-1}}{2(d+1)}$.
\end{lemma}

\begin{remark}[Lower bound on $C_y$]\label{rmk:Cy_lowbd}
Suppose $d_{\Omega}(x^0)\leq 2\rr/(3\sqrt{d})$. Then $\left(\frac{d_\Omega(x^0)}{\rr}-\frac{\rr}{R}\right)^2\leq \frac{d_\Omega(x^0)^2}{\rr^2}\leq\frac{4}{9d}\leq\frac{1}{d+1}$
\begin{equation}\label{eq:simplified_const_half}
\left(1-\left(\tfrac{d_{\Omega}(x^0)}{\rr}-\tfrac{\rr}{R}\right)^2\right)^{\frac{d+1}{2}}\geq 1 -\frac{d+1}{2} \left(\tfrac{d_{\Omega}(x^0)}{\rr}-\tfrac{\rr}{R}\right)^2\geq \frac12
\end{equation}
and so
\begin{equation}\label{eq:simplified_const_C}
 C_{y}(x^0)\geq\frac{\wod[d-1]}{2(d+1)}.
\end{equation}
This lower bound will be important for results to follow. Observe that $d_\Omega(x^0)\lesssim \rr/\sqrt{d}$ allows a similar bound $C_y(x^0)\gtrsim \omega_{d-1}/d$.

Note also by Assumption \ref{ass:eps/r}, $d_\Omega(x^0)\leq 2\eps$ is a sufficient condition. As this is more intuitive and sufficient for theoretical results on the boundary test, we henceforth state the condition as $d_\Omega(x^0)\leq 2\eps$, but note here that all such conditions can be replaced by $d_\Omega(x^0)\leq 2\rr/(3\sqrt{d})$.
\end{remark}

\begin{proof}[Proof of Lemma \ref{lem:v0}]
We write
\begin{equation}\label{eq:Vbreakup}
\bar{v}_{\rr}(x^0)=E_1 + \rho(x^0)E_2,
\end{equation}
where
\begin{equation}\label{eq:A}
E_1=\int_{\Omega\cap B(x^0,\rr)}(x-x^0)(\rho(x)-\rho(x^0))\, dx,
\end{equation}
and
\begin{equation}\label{eq:B}
E_2 = \int_{\Omega\cap B(x^0,\rr)}(x-x^0)\, dx.
\end{equation}
Since $\rho$ is Lipschitz with constant $L$, the term $E_1$ is bounded by
\begin{equation}\label{eq:E1bound}
|E_1\:|\: \leq L\int_{B(x^0,\rr)}|x-x^0|^2\, dx= L\int_0^\rr\int_{\partial B(x^0,t)}t^2 \, dS \, dt\, dx= L\int_0^\rr d\omega_dt^{d+1}\, dt = \frac{Ld\omega_d}{d+2} \rr^{d+2}.
\end{equation}

We now estimate $E_2$. Without loss of generality, we may assume $x^0 = (0,0,\dots,0,\alpha)$ for $\alpha=\text{dist}(x^0,\partial\Omega)$. By the assumption that the reach of $\partial\Omega$ is greater than $R>0$, we have
\[\partial\Omega \cap B(x^0,\rr) \subset \left\{x\in B(x^0,\rr) \, : \, |x_d| \leq \frac{\rr^2}{R}\right\},\]
provided $\rr\leq R/2$.   Therefore
\begin{equation}\label{eq:E2bound}
\left| E_2 - \int_{B(x^0,\rr)\cap\{x_d\geq \frac{\rr^2}{R}\}}(x-x^0)\,dx\right| \leq \int_{B(x^0,\rr)\cap\{|x_d|\leq \frac{\rr^2}{R}\}}|x^0-x|\, dx \leq \frac{2\omega_{d-1}\rr^{d+2}}{R}.
\end{equation}
We now change variables $z=(x-x^0)/\rr$ and write
\begin{align*}
\int_{B(x^0,\rr)\cap\{x_d \geq \frac{\rr^2}{R}\}}(x_d-x^0_d)\, dx&=\rr^{d+1}\int_{B(0,1)\cap\{z_d \geq \frac{\rr}{R}-\frac{\alpha}{\rr}\}}z_d \, dz\\
&=\rr^{d+1}\int_{B(0,1)\cap\{z_d \geq \left|\frac{\alpha}{\rr}-\frac{\rr}{R}\right|\}}z_d \, dz,
\end{align*}
where the last inequality comes from symmetry of the integrand. We now compute for any $0\leq t \leq 1$
\begin{align*}
\int_{B(0,1)\cap\{z_d\geq t\}}z_d\, dz&= \omega_{d-1}\int_{t}^1 z_d(1-z_d^2)^{\frac{d-1}{2}}\, dz \\
&=\frac{\omega_{d-1}}{2}\int_{t^2}^1 (1-s)^{\frac{d-1}{2}}\, ds\\
&=\frac{\omega_{d-1}}{d+1}(1-t^2)^{\frac{d+1}{2}}.
\end{align*}
Due to symmetry of the integrand, we have
\[\int_{B(x^0,\rr)\cap\{x_d \geq \frac{\rr^2}{R}\}}(x_j-x^0_j)\, dx = 0\]
for all $j=1,\dots,d-1$. Combining this with \eqref{eq:E2bound} we find that
\begin{equation}\label{eq:E2bound2}
\left| E_2 - \frac{\omega_{d-1}}{d+1}\left(1-\left(\tfrac{\alpha}{\rr}-\tfrac{\rr}{R}\right)^2\right)^{\frac{d+1}{2}}\rr^{d+1}\nu(x^0)\right| \leq  \frac{2\omega_{d-1}\rr^{d+2}}{R},
\end{equation}
provided $|\tfrac{\alpha}{\rr}-\tfrac{\rr}{R}|\leq 1$, since $\nu(x^0)=e_d$. Thus
\[
    \left|\bar{\nu}_{\rr}(x^0)- \frac{\omega_{d-1}}{d+1}\left(1-\left(\tfrac{\alpha}{\rr}-\tfrac{\rr}{R}\right)^2\right)^{\frac{d+1}{2}}\rr^{d+1}\nu(x^0)\right|
    \leq \left(\frac{2\wod[d-1]}{R}\rho(x^0)+L\wod\right)\rr^{d+2}.
\]
We complete the proof by noting
\[
    \frac{2\wod[d-1]}{R}\rho(x^0)+L\wod
    = \frac{\rho(x^0)}{R}\left(2\wod[d-1]+\frac{LR}{\rho(x^0)}\right)\leq\frac{\rho(x^0)}{R}\left(2\wod[d-1]+\frac{LR}{\rho_{\min}}\right)=:\frac{C_{x}\rho(x^0)}{R}.
\]
\end{proof}

Based on the bias of the estimated normal, we can approximate the bias of the distance estimator.

 \begin{lemma}[Bias of the distance estimator]\label{lem:bias}
 Let $x^0\in\Omega$ with $ d_{\Omega}(x^0)\leq 2\eps$. If
    \begin{equation}\label{eq:bias;eps_constraint}
        \rr\leq\frac{R C_{y}}{2 C_{x}}
    \end{equation}
 then
    \begin{equation}\label{eq:bias_dist}
     d_{\Omega}(x^0)\leq\bar{d}^1_{\rr}(x^0)
    \leq  d_{\Omega}(x^0)+\left(\frac{7C_{x}}{RC_{y}}+\frac1R\right)\rr^2.
    \end{equation}
\end{lemma}
\begin{proof}
\begin{enumerate}
\item Recall
\[{\hat{v}_{\rr}(x^0)}=\frac{1}{n}\sum_{i=1}^n\one_{B(x^0,\rr)}(x^i)(x^i-x^0)\]
and
\[\E{\hat{v}_{\rr}(x^0)} = \int_{B(x^0,\rr)}(x-x^0)\rho(x)\,dx = \bar{v}_{\rr}(x^0).\]
We consider the population based statistic 
\[\bar{d}_\Omega^1(x^0) = \max_{x\in \Omega\cap B(x^0,\rr)}\{ (x^0-x)\cdot  \bar{v}_{\rr} \},\]
where $\bnu :=\frac{\bar{v}_{\rr}(x^0)}{\|\bar{v}_{\rr}(x^0)\|}$. 
\item By Lemma \ref{lem:v0} we have
\[\bar{v}_{\rr}(x^0) = C_{y}\rho(x^0)\rr^{d+1}\nu(x^0) + \frac{1}{R}\O\left( C_{x}\rho(x^0)\rr^{d+2} \right).\]
 Here, we can use the big-Oh notation very precisely, to mean that $f\in\O(g)$ if $|f|\leq g$ (without any implicit constant). Therefore
\[|\bar{v}_{\rr}(x^0)| = C_{y}\rho(x^0)\rr^{d+1} + \frac{1}{R}\O\left(C_{x}\rho(x^0)\rr^{d+2} \right).\]
We also have
\begin{equation}\label{eq:dot}
(x^0-x)\cdot \bar{v}_{\rr}(x^0) = C_{y}\rho(x^0)\rr^{d+1}(x^0-x)\cdot \nu(x^0) + \frac{1}{R}\O(C_{x}\rho(x^0)\rr^{d+3}).
\end{equation}
We now write
\begin{align*}
\frac{1}{|\bar{v}_{\rr}(x^0)|}&=\frac{1}{C_{y}\rho(x^0)\rr^{d+1} + \frac1R \O(C_{x}\rho(x^0)\rr^{d+2})}\\
&=\frac{1}{C_{y}\rho(x^0)\rr^{d+1}\left( 1 + \frac1R\O\left(\frac{C_{x}\rr}{C_{y}}\right) \right)}.
\end{align*}
We now use that
\[\frac{1}{1+t} = 1 + \O(4|t|) \ \ \text{for }|x|\leq \frac{1}{2}.\]
Hence, if 
\begin{equation}\label{eq:restriction1}
\rr \leq \frac{R C_{y}}{2C_{x}},
\end{equation}
which implies that $\frac{C_{x}\rr}{RC_{y}}\leq \frac{1}{2}$, then we have
\begin{equation}\label{eq:inv}
\frac{1}{|\bar{v}_{\rr}(x^0)|}=\frac{1}{C_{y}\rho(x^0)\rr^{d+1}}\left( 1 + \O\left(\frac{4 C_{x}\rr}{RC_{y}}  \right) \right).
\end{equation}
Recall from \eqref{eq:simplified_const_C} that $C_{y}>\frac{\wod[d-1]}{2(d+1)}$. Thus
\[
    \frac{C_{x}}{C_{y}}\leq 2(d+1)\left(2+\frac{LR \kappa_{d}}{\rho_{\min}}\right).
\]
\item Inserting \eqref{eq:inv} into \eqref{eq:dot} we have
\begin{align}
(x^0-x)\cdot \bar{\nu}_{\rr}(x^0) &= \frac{(x^0-x)\cdot \bar{v}_{\rr}(x^0)}{|\bar{v}_{\rr}(x^0)|} \notag \\
&= \left( (x^0-x)\cdot \nu(x^0) + \O\left( \frac{C_{x} \rr^2}{RC_{y}} \right) \right)\left( 1 + \O\left( \frac{4C_{x}\rr}{RC_{y}} \right) \right) \notag \\
&= (x^0-x)\cdot \nu(x^0) + \O\left( \frac{C_{x} \rr^2}{RC_{y}} + \frac{4C_{x}\rr^2}{RC_{y}}+ \frac{4C_{x}^2\rr^3}{R^2C_{y}^2}\right) \notag \\
&= (x^0-x)\cdot \nu(x^0) + \O\left( \frac{5C_{x}\rr^2}{RC_{y}}+ \frac{4C_{x}^2\rr^3}{R^2C_{y}^2}\right),
\end{align}
where $x\in \Omega\cap B(x^0,\rr)$.
\item To obtain the lower bound we simply observe that $\max_{|x^i-x^0|\leq\rr}(x^0 - x^i )\cdot v$ is smallest when $v=\nu(x^0)$, in which case $\max_{|x^i-x^0|\leq\rr}(x^0 - x^i )\cdot \nu(x^0)= d_{\Omega}(x^0)$. Thus
\begin{equation}\label{eq:lowerbound_dist}
     d_{\Omega}(x^0)\leq\bar{d}_\Omega^1(x^0)
\end{equation}

\item For the other direction, by the assumption that the reach of $\partial\Omega$ is greater than $R$, we have
\begin{equation}\label{eq:reach;1/R}
    \Omega\cap B(x^0,\rr) \subset \left\{x\in B(x^0,\rr) \, : \, (x^0-x)\cdot \nu(x^0) \leq  d_{\Omega}(x^0) + \frac{\rr^2}{R}\right\},
\end{equation}
provided $\rr\leq R/2$. 
It follows that
\begin{equation}\label{eq:upper_bound}
\bar{d}_\Omega^1(x^0) \leq  d_{\Omega}(x^0) +\left(\frac{5C_{x}}{RC_{y}}+\frac1R\right)\rr^2+ \frac{4C_{x}^2\rr^3}{R^2C_{y}^2}.
\end{equation}

\item Now combining \eqref{eq:lowerbound_dist} and \eqref{eq:upper_bound}
\end{enumerate} we have 
\[
     d_{\Omega}(x^0)\leq\bar{d}_\Omega^1(x^0)
    \leq \left(\frac{5C_{x}}{RC_{y}}+\frac{1}{R}\right)\rr^2 + \frac{4C_{x}^2}{R^2C_{y}^2}\rr^3\leq\left(\frac{7C_{x}}{RC_{y}}+\frac{1}{R}\right)\rr^2
\] as desired, where the last inequality follows from the condition $\rr\leq\frac{RC_y}{2C_x}$. Finally, as $\kappa_{d}\sim\sqrt{d}$ by \eqref{eq:sphereVol_ratio}, $\frac{C_{x}}{C_{y}}\sim d^{\frac32}$.
\end{proof}

Next, we bound the variance of $\hat{\nu}_\rr$, the empirical estimator of the normal vector.

\begin{lemma}[Bound on the variance]\label{lem:bernstein}
    Let $\gamma>0$
    and $c\leq\frac{6d^3C_{x}\rho_{\max}\wod}{RC_{y}}$. If $d_\Omega(x^0)\leq 2\eps$ and $\rr$ satisfies
    \begin{equation}\label{eq:lem_bernstein_eps}
      \left(\frac{3\gamma \rho_{\max}d^2\wod}{c^2}\frac{\log n}{n}\right)^{\frac{1}{d+2}}
      \leq\rr \leq \frac{RC_{y}}{2C_{x}}
    \end{equation}
    then
    \begin{equation}\label{eq:lem;bd_var}
        \mathbb{P}\left(|\hat{\nu}_{\rr}(x^0)-\bar{\nu}_{\rr}(x^0)|> \frac{6c\rr}{C_y \rho(x^0)}\right)
        \leq 2dn^{-\gamma}
    \end{equation}
\end{lemma}
\begin{proof}
    Let us first fix $x^0\in \X$. For each $j=1,2,\cdots,d$ let 
 		\[S_{n}^{j}=
 		\sum_{i=1}^n \one_{B(x^0,\rr)}(x^i)(x_j^i - x_j^0).\]
    Note
    \[\sigma^2 = \Var\left(\one_{B(x^0,\rr)}(x^i)(x_j^i - x_j^0)\right) \leq \int_{B(x^0,\rr)}|x^i-x^0|^2 \rho(x) \, dx \leq \rho_{\max}\wod\rr^{d+2}.\]
    By Bernstein's Inequality \eqref{thm:bernstein}, we have
    \begin{align*}
        \mathbb{P}\left( \left|\frac1n S_n-\bar{v}_{\rr}(x^0) \right|>c\rr^{d+2}\right)
        &\leq\sum_{j=1}^{d}\mathbb{P}\left( \left|\frac1n S_{n}^{j}-\bar{v}_{\rr}(x^0)_j \right|>\frac{c\rr^{d+2}}{d}\right)  \\
        &\leq\sum_{j=1}^{d}  
        2\exp\left[-\frac{-nc^2\rr^{2d+4}}{2d^2\rho_{\max}\wod\rr^{d+2}+\frac{c}{3d}\rr^{d+3}}\right]   \\
        &\leq 2\sum_{j=1}^{d}
        \exp\left[-\frac{nc^2\rr^{d+2}}{2d^2\rho_{\max}\wod+\frac{cRC_{y}}{6d C_{x}}}\right]
        \leq 2d\exp\left[-\frac{nc^2\rr^{d+2}}{3d^2\rho_{\max}\wod}\right]
    \end{align*}
    where the second last inequality follows from \eqref{eq:bias;eps_constraint}, and the last inequality from the condition \[c\leq\frac{6d^3C_{x}\rho_{\max}\wod}{RC_{y}}.\] 
    The exponent is smaller than $-\gamma\log n$ when
    \[
      \rr\geq\left(\frac{3\gamma \rho_{\max}d^2\wod}{c^2}\frac{\log n}{n}\right)^{\frac{1}{d+2}}
    \]
    which is \eqref{eq:lem_bernstein_eps}. Thus
    \begin{equation}\label{eq:lem_bernstein_pf}
        \mathbb{P}\left(|\hat{v}_{\rr}(x^0)-\bar{v}_{\rr}(x^0)|>c\rr^{d+2}\right)\leq 2dn^{-\gamma}
    \end{equation}
    Now, note that
    \[
        |\hat{\nu}_{\rr}(x^0)-\bar{\nu}_{\rr}(x^0)|=\left|\frac{\hat{v}_{\rr}(x^0)}{|\hat{v}_{\rr}(x^0)|}-\frac{\bar{v}_{\rr}(x^0)}{|\bar{v}_{\rr}(x^0)|}\right|
        \leq\left|\hat{v}_{\rr}(x^0)\left(\frac{1}{|\hat{v}_{\rr}(x^0)|}-\frac{1}{|\bar{v}_{\rr}(x^0)|}\right)\right|+\frac{|\hat{v}_{\rr}(x^0)-\bar{v}_{\rr}(x^0)|}{|\bar{v}_{\rr}(x^0)|}.
    \]
    Then \eqref{eq:lem_bernstein_pf} implies
    \begin{align*}
      \left|\hat{v}_{\rr}(x^0)\left(\frac{1}{|\hat{v}_{\rr}(x^0)|}-\frac{1}{|\bar{v}_{\rr}(x^0)|}\right)\right|
      =\frac{1}{|\bar{v}_{\rr}(x^0)|}|\hat{\nu}_{\rr}(x^0)(|\bar{v}_{\rr}(x^0)|-|\hat{v}_{\rr}(x^0)|)|
      \\
      \leq\frac{1}{|\bar{v}_{\rr}(x^0)|}|\bar{v}_{\rr}(x^0)-\hat{v}_{\rr}(x^0)|
      \leq\frac{c\rr^{d+2}}{|\bar{v}_{\rr}(x^0)|}
    \end{align*}
    and
    \[
        \frac{1}{|\bar{v}_{\rr}(x^0)|}|\hat{v}_{\rr}(x^0)-\bar{v}_{\rr}(x^0)|
        \leq\frac{c\rr^{d+2}}{|\bar{v}_{\rr}(x^0)|}.
    \]
    Therefore, we have
    \begin{equation}\label{eq:lem_bernstein_prob}
        \mathbb{P}\left(|\hat{\nu}_{\rr}(x^0)-\bar{\nu}_{\rr}(x^0)|> \frac{2c\rr^{d+2}}{|\bar{v}_{\rr}(x^0)|}\right)
        \leq 2dn^{-\gamma}.
    \end{equation}
    Finally, from \eqref{eq:inv} and the condition $\rr\leq \frac{R C_y}{2C_x}$ we can deduce \eqref{eq:lem;bd_var} as
    \[\frac{2c\rr^{d+2}}{|\bar{v}_{\rr}(x^0)|}\leq \frac{2c\rr^{d+2}}{C_y\rho(x^0)\rr^{d+1}}\left(1+\O\left(\frac{4C_x\rr}{RC_y}\right)\right)\leq \frac{6c\rr}{C_y \rho(x^0)}.\]
\end{proof}

	    
	
	\begin{theorem} (Error estimates for the estimated normal vector)\label{thm:normal_vector}    \\
		Let $x^0\in \X$ with $d_\Omega(x^0)\leq 2\eps$. Let $\gamma>2$ and $\eps,\rr>0$ satisfy Assumption \ref{ass:r/R}. Let $\rr$ and $n$ satisfy
		\begin{equation}\label{eq:r_cond_nvec}
		    \left(\frac{3\gamma\rho_{max} d^2\wod R^2}{C_{x}^2\rho_{\min}^2}\frac{\log n}{n}\right)^{\frac{1}{d+2}}\leq r\leq \frac{RC_y}{2 C_x}.
		\end{equation}
	  Then
	\begin{equation}\label{eq:normalvec;main}
	    \mathbb{P}\left(|\hat{\nu}_{\rr}(x^0)-\nu(x^0)|\geq \frac{13C_{x}}{RC_{y}}\rr\right)\leq 2dn^{-\gamma}
	\end{equation}
	
	\end{theorem}
	
	 \begin{remark}\label{rmk:normal_vector}
	    Observe that if $\rr$ satisfies \eqref{eq:r_cond_nvec}, then we may choose $\rr=\left(\frac{3\gamma\rho_{max} d^2\wod R^2}{C_{x}^2\rho_{\min}^2}\frac{\log n}{n}\right)^{\frac{1}{d+2}}$, which means
	    \[\mathbb{P}\left(|\hat{\nu}_{\rr}(x^0)-\nu(x^0)|\geq C\left(\frac{\log n}{n}\right)^{\frac{1}{d+2}}\right)\leq 2dn^{-\gamma}\]
	    with
	    \[
	        C:=\frac{C_{x}}{C_{y}}\left(\frac{3\gamma\rho_{max} d^2\wod R^2}{C_{x}^2\rho_{\min}^2}\right)\sim d^{2},
	    \]
	    where the asymptotics in $d$ can be derived using Stirling's formula. For a more detailed analysis of the how the constants scale with dimension, please see Remarks \ref{rmk:main} and \ref{rmk:boundary_test}.
	    
	    Further, we note that the above result holds for $x^0\in \Omega_{2\eps}$ -- i.e. the reference point need not be one of the samples. The same applies to following results on the distance estimator.
	\end{remark}
	
	\begin{proof}
	The upper bound of \eqref{eq:new_main;condition} allows us to apply Lemma \ref{lem:bias}, which we will combine with Lemma \ref{lem:bernstein}. The lower bound in \eqref{eq:sphereVol_ratio} implies that
\[
    \frac{6d^3\wod}{C_{y}} \geq \frac{12d^3(d+1)}{\kappa_{d}} \geq 12
\] from which easily follows $\frac{C_{x}\rho(x^0)}{R}\leq\frac{6d^3\rho_{\max}\wod C_{x}}{RC_{y}}$. Thus we may set $c=\frac{C_{x}\rho(x^0)}{R}$. Then Lemma \ref{lem:bernstein} implies that if
\[
    \rr\geq\left(\frac{3\gamma\rho_{max} d^2\wod R^2}{C_{x}^2\rho_{\min}^2}\frac{\log n}{n}\right)^{\frac{1}{d+2}}
\]
then, by \eqref{eq:inv},
\begin{equation}\label{eq:new_main;normal_upperbd}
    |\hat{\nu}_{\rr}(x^0)-\bar{\nu}_{\rr}(x^0)|\leq \frac{2C_{x}}{RC_{y}}\left(1+\O\left(\frac{4C_{x}\rr}{RC_{y}}\right)\right)\rr\leq\frac{6C_{x}}{RC_{y}}\rr
\end{equation}
with probability at least $1-2dn^{-\gamma}$, where the last inequality follows from the condition $\rr\leq\frac{RC_y}{2C_x}$.

	
	Next we bound $|\bar{\nu}_{\rr}(x^0)-\nu(x^0)|$. Again by \eqref{eq:inv}
	\begin{align*}
	    |\bar{\nu}_{\rr}(x^0)-\nu(x^0)|&=\left|\frac{\bar{v}_{\rr}(x^0)}{|\bar{v}_{\rr}(x^0)|}-\nu(x^0)\right|
	    \\
	    &=\frac{1}{C_{y}\rho(x^0)\rr^{d+1}}\left|\bar{v}_{\rr}(x^0)\left(1+\O\left(\frac{4C_{x}\rr}{RC_{y}}\right)\right)-C_{y}\rho(x^0)\rr^{d+1}\nu(x^0)\right|.
	 \end{align*}
	By Lemma \ref{lem:v0}
	\[
	    \left|\bar{v}_{\rr}(x^0)\left(1+\O\left(\frac{4C_{x}\rr}{RC_{y}}\right)\right)-C_{y}\rho(x^0)\rr^{d+1}\nu(x^0)\right|
	    \leq|\bar{v}_{\rr}(x^0)-C_{y}\rho(x^0)\rr^{d+1}\nu(x^0)|+\frac{4C_{x}|\bar{v}_{\rr}(x^0)|\rr}{RC_{y}}.\]
	Thus
	\begin{equation}\label{eq:normalvec;population}
	    |\bar{\nu}_{\rr}(x^0)-\nu(x^0)|
	    \leq \frac{C_{x}}{RC_{y}}\rr
	    +\frac{4C_{x}\rr}{RC_{y}}+\frac{4C_{x}^2\rr^2}{R^2C_{y}^2}\leq \frac{7C_{x}}{RC_{y}}\rr
	\end{equation}
	where the last inequality follows from \eqref{eq:restriction1}. Combining \eqref{eq:new_main;normal_upperbd} and \eqref{eq:normalvec;population} we have
	\[
	    |\hat{\nu}_{\rr}(x^0)-\nu(x^0)|\leq\frac{13C_{x}}{RC_{y}}\rr
	\]
	with probability at least $1-2dn^{-\gamma}$. 
	\end{proof}

\subsection{Second-order estimators: asymptotic error scaling} \label{sec:asym2}
	 
Here we analyze the asymptotic error of the ``second-order'' estimator of the normal vector, $\hat{\nu}^{2}_r(x^0)$, defined in \eqref{eq:hatvn}, and show that the error is indeed second-order in $r$, for points $x^0$ sufficiently close to the boundary, namely $d_\Omega(x^0)\lesssim \rr/\sqrt{d}$, which allows us to use \eqref{eq:v_expansion} with a reasonable lower bound on $C_y(x^0)$ (see Remark \ref{rmk:Cy_lowbd}). We note that in this section, in order to simplify expressions we use radius $\rr$ for estimating $\theta$, instead of the radius $\rr/2$ as in \eqref{def:theta;sample} and \eqref{def:theta;population}. However, a similar argument works when we set the radius to be $\rr/2$.



For simplicity, we first assume the boundary is the graph of a quadratic function near $x^0$. That is that near $x^0 = |x^0| e_d$ and  the boundary is given by 
\[ x_d = H(x)^T A H(x) \]
where $A$ is a $(d-1) \times (d-1)$ symmetric matrix and 
\[ H(x) = (x_1, \dots, x_{d-1})^T \] 
We also introduce the symbols for projection of a vector to the $e_d$ direction and for central symmetry with respect to the first $d-1$ variables
\[ N(x) = x \cdot e_d \, e_d \quad \te{and} \quad S(x) = (-H(x), x_d). \]
Furthermore let $U(x) = B(x,r) \cap \Omega$. 

Since $\bar v^2_r(x^0) \cdot e_d > Cr^{d+1}$ by estimate \eqref{eq:v_expansion} it suffices to show that $|H(\bar v^2_r(x^0))|  \leq C r^{d+3}$. 
We start by noting that due to symmetry of the quadratic function near $x^0$
\begin{align*}
H(\bar v^2_r(x^0))  & = \frac12
\int_{U(x^0)}  H  \left( \frac{ \rho(x) }{\theta(x)}(x - x^0)  +  \frac{ \rho(S(x)) }{\theta(S(x))}(S(x) - x^0)   \right) dx \\
& \leq  \frac12 \int_{U(x^0)}
\frac{ |\rho(x) \theta(S(x)) - \rho(S(x)) \theta(x) |}{ \theta(S(x)) \theta(x)}  |H(x)| dx  \\
& \leq \frac{8}{\rho_{min}^2} r^{d+1} \sup_{x \in U(x^0)} |\rho(x) \theta(S(x)) - \rho(S(x)) \theta(x) |. 
\end{align*}
 For $x \in U(x^0)$ we now estimate, assuming $4 r < R$ and using
 that $S$ is isometry between  $U(x)$ and $U(S(x))$
\begin{align*}
| \rho(x) \theta(S(x)) &  - \rho(S(x)) \theta(x) |  = \frac{1}{\omega_d r^d} 
 \left| \rho (x)\int_{U(S(x))} \rho(z) - \rho(S(x)) dz - \rho(S(x)) \int_{U(x)} \rho(z) - \rho(x) dz \right| \\
&  \leq  \frac{1}{\omega_d r^d} \left| \rho (x)\int_{U(S(x))} \nabla \rho(N(0)) \cdot (z - S(x))dz -    \rho(S(x)) \int_{U(x)} \nabla \rho(N(0)) (z-x) dz  \right| \\
     & \phantom{\leq}  + 4 \|\rho\|_{L^\infty} \|D^2 \rho\|_{L^\infty} r^{2} \\
&=  \frac{1}{\omega_d r^d}  \left| (\rho(S(x)) - \rho(x))  \int_{U(x)} \nabla \rho(N(0)) (z-x) dz  \right| + 4\|\rho\|_{L^\infty} \|D^2 \rho\|_{L^\infty} r^{2} \\
& \leq 4 \left( \| \nabla \rho \|_{\L^\infty}^2 + \|\rho\|_{L^\infty} \|D^2 \rho\|_{L^\infty} \right) r^{2}
\end{align*} 

Combining with the estimate above we obtain
\[ |H(\bar v^2_r(x^0))|  \leq C r^{d+3} \]
where $C$ depends on $\rho$ alone. 

We now relax the assumption that the boundary of $\Omega$ is a graph of a quadratic function. Namely note that since the boundary of $\Omega$ is $C^3$ there exists $C^r>0$ such that near $x^0$ the boundary of $\Omega$ is between the graphs of $x_d = H(x)^T A H(x) - C_r |H(x)|^3$ and 
$x_d = H(x)^T A H(x) + C_r |H(x)|^3$. Note that neglecting the part of $\Omega$ between the graphs produces an error of size $r^{d+3}$ and that all of the estimates above carry over to the part of $\Omega$ where $x_d > H(x)^T A H(x) + C_r |H(x)|^3$. Thus it still holds that $|T(\bar v^n_r(x^0))| \leq C r^{d+3}$, only that $C$ depends both of $\rho$ and $\Omega$.

\medskip
We now outline the argument at the level of the sample. One can use standard concentration inequalities to control the variance and  obtain the regime in which the empirical estimator $\hat v^n_r$ is within $Cr^3$ of the population based estimate $\bar v^n_r$.

Applying  Bernstein's inequality to the random variables $Y^j = \frac{1}{\omega_d r^d }  \one_{|x^j-x|\leq r/2}$ one obtains 
\begin{equation}\label{eq:thera-var}
|\hat \theta(x^i) - \theta(x^i)| \lesssim r^2 
\end{equation}
with high probability provided that $r \gtrsim (\log n/n)^{1/(d+4)}$.
Using the union bound the estimate holds uniformly for all $i$.
Thus 
\[ \left| \hat{v}^{2}_r(x^0) -  \frac{1}{n}\sum_{i=1}^n \frac{\one_{B(x^0,\rr)}(x^i)}{ \theta(x^i)}(x^i - x^0) \right| \lesssim r^{d+3} \]
Using the Bernstein inequality once more one obtains that 
\[ \left|   \frac{1}{n}\sum_{i=1}^n \frac{\one_{B(x^0,\rr)}(x^i)}{ \theta(x^i)}(x^i - x^0)  - \bar v^2_r(x^0) \right| \lesssim r^{d+3} \]
with high probability if $r \gtrsim (\log n/n)^{1/(d+4)}$

Combining with $\bar v^2_r(x^0) \cdot e_d \gtrsim r^{d+1}$ and 
$|H(\bar v^2_r(x^0))| \lesssim r^{d+3}$ we conclude that $|\hat \nu^2_r(x^0) - \nu(x^0) | \lesssim r^2$, as desired.

\section{Nonasymptotic error bounds for first-order distance and boundary estimators}\label{sec:main}

 In this section we establish the main results. Namely in Theorem \ref{thm:main} we show that the estimator $\hat{d}^1_\rr(x^0)$ has $O(\rr^2)$ error, provided that $r \gtrsim (\log n/n)^{1/(d+2)}$. 
 We then use this estimate to show that when $\frac{r^2}{R} \lesssim \eps \lesssim r$
 then we can accurately identify the $\eps$-boundary points.
    
  We start with establishing a lower bound on error of the distance estimator $\hat{d}^1_\rr$.
  
  	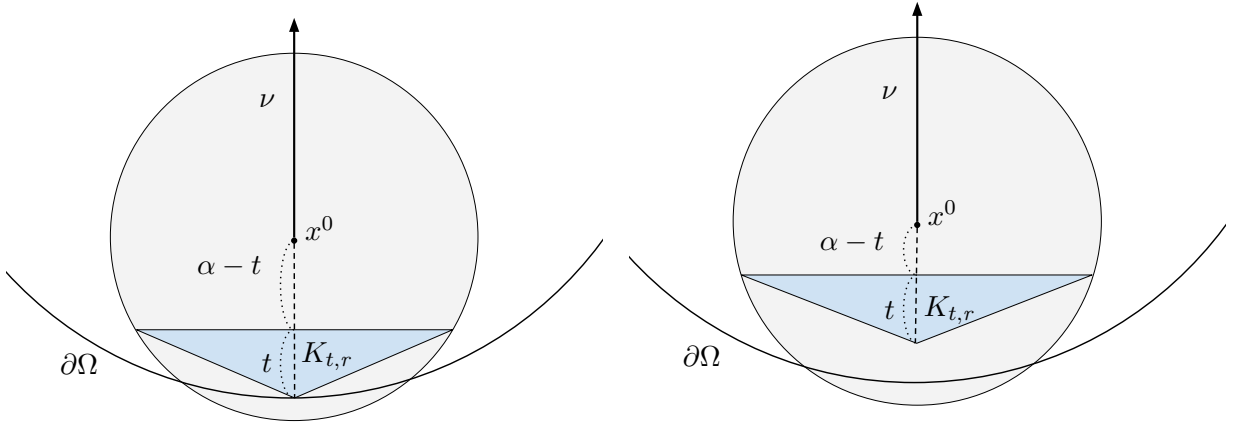
\begin{figure}[htb]
	\begin{tikzpicture}
	    \node at (0,0){\includegraphics[width=0.48\textwidth]{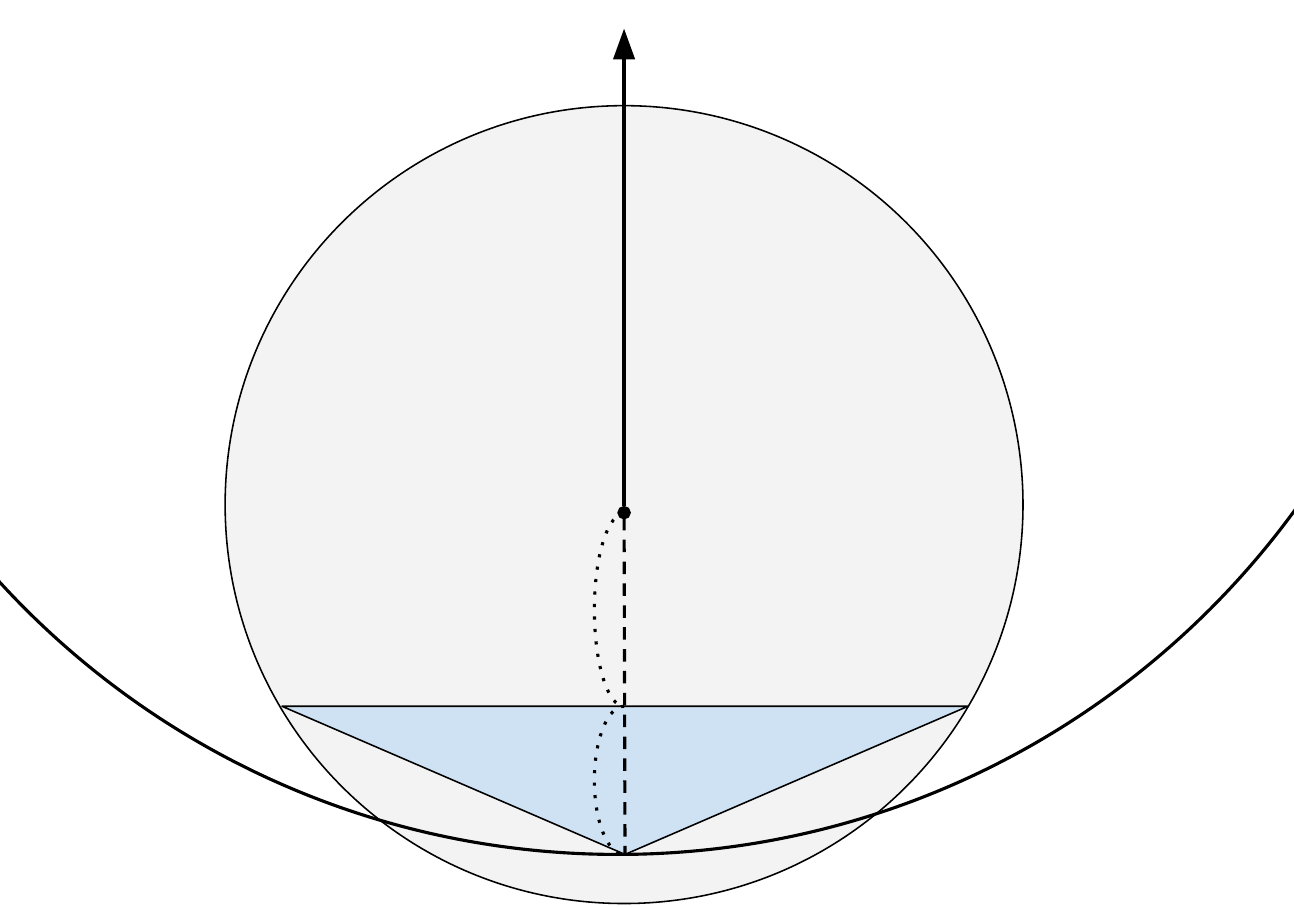}};
	    \node at (0.2,-0.2){$x^0$};
	    \node at (-1,-0.7){$\alpha-t$};
	    \node at (-0.5,-2){$t$};
	    \node at (-0.5,1.5){$\nu$};
	    \node at (0.3,-1.9){$K_{t,\rr}$};
	    \node at (-3,-2){$\partial\Omega$};
	\end{tikzpicture}
	\begin{tikzpicture}
	    \node at (0,0){\includegraphics[width=0.48\textwidth]{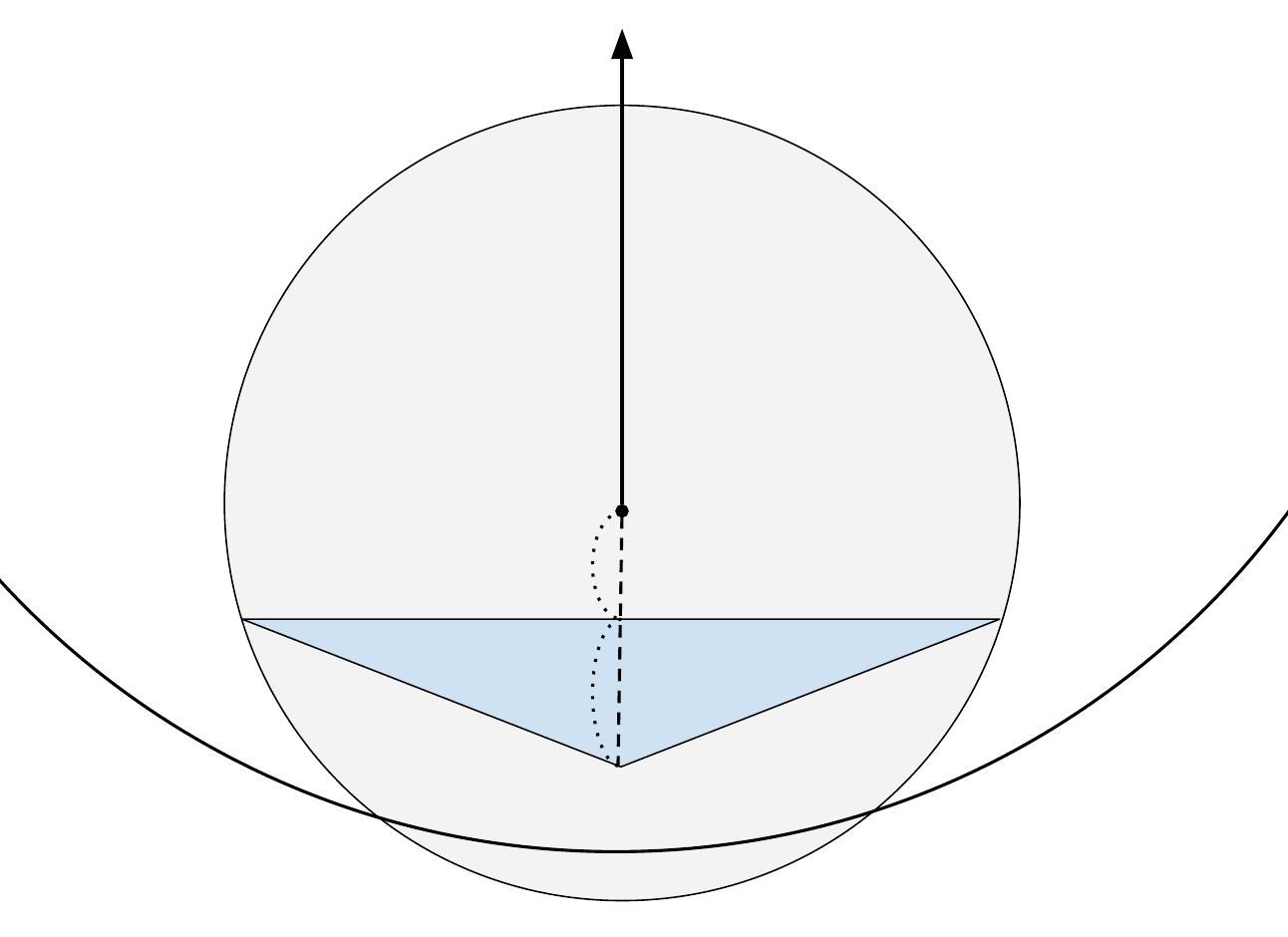}};
	    \node at (0.2,-0.1){$x^0$};
	    \node at (-1,-0.5){$\alpha-t$};
	    \node at (-0.5,-1.4){$t$};
	    \node at (-0.5,1.5){$\nu$};
	    \node at (0.3,-1.4){$K_{t,\rr}$};
	    \node at (-3,-2){$\partial\Omega$};
	\end{tikzpicture}
	\caption{Geometry relevant to the lower bound on $\hat{d}^1_\rr(x^0)$. $\alpha=d_\Omega(x^0)\wedge \tfrac\rr2$. (Left) Case where $d_\Omega(x^0)<\tfrac\rr2$; (Right) case where $d_\Omega(x^0)>\tfrac\rr2$.}\label{fig:false_positive}
\end{figure}

\begin{lemma}[Lower bound on the distance estimator]\label{lem:hatd1_lowbd}
 Let $\gamma>2$, $0<t\leq d_{\Omega}(x^0)$, and suppose Assumption \ref{ass:eps/r} holds. If $n$ and $\lambda>0$ satisfy
 \[n\geq d\vee (1+4\lambda^{-1})\]
 and $t,\rr$ satisfy
    \begin{equation}\label{eq:tr_lowbd}
        t\rr^{d-1}\geq\frac{\gamma d^2 2^{(d-1)/2}}{\rho_{\min}\omega_{d-1}}\left(\frac{\log n}{n}\right),
    \end{equation}
    then
    \begin{equation}\label{eq:hatd1_lowbd}
        \hat d^1_r(x^0)\geq (1-\lambda) (d_\Omega(x^0)\wedge \frac{\rr}{2}) - t
    \end{equation}
    with probability at least $1-n^{-\gamma}$.
\end{lemma}

\begin{remark}\label{rmk:lowbd}
   In fact, the lemma holds for any unit vector $\hat u$ that may depend on $\X$. Recall that the second-order distance estimator $\hat d_\rr^2$ defined in \eqref{eq:Deps2} is of the form
   \[\hat d_\rr^2(x^0)=\max_{x^\in B(x^0,\rr)\cap\X}(x^0-x^i)\cdot \hat u,\]
   where $|\hat u|$ can be as small as $\frac{1}{\sqrt{2}}$ in the interior, when $\hat u$ is an average of orthogonal unit vectors. Thus a slight modification allows us to obtain a similar result to the second-order distance estimator $\hat d_\rr^2$.
\end{remark}

\begin{proofsketch}
    As the proof involves lengthy elementary calculations, we delay the full proof to Appendix \ref{appendix:proof_lowbd}, and only present the main ideas here. The idea is to ensure that for any unit vector $u\in\S^{d-1}$, possibly depending on the samples $\X$, there is a point in the spherical segment $S^u \cap\Omega$ that contains points at least $(1-\lambda) (d_\Omega(x^0)\wedge \frac\rr2) - t$ away in the opposite direction of $u$. See Figure \ref{fig:false_positive} for the illustration in the case $u=\nu$. As there are infinitely many choices of $u$, we shrink the spherical segment slightly so that we have a finite family $\{\tilde S^1,\cdots,\tilde S^N\}$ such that for any $u\in\S^{d-1}$ we can find $\tilde S^i\subset S^u$. This means it suffices to show that each $\tilde S^i$ is nonempty for $i=1,\cdots,N$, and
    \begin{align*}
        \P(\hat d^1_\rr(x^0) \leq (1-\lambda) (d_\Omega(x^0)\wedge \frac{\rr}{2} - t) &\leq \P(S^u\cap\Omega \text{ is nonempty for all } u\in\S^{d-1})\\
        &\leq \sum_{i=1}^N \P(\tilde S^i\cap\Omega \text{ is nonempty for all } i=1,\cdots,N).
    \end{align*}
    For suitably chosen spherical segments, we may observe that $\tilde S^i\cap\Omega$ contains a cone $K$ with the same base and height as the spherical segment. Thus the proof comes down to obtaining a lower bound for the volume of this cone, and an upper bound on the number $N$.
\end{proofsketch}

We now state the nonasymptotic error bounds on the first-order distance estimator. 
    
    \begin{theorem}[Error bounds for the distance estimator]\label{thm:main}
    Let $\eps,\rr>0$ satisfy Assumptions \ref{ass:eps/r} and \ref{ass:r/R}. Let constants $C_x$ and $C_y$ be as in \eqref{eq:v_const}, and
    \begin{align*}
    C_{\rr}&=\frac{1}{R}\max\left[\left(\frac{3\gamma\rho_{\max}d^2\wod R^2}{{C_{x}}^2 \rho_{\min}^2 }\right)^{\frac{1}{d+2}},\left(\frac{4\gamma C_yd^2 2^{(d-1)/2}}{13\rho_{\min}\wod[d-1]C_{x}}\right)^{\frac{1}{d+1}}\right]   \\
    \end{align*}
    Suppose $\gamma>2$, and $n,\rr$ satisfy 
    \begin{equation}\label{eq:new_main;n_lowbd}
        n\geq d\vee \left(1+\frac{RC_y}{13C_x}\rr^{-1}\right)
    \end{equation}
    and
    \begin{equation}\label{eq:new_main;condition}
    \begin{split}
        &RC_{\rr}\left(\frac{\log n}{n}\right)^{\frac{1}{d+2}}\leq\rr
        \leq \frac{RC_{y}}{2C_{x}}  \\
    \end{split}
    \end{equation}
    Then, for $x^0\in \X$ we have
    \begin{equation}\label{eq:new_main;lowbd}
       d_{\Omega}(x^0)\wedge \frac\rr2-\frac{13C_{x}}{RC_{y}}\rr^2\leq \hat{d}^1_\rr(x^0)
    \end{equation}
    with probability at least $1-n^{\gamma}$. Moreover, if $d_{\Omega}(x^0)\leq 2\eps \leq \rr$, 
        \begin{equation}\label{eq:new_main;conclusion}
             \hat{d}^1_\rr(x^0) \leq d_{\Omega}(x^0)
            +\left(\frac{13 C_{x}}{RC_{y}}+\frac1R\right)\rr^2
        \end{equation}
    with probability at least $1-2dn^{-\gamma}$.
    \end{theorem}
    
\begin{remark}\label{rmk:main}
    We make two brief remarks. Firstly, \eqref{eq:new_main;n_lowbd} is a much weaker condition than the lower bound of \eqref{eq:new_main;condition}, as $\rr\geq RC_r\left(\frac{\log n}{n}\right)^{\frac{1}{d+2}}$ implies
    \[\frac{RC_y}{13C_x}r^{-1}\leq \frac{RC_y}{13C_x}(RC_r)^{-\frac{1}{d+2}} \left(\frac{n}{\log n} \right)^{\frac{1}{d+2}},\]
    which is much smaller than $n$ for reasonably large $n$.

     Secondly, we note that $C_{\rr}\sim \omega_d^{-1/d}$. Using  Stirling's formula $d!\sim \sqrt{2\pi d}(d/e)^{d}$ one obtains $\omega_d \sim (1/\sqrt{\pi d}) (2\pi e/d)^{d/2}$. Therefore $C_{\rr}\sim \omega_d^{-1/d}=O(\sqrt{d})$
\end{remark}
\begin{proof} 
We first prove the upper bound \eqref{eq:new_main;conclusion}. Suppose $d_\Omega(x^0)\leq 2\eps\leq \rr$. Condition \eqref{eq:new_main;condition} allows us to apply Theorem \ref{thm:normal_vector} to obtain \eqref{eq:new_main;normal_upperbd} --i.e. 
\[|\hat \nu_r(x^0)-\nu(x^0)|\leq \frac{13 C_x}{RC_y}\rr\]
with probability at least $1-2dn^{-\gamma}$. Thus
\begin{align*}
    \hat{d}^1_\rr(x^0) &= \max_{x^i\in B(x_0,\rr)\cap\X} \left\{(x^0-x^i)\cdot (\hat\nu_\rr(x^0)-\nu(x^0)+\nu(x^0))\right\}  \\
    &\leq \max_{x^i\in B(x_0,\rr)\cap\X} (x^0-x^i)\cdot (\hat\nu_\rr(x^0)-\nu(x^0)) + \max_{x^i\in B(x_0,\rr)\cap\X} (x^0-x^i)\cdot \nu(x^0)\\
    &\leq \frac{13 C_x}{RC_y}\rr^2+d_\Omega(x^0)+\frac1R \rr^2
\end{align*}
with the same probability. The last inequality uses the bound on $|\hat\nu_r(x^0)-\nu(x^0)|$ and that positive reach condition implies \eqref{eq:reach;1/R}. Thus we have the upper bound \eqref{eq:new_main;conclusion}.

Next, suppose $x^0\in\X$, not necessarily close to the boundary. Letting $t=\frac{13C_{x}}{2C_{y}}\rr^2$ in Lemma \ref{lem:hatd1_lowbd}, if $\rr$ satisfies
\[  
    \rr^{d+1} \geq\frac{4\gamma C_{y}d^2 2^{(d-1)/2}}{13\rho_{\min}\wod[d-1]C_{x}}\frac{\log n}{n}
\]
then Lemma \ref{lem:hatd1_lowbd} implies that
\begin{equation}\label{eq:hatd1_lowbd;lambda}
    \hat{d}^1_\rr(x^0)\geq  (1-\lambda)(d_{\Omega}(x^0)\wedge \frac\rr2)-\frac{13C_{x}}{2C_{y}}\rr^2
\end{equation}
with probability at least $1-n^{-\gamma}$, given $n\geq d\vee (1+4\lambda^{-1})$. Further, choose $\lambda=\frac{13 C_x}{RC_y}\rr$, so that by Assumption \ref{ass:eps/r}
\[\lambda(d_\Omega(x^0)\wedge \frac\rr2)\leq \frac{\lambda \rr}{2} = \frac{13 C_x}{2R C_y}\rr^2.\]
Then \eqref{eq:new_main;n_lowbd} implies
\[\hat d_\rr^1(x^0)\geq d_\Omega(x^0)\wedge 2\eps - \lambda(d_\Omega(x^0)\wedge \frac\rr2) - t \geq d_\Omega(x^0) - \frac{13 C_x}{RC_y}\rr^2,\]
hence we obtain \eqref{eq:new_main;lowbd}.
\end{proof}

\begin{corollary}[Accuracy of the boundary test]\label{corol:boundary_test}
    Let $x^0\in \X$, $\gamma>2$ and $\eps,\rr>0$ satisfy Assumptions \ref{ass:eps/r} and \ref{ass:r/R}. Let $C_\rr$ be as in \eqref{thm:main}.
    If $n\geq d\vee 33$ and $\rr,n$ satisfy
    \begin{equation}\label{eq:corol:boundary_test;eps_cond}
        RC_r\left(\frac{\log n}{n}\right)^{\frac{1}{d+2}} \leq \rr \leq \frac{RC_{y}}{2 C_{x}}
    \end{equation}
    and $\eps$ satisfies
    \begin{equation}\label{eq:corol:boundary_test;delta_cond}
        \frac1R \left(\frac{26C_{x}}{C_{y}}+2\right)\rr^2
        <\eps
    \end{equation}
    then 
    \begin{equation}\label{eq:corol:boundary_test;Prob}
        \falseneg+\falsepos\leq (2d+1)n^{-\gamma}.
    \end{equation}
    In particular, choosing the optimal $\rr,\eps$
    \begin{equation}\label{eq:eps_opt}
        \eps=\frac1R \left(\frac{26C_{x}}{C_{y}}+2\right)\rr^2=RC_r^2 \left(\frac{26C_{x}}{C_{y}}+2\right) \left(\frac{\log n}{n}\right)^{\frac{2}{d+2}},
    \end{equation}
    the test identifies the $\eps$-boundary with probability at least $1-(2d+1)n^{-\gamma}$.
\end{corollary}

\begin{remark}\label{rmk:boundary_test}
    Recall that \eqref{eq:assumption_consequence1} implies
    \[
        \frac{C_{x}}{C_{y}}\leq 2(d+1)\left(1+  \frac{RL}{\rho_{\min}}\kappa_{d} \right)=O(d^{\frac32})
    \]
    as $\kappa_{d}\sim\sqrt{d}$ by \eqref{eq:sphereVol_ratio}. Also, recall from Remark \eqref{rmk:main} that $C_{\rr}=O(\sqrt{d})$. Therefore the constant for the optimal choice $\eps=C(\log n/n)^{2/(d+2)}$ in \eqref{eq:eps_opt} satisfies $C\sim C_r^2 C_x/C_y\sim d^{5/2}$.
\end{remark}

    \begin{proof}
    Suppose $n\geq d\vee \left(1+4\cdot 8\right)=d\vee 33$ and $d_\Omega(x^0)\geq 2\eps$. Then we may choose $\lambda=\frac{1}{8}$ in \eqref{eq:hatd1_lowbd;lambda} and apply Lemma \ref{lem:hatd1_lowbd} to deduce
    \[\hat d_\rr^1(x^0)\geq \frac78(d_\Omega(x^0)\wedge \frac{\rr}{2}) - \frac{13 C_x}{2R C_y}\rr^2 \geq
    \frac78(d_\Omega(x^0)\wedge 2\eps) - \frac{13 C_x}{2R C_y}\rr^2
    \geq \frac{7\eps}{4} - \frac{13 C_x}{2RC_y}\rr^2 > \frac{3\eps}{2}\]
    with probability at least $1-n^{-\gamma}$, where last inequality follows from the condition \eqref{eq:corol:boundary_test;delta_cond}. Note that we have used that Assumption \ref{ass:eps/r} implies $2\eps\leq \frac{3\sqrt{d}\eps}{2}\leq \frac{\rr}{2}$.Thus we deduce
    \[\P(\widehat T_{\eps,\rr}^1(x^0)=1\,|\,d_\Omega(x^0)\geq 2\eps)\leq n^{-\gamma}.\]
    
    On the other hand, if $d_\Omega(x^0)\leq \eps$, then the upper bound in \eqref{eq:new_main;conclusion} applies. Thus, again using \eqref{eq:corol:boundary_test;delta_cond}
    \[\hat d_\rr^1(x^0)\leq d_\Omega(x^0)+  \left(\frac{13 C_x}{R C_y}+\frac1R\right)\rr^2 \leq \frac{3\eps}{2},\]
    with probability at least $1-2dn^{-\gamma}$. Hence
    \[\P(\widehat T_{\eps,\rr}^1(x^0)=0\,|\,d_\Omega(x^0)\leq \eps)\leq n^{-\gamma}.\]
    Combining this with the bound for the probability of false positive occurring, we obtain \eqref{eq:corol:boundary_test;Prob}.
    \end{proof}
    
    For application to solving boundary value problems on graphs \cite{CST20}, it is crucial to limit the number of false positives, while the false negatives are not as detrimental. If we are only interested in bounding the probability of false positives, we may obtain the improved rate $\eps\geq C\left(\frac{\log n}{n}\right)^{\frac{1}{d+1}}$ with $C\sim d$.
    	     
     \begin{theorem}[One-sided accuracy of the boundary test]\label{thm:delta_lowbd}
    Let $\gamma>2$, and $x^0\in \X$. Suppose $\eps,\rr>0$ satisfy Assumptions \ref{ass:eps/r} and \ref{ass:r/R}. If $n\geq d\vee 33$ and $\eps, \rr$ satisfy
        \[
            \left(\frac{\gamma d^2 2^{(d-1)/2}}{\rho_{\min}\wod[d-1]}\frac{\log n}{n}\right)^{\frac{1}{d+1}}\leq \rr^2 < \frac{\eps}{4}
        \]
        then
        \[
            \mathbb{P}(\widehat{T}^1_{\eps,\rr}(x^0)=1\:|\: d_{\Omega}(x^0) > 2\eps)\leq n^{-\gamma}.
        \]
     \end{theorem}
     
     
     \begin{proof}
        Again, recall that Assumption \ref{ass:eps/r} implies $2\eps\leq \frac{\rr}{2}$. Applying Lemma \ref{lem:hatd1_lowbd} with $t=\rr^2$ and $\lambda=\frac18$, we have $\hat d_\rr^1(x^0)\geq \frac78(d_{\Omega}(x^0)\wedge \rr/2)-\rr^2$ with probability at least $n^{-\gamma}$. Thus if
        \[
            4\rr^2\leq\eps
        \]
        then, with probability at least $1-n^{-\gamma}$
        \[
            \hat{d}^1_\rr(x^0) \geq \frac{7}{8}(d_{\Omega}(x^0)\wedge \frac\rr2)-\rr^2 \geq \frac{7}{8}(d_{\Omega}(x^0)\wedge 2\eps)-\rr^2> \frac{7\eps}{4}-\frac{\eps}{4}\geq \frac{3\eps}{2}.
        \]
        This implies that $\widehat{T}^1_{\eps,\rr}(x^0)=0$ by  \eqref{def:test}.
     \end{proof}

    \begin{corollary}\label{corol:BorelCantelli}
        Let $x^0\in \X$, $\gamma>2$. Let $n\geq d\vee 33$ and be sufficiently large such that
        \begin{equation}\label{eq:eps_r;optimal}
              	\eps=RC_{\eps}\left(\frac{\log n}{n}\right)^{\frac{2}{d+2}},\,
    	        \rr=C_{\rr}\left(\frac{\log n}{n}\right)^{\frac{1}{d+2}}
        \end{equation}
        satisfy Assumptions \ref{ass:eps/r} and \ref{ass:r/R}. Recall the definitions
        \begin{align*}
            &\partial_a\Omega=\{x^{0}\in \X:  d_{\Omega}(x^0)\leq a\}   \\
            &\partial_{\eps,\rr}\X=\{x^{0}\in \X: \widehat{T}^1_{\eps,\rr}\left(x^{0}\right)=1 \}.
        \end{align*}
        Then, with probability at least $1-(2d+1)n^{1-\gamma}$.
	    
	    \begin{equation}\label{boundarySet}
	        \partial_{\eps}\Omega\subset
	        \partial_{\eps,\rr}\X\subset
	        \partial_{2\eps}\Omega.
	    \end{equation}
	   In particular, by the Borel-Cantelli lemma, the test identifies a set between $\partial_{\eps}\Omega$ and $\partial_{2\eps}\Omega$ eventually with probability 1.
    \end{corollary}

	\begin{proof}
	    By Corollary \ref{corol:boundary_test}, applying the test to all $n$ points we have \eqref{boundarySet} hold with probability at least $1-(2d+1)n^{-\gamma}\cdot n=1-(2d+1)n^{1-\gamma}$.
	\end{proof}

	\begin{remark}[Reconstruction of boundary from boundary points]\label{rmk:bdry_reconstruction}
	   Based on the set $\partial_{\eps\rr}\X$ of boundary points we can reconstruct the boundary strip that approximates $\partial\Omega$ in the Hausdorff distance. See for instance Theorem 3.11 of \cite{AAL21} and the comment preceding it on the reconstruction process using Delaunay Complex, and \cite{AL19} for further details.
	\end{remark}
	    
	
\section{Asymptotic error bounds for second-order distance and boundary estimators}\label{sec:second-asymp}
In this section, we use the $O(\rr^2)$ bound on the second-order normal estimator $\hat{\nu}_\rr^2$ from Section \ref{sec:asym2} to obtain $O(\rr^3)$ error bound on the second-order distance estimator $\hat d_\rr^2$ in the asymptotic regime, additionally assuming $\partial\Omega$ is of class $C^3$ and $\rho\in C^2_b(\Omega)$. Namely, we show that we can find some constant $C>0$ independent of $\rr$ such that
\begin{equation}\label{eq:hat_d_2;error}
    \begin{split}
        \hat d_\rr^2(x^0) &\geq d_\Omega(x^0)\wedge \frac{\rr}{2} - C\rr^3, \text{ and }    \\
        \hat d_\rr^2(x^0) &\leq d_\Omega(x^0)+C\rr^3 \text{ if } d_\Omega(x^0)\leq 2\eps
    \end{split}
\end{equation}
with high probability under the scaling $\rr\gtrsim (\log n/n)^{1/(d+4)}$. Note that the lower bound holds for general $x^0\in\X$, not just those close to the boundary. Given the estimates above, we may set $\eps=C\rr^3/2\sim (\log n/n)^{3/(d+4)}$ to see that our test \eqref{def:test} will identify the $\eps$-boundary points with high probability. For a detailed argument deducing accuracy of the boundary estimator from that of the distance estimator, please see the the proof of Corollary \ref{corol:boundary_test}; while the corollary applies to the first-order estimator, the same argument carries over to the second-order estimator.

For simplicity, we will show \eqref{eq:hat_d_2;error} for a slight modification of the estimator \eqref{eq:Deps2}. Namely, instead of the cutoff $\one_{\R^+}(\hat{\nu}_\rr^2(x^i)\cdot\hat{\nu}_\rr^2(x^0))$, we use $\one_{\{x:\,x\leq c\rr\}}(|\hat{\nu}_\rr^2(x^i)-\hat{\nu}_\rr^2(x^0)|)$ for suitably large $c$, say, twice the Lipschitz constant of $d_\Omega(\cdot)$. Note that this is a reasonable cutoff, as
\[|\hat{\nu}_\rr^2(x^i)-\hat{\nu}_\rr^2 (x^0)|\leq |\hat{\nu}_\rr^2(x^i)-\nu(x^i)|+|\nu(x^i)-\nu(x^0)|+|\nu(x^0)-\hat{\nu}_\rr^2(x^0)|.\] 
From Section \ref{sec:asym2} we know that the first and third terms are small are of order $O(\rr^2)$ when $\rr\gtrsim(\log n/n)^{1/(d+4)}$; the second term is of order $O(\rr)$ as $\nu(x)=\nabla d_\Omega(x)$ near the boundary, which is a $C^2$ function as we assumed $\partial\Omega$ to be of class $C^3$. Thus, for sufficiently small $\rr$ we have
\[|\hat{\nu}_\rr^2(x^i)-\hat{\nu}_\rr^2 (x^0)|\leq \frac{c}{2}|x^i-x^0|+O(\rr^2) \leq c\rr.\]

\vspace*{10pt}

\textbf{Upper bound.} For the upper bound, suppose $d_\Omega(x^0)\leq 2\eps$. Fix $c',C'>0$ and $\rr>0$, and denote by $E_0$ the event 
\[E_0:=\{|\hat\nu_r^2(x^i)-\nu(x^i)|\leq C'\rr^2\,\,\text{ for all } x_i\in B(x^0,\rr)\cap\X \text{ such that } d_\Omega(x^i)\leq \rr/\sqrt{d}\}.\]
Recall from Section \ref{sec:asym2} that $E_0$ occurs with high probability when $\rr\gtrsim (\log n/n)^{1/(d+4)}$ and $C'>0$ is chosen suitably large. 

For simplified notation, let us temporarily define $\hat u^i(x^0)$ for each $i=1,\cdots,n$ by
\begin{equation}\label{def:hatu}
    \hat u^i(x^0):=\left[\hat{\nu}^{2}_\rr(x^0)+\frac{\hat{\nu}^{2}_\rr(x^i)-\hat{\nu}^{2}_\rr(x^0)}{2}\one_{\{x:\,x\leq c\rr\}}(|\hat{\nu}_\rr^2(x^i)-\hat{\nu}_\rr^2(x^0)|)\right],
\end{equation}
so that $\hat d_\rr^2(x^0)=\max_{x^i\in B(x^0,\rr)\cap\X}(x^0-x^i)\cdot \hat u^i(x^0)$. Define the set $\hat\X$ by
\[\hat\X :=\{x^i\in\X:(x^0-x^i)\cdot\hat u^i(x^0)\geq 0\}.\]
Then we may write
\[\hat d_\rr^{2}(x^0)=\max_{x^i\in B(x^0,\rr)\cap\X}(x^0-x^i)\cdot \hat u^i(x^0)=\max_{x^i\in B(x^0,\rr)\cap\hat\X}(x^0-x^i)\cdot \hat u^i(x^0).\]
Indeed the right-hand side is the nonnegative part of $\hat d_\rr^2(x^0)$, while $\hat d_\rr^2(x^0)\geq 0$ due to that $x^0\in B(x^0,\rr)\cap\hat\X$. Thus the above equality holds.

Due to the cutoff, note 
\[\left|\frac{\hat u^i(x^0)}{|\hat u^i(x^0)|}-\nu(x^0)\right|\leq \left|\frac{\hat u^i(x^0)}{|\hat u^i(x^0)|}-\hat\nu_\rr^2(x^0)\right|+|\hat\nu_\rr^2(x^0)-\nu(x^0)| \leq c\rr+O(\rr^2)\leq 2c\rr\]
for sufficiently small $\rr$. Thus, if $x^i\in\hat\X$, it is in the half plane opposite of $\hat u^i(x^0)$, which is closely approximated by the half plane opposite of $\nu(x^0)$. As $d_\Omega(x^0)\leq 2\eps$, collecting the errors due to curvature of the boundary and the difference between $\hat u^i(x^0)/|\hat u^i(x^0)|$ and $\nu(x^0)$, we see
\[d_\Omega(x^i)\leq 2\eps + \frac{\rr^2}{R}+ 2c\rr^2 \leq \frac{\rr}{\sqrt{d}}\]
when $\eps\ll\rr$ and $\rr$ is sufficiently small. Thus, by $E_0$ we have $|\hat{\nu}_\rr^2(x^i)-\nu(x^i)|\leq C'\rr^2$ for all $x^i\in\hat\X$, and
\[\hat d_\rr^{2}(x^0)=\max_{x^i\in B(x^0,\rr)\cap\hat\X}(x^0-x^i)\cdot\frac{\hat{\nu}_\rr^2(x^i)+\hat{\nu}_\rr^2(x^0)}{2}.\]
Now, when $\partial\Omega$ is of class $C^3$, recall \eqref{eq:test2} holds. Thus, we have
\[d_\Omega(x^0)\geq \max_{x^j\in B(x^0,\rr)}\left\{\frac12(\nu(x^0)+\nu(x^j))\cdot (x^0-x^j)\right\}+O(\rr^3).\]
Then we have the upper bound on $\hat d_\rr^2$
\[\hat d_\rr^2(x^0)-d_\Omega(x^0)\leq \max_{x^i\in B(x^0,\rr)\cap\hat\X}\left\{\frac12\left(\hat{\nu}_\rr^2(x^i)+\hat\nu_\rr(x^0)-\nu(x^0)-\nu(x^i)\right)\cdot(x^0-x^i)\right\}+O(\rr^3)=O(\rr^3),\]
as $|x^0-x^i|\leq\rr$ and $|\hat{\nu}_\rr^2(x^i)-\nu(x^i)|+|\hat{\nu}_\rr^2(x^0)-\nu(x^0)|\lesssim r^2$.

\vspace*{10pt}

\textbf{Lower bound.} Recall the elementary equality $\tfrac{|u+w|^2}{4}=1-\tfrac{|u-w|^2}{4}$ that holds when $|u|=|w|=1$. This implies the following lower bound on the magnitude of $\hat u^i(x^0)$ defined in \eqref{def:hatu}
\begin{equation}\label{eq:hatui_lowbd}
    |\hat u^i(x^0)|\geq (1-c'\rr^2)^{1/2}.
\end{equation}

Writing $\alpha=d_\Omega(x^0)\wedge \frac{\rr}{2}$, under the assumptions of Lemma \ref{lem:hatd1_lowbd}, we have
\begin{align*}
    \P(\hat d_\rr^2(x^0)\leq (1-\lambda)\alpha-t)
    &=\P\left(\max_{x^i\in B(x^0,\rr)\cap\X}(x^0-x^i)\cdot\hat u^i\leq (1-\lambda)\alpha-t\right)\\
    &=\P\left(\max_{x^i\in B(x^0,\rr)\cap\X}(x^0-x^i)\cdot\frac{\hat u^i}{|\hat u^i|}
    \leq \frac{1}{|\hat u^i|}((1-\lambda)\alpha-t)\right).
\end{align*}

By \eqref{eq:hatui_lowbd}, we can fix $C>0$ such that $\frac{1}{|\hat u^i|}\leq \frac{1}{\sqrt{1-c'\rr^2}}\leq 1+C\rr^2$ when $\rr$ is sufficiently small. As $t<\alpha\leq \rr$, we have
\begin{align*}
    \P(\hat d_\rr^2(x^0)\leq (1-\lambda)\alpha-t)
    &\leq \P\left(\max_{x^i\in B(x^0,\rr)\cap\X}(x^0-x^i)\cdot\frac{\hat u^i}{|\hat u^i|}
    \leq (1-\lambda)\alpha-t+C\rr^2((1-\lambda)\alpha-t)\right) \\
    &\leq \P\left(\max_{x^i\in B(x^0,\rr)\cap\X}(x^0-x^i)\cdot\frac{\hat u^i}{|\hat u^i|}
    \leq (1-\lambda)\alpha-t+C\rr^3\right)   
    \leq n^{-\gamma}.
\end{align*}
The last inequality follows when $t>C\rr^3$ by Lemma \ref{lem:hatd1_lowbd}, as its proof only uses that $|\hat\nu_\rr(x^0)|=1$. Choosing $t=2C\rr^3$ and $\lambda\leq C\rr^2$ for instance, we obtain that $\hat d_\rr^2(x^0)\geq d_\Omega(x^0)-3C\rr^3$ with high probability, and the condition \eqref{eq:tr_lowbd} becomes $\rr\gtrsim(\log n/n)^{1/(d+2)}$. Note that this is less restrictive than the scaling $\rr\gtrsim (\log n/n)^{1/(d+4)}$, required for the upper bound. While Lemma \ref{lem:hatd1_lowbd} also requires $n\geq d\wedge 4\lambda^{-1}$, but this is a much milder condition when $\lambda\sim \rr^2$. Thus we deduce that \eqref{eq:hat_d_2;error} holds with high probability, when $\rr\gtrsim (\log n/n)^{1/(d+4)}$.

\section{Algorithms and Experiments}\label{sec:algorithm_experiments}
     We now turn to the algorithms for our boundary tests and related numerical experiments. After presenting the pseudocode for the boundary tests and briefly commenting on the computational complexity, we demonstrate the efficiency and accuracy of our results, focusing on domains with constant positive or negative curvatures. Again we stress that, while the rigorous theoretical results in Section \ref{sec:main} are established for the first-order test, we recommend the second-order test for practical purposes. As we will see, the second-order test takes into account the curvature, hence performs much better than the first-order test.
    
    To begin, we present the pseudocodes for the first- and second-order boundary tests, and the generalization of the second-order test to point clouds supported on manifolds.
    
    \begin{algorithm}[ht] 
    \caption{First-order boundary test}\label{alg:1st}
    \begin{flushleft}
   \textbf{Input:} The set of points $\X=\{x^1,\cdots,x^n\}$, and parameters $\rr,\eps>0$     \\
    \textbf{Output:} $T\left(x_k\right)=1$ if $x_k$ is a $\eps$-boundary point, $0$ if an $\eps$-interior point
    \end{flushleft}
    \begin{algorithmic}[1]
        \For{$i=1 \cdots n$}
            \State $T\left(i\right)\gets 1$
            \State $\hat{v}_\rr(x^i)\gets\sum_{y\in B(x^i,\rr) \cap \X}\left(y-x^i\right)$ \smallskip
            \State $\hat{\nu}_\rr(x^i)\gets \hat{v}_\rr(x^i)/|\hat{v}_\rr(x^i)|$  \smallskip
            \If{$\max_{x^j\in B(x^i,\rr) \cap \X} \, (x^i-x^j)\cdot\hat{\nu}_\rr>\frac{3\eps}{2}$} $T(i)=0$\EndIf
        \EndFor
    \end{algorithmic}
    \end{algorithm}
 
    \begin{algorithm}[ht]
    \caption{Second-order boundary test}\label{alg:2nd_manifold}
    \begin{flushleft}
    \textbf{Input:} The set of points $\X=\{x^1,\cdots,x^n\}$, and parameters $\rr,\eps>0$     \\
    \textbf{Output:} $T\left(x_k\right)=1$ if $x_k$ is a $\eps$-boundary point, $0$ if an $\eps$-interior point
    \end{flushleft}
    \begin{algorithmic}[1]
    \For{$i=1 \cdots n$}
            \State $\hat\theta(x^i)\gets \sum_{j=1}^n \one_{B(x^i,\rr/2)}(x^j)$ \smallskip
            \State $\hat{v}_\rr^2(x^i)\gets\sum_{x^j\in B(x^i,\rr) \cap \X}\frac{\left(x^j-x^i\right)}{\hat\theta(x^j)}$ \smallskip
            \State $\displaystyle{\hat\nu_{\rr}^2(x^i)\gets \hat{v}_\rr^2(x^i)/|\hat{v}_\rr^2(x^i)|}$
    \EndFor
    \For{$i=1 \cdots n$}
            \For{$j=1\cdots n$}
                \State $\displaystyle{\hat{\nu}_{\rr,test}^{ij}=\hat{\nu}_{\rr}^2(x^i)+\frac{\hat{\nu}_{\rr}^2(x^j)-\hat{\nu}_{\rr}^2(x^i)}{2}\one_{\R_+}(\hat{\nu}_{\rr}^2(x^i)\cdot\hat{\nu}_{\rr}^2(x^i))}$ 
            \EndFor
            \If{$\max_{x^j\in B(x^i,\rr) \cap \X}\,(x^i-x^j)\cdot\hat{\nu}_{\rr,test}^{ij}>\frac{3\eps}{2}$} $T(i)=0$
            \EndIf
    \EndFor
    \end{algorithmic}
    \end{algorithm}

        \begin{algorithm}[htb]
    \caption{Second-order boundary test for point clouds supported on manifolds}\label{alg:2nd}
    \begin{flushleft}
   \textbf{Input:} The set of points $\X=\{x^1,\cdots,x^n\}$,  parameters $\rr,\eps>0$, and the dimension of the manifold $m$    \\
    \textbf{Output:} $T\left(x_k\right)=1$ if $x_k$ is a $\eps$-boundary point, $0$ if an $\eps$-interior point 
    \end{flushleft}
    \begin{algorithmic}[1]
    \For{$i=1 \cdots n$}
            \State $\hat\theta(x^i)\gets \sum_{j=1}^n \one_{B(x^i,\rr/2)}(x^j)$ \smallskip
            \State $\hat{v}_\rr^2(x^i)\gets\sum_{x^j\in B(x^i,\rr) \cap \X}\frac{\left(x^j-x^i\right)}{\hat\theta(x^j)}$ \smallskip
            \State $\displaystyle{\hat\nu_{\rr}^2(x^i)\gets \hat{v}_\rr^2(x^i)/|\hat{v}_\rr^2(x^i)|}$ \smallskip
            \State $Y^i\gets$ rangesearch$(x^i,\rr)$ \smallskip
            \State $Y_i\gets Y_i-\overline{Y_i}$ \smallskip
            \State $\{v_1,\cdots,v_m\}\gets$ eigenvectors associated to $m$ largest eigenvalues of $(Y^i-x^i)^T(Y^i-x^i)$ \smallskip
            \State $T^i\gets \text{Span}\{v_1,\cdots,v_m\}$
    \EndFor
    \For{$i=1 \cdots n$}
        \For{$j=1\cdots n$}
            \State $\displaystyle{\hat{\nu}_{\rr,test}^{ij}=\hat{\nu}_{\rr}^2(x^i)+\frac{\hat{\nu}_{\rr}^2(x^j)-\hat{\nu}_{\rr}^2(x^i)}{2}\one_{\R_+}(\Pi^i(\hat{\nu}_{\rr}^2(x^i))\cdot \Pi^i(\hat{\nu}_{\rr}^2(x^i))})$ \smallskip
        \EndFor
        \If{$\max_{x^j\in B(x^i,\rr) \cap \X}\,\Pi^i[(x^i-x^j)]\cdot\hat{\nu}_{\rr,test}^{ij}>\frac{3\eps}{2}$} $T(i)=0$
            \EndIf
    \EndFor
    \end{algorithmic}
    \end{algorithm}

We add that the algorithms can take a percentile $p\%$ as an input instead of $\eps$, so that it outputs the top $p\%$ of points with smallest estimated distance. This may be easier to implement in practice than choosing $\eps$, as the lower bound for $\eps$ depends not only on $n$ but also on $R,\rho$ and $d$. Theoretically, $p\%$ and $\eps$ are interchangeable; we may set the largest estimated distance within the $p\%$ percentile to equal to the threshold, $\frac{3\eps}{2}$.





\begin{remark}[Computational complexity]\label{rmk:complexity}
   Noting that range search task is essentially equivalent to k-nearest neighbor search for suitable $k$, we briefly remark on the computational expense. The best rigorous upper bounds for computing all-kNN for $n$ points in $\R^d$ known to us, without number of parallel processors growing with $n$, are $O(n(\log n)^{d-1})$ \cite{Ben80} and $O(kd^dn\log n)$ \cite{Vai89}. Note that the the suitable choice of $k$ for us is $k\sim \omega_d\rr^d n$, which, under the optimal choice of the test radius $\rr= RC_r(\log n/n)^{\frac{1}{d+2}}\leq C\omega_d^{-1/d}(\log n/n)^{\frac{1}{d+2}}$ for our first-order test, has the following scaling in $n$ and $d$
   \[k\lesssim (\log n)^{\frac{d}{d+2}} n^{\frac{2}{d+2}}.\]
   Please see Remark \ref{rmk:main} for further details. 
   
   While the computational cost of exact all-kNN is not cheap, approximate all-kNN can be performed at nearly linear time in $n$. For instance, the algorithm suggested in \cite{DongMosLi11} reports that empirical cost scales like $n^{1.14}$ on average with above 90 percent accuracy. Python GraphLearning \cite{Github:GraphLearning} package 
the Approximate Nearest Neighbors algorithm (ANNOY) \cite{Github:annoy}, which also provides close to linear scaling in $n$.
\end{remark}

\begin{remark}[Intrinsic dimension of $\M$]
   In practice the intrinsic dimension $m$ of $\M$ often unknown. However there are many ways to recover this from the eigenvalues $\lambda_1\leq \cdots\leq \lambda_d$ of the sample covariance matrix $(Y^i-x^i)(Y^i-x^i)^T$. There are two big drops in the eigenvalue distribution. Near the boundary, eigenvectors sufficiently parallel to the normal direction have smaller eigenvalues due to the absence of points one one side of $\partial\M$. However, this gap should not reduce the eigenvalues much more than halving. On the other hand, $\lambda_{m+1},\cdots,\lambda_{d}$ are due to curvature, and thus are much smaller compared to the first $m$ when curvature is bounded. Thus we may recover the dimension $m$ by for instance, counting the number of eigenvalues before the steepest drop in ratio $\frac{\lambda_{i+1}}{\lambda_i}$.
\end{remark}

We now describe the setting of our numerical experiments. In Figures \ref{fig:dist_bdry}, \ref{fig:subplots}, and \ref{fig:asymp} we consider two types of domains: a ball, and an annulus, both with reach $R=0.5$. Recall that this means the ball has radius $R$ and the annulus has inner and outer radii $R_1=R,\, R_2=1.6 R$. By the boundary of the ball mean the sphere, and by that of the annulus we refer only to the inner boundary $\{x:|x|=R_1\}$, so we can observe how the test performs when the curvature is negative. Thus we test only the points satisfying $|x|\in[R_1,R_2-\rr]$.

    
We consider the density function $\rho$ parametrized by the Lipschitz constant $L$. The sinusoidal density has the form 
\begin{equation}\label{eq:rho_L}
    \rho(x)=\frac1{|\Omega|}\left(1+\frac12\sin(L|\Omega|x_1)\right),
\end{equation} so that $\sup_{x\in\Omega}|\partial_1\rho(x)|=L$. 
Note that our theory in Section \ref{sec:main} applies to Lipschitz functions that are not necessarily of class $C^1$. Indeed, we note that results obtained using the triangular wave density were similar. 

The boundary tests are as described in \eqref{def:test}, where the first-order test (`1st') uses the distance estimator \eqref{eq:D_hat}, and the second-order test (`2nd') uses the estimator \eqref{eq:Deps2}.

\medskip
  \begin{figure}[h]
       \centering
       \includegraphics[trim={60pt 120pt 40pt 100pt}, clip=true, width=0.49\textwidth]{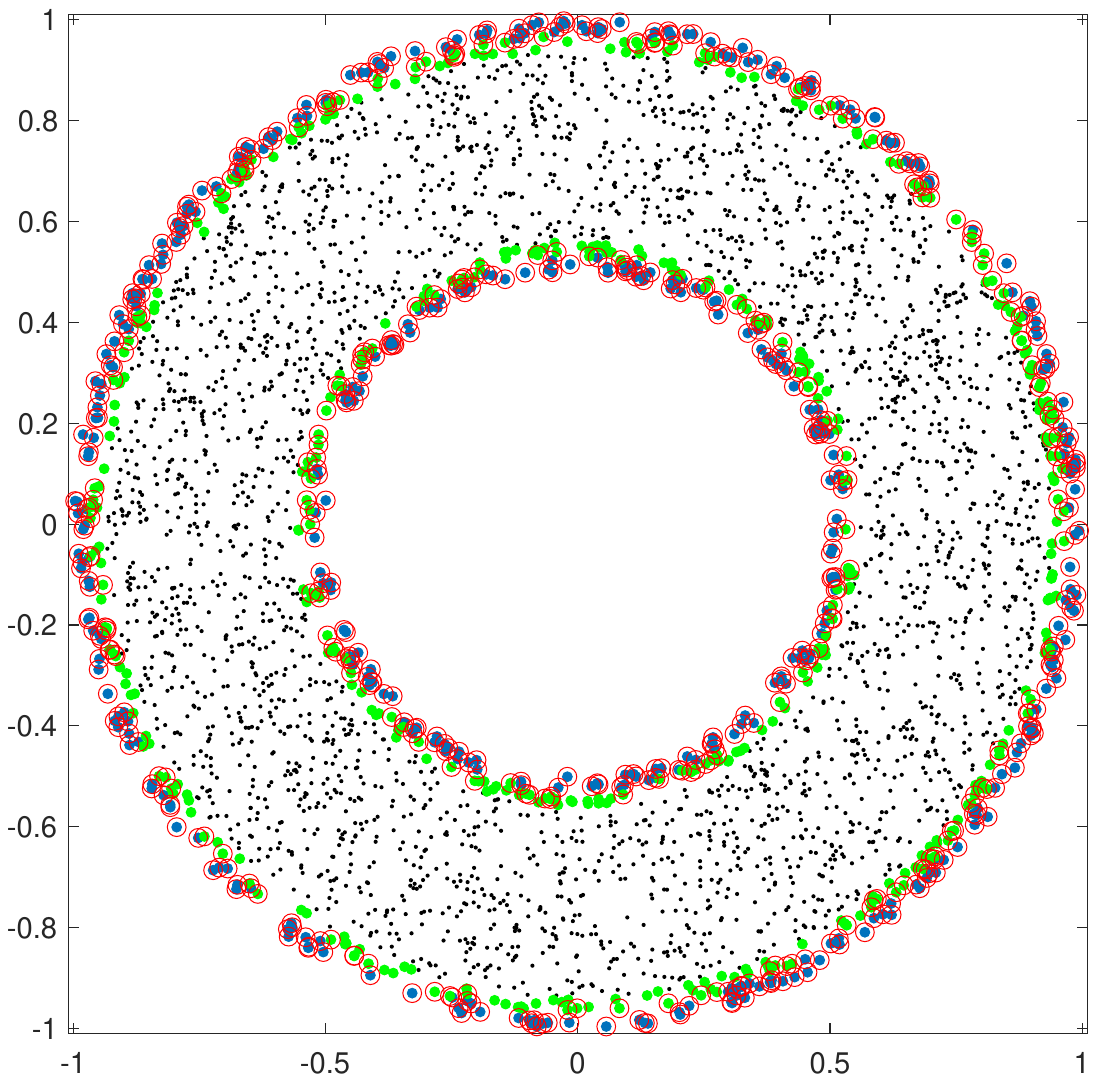}
       \includegraphics[trim={60pt 120pt 40pt 100pt}, clip=true, width=0.49\textwidth]{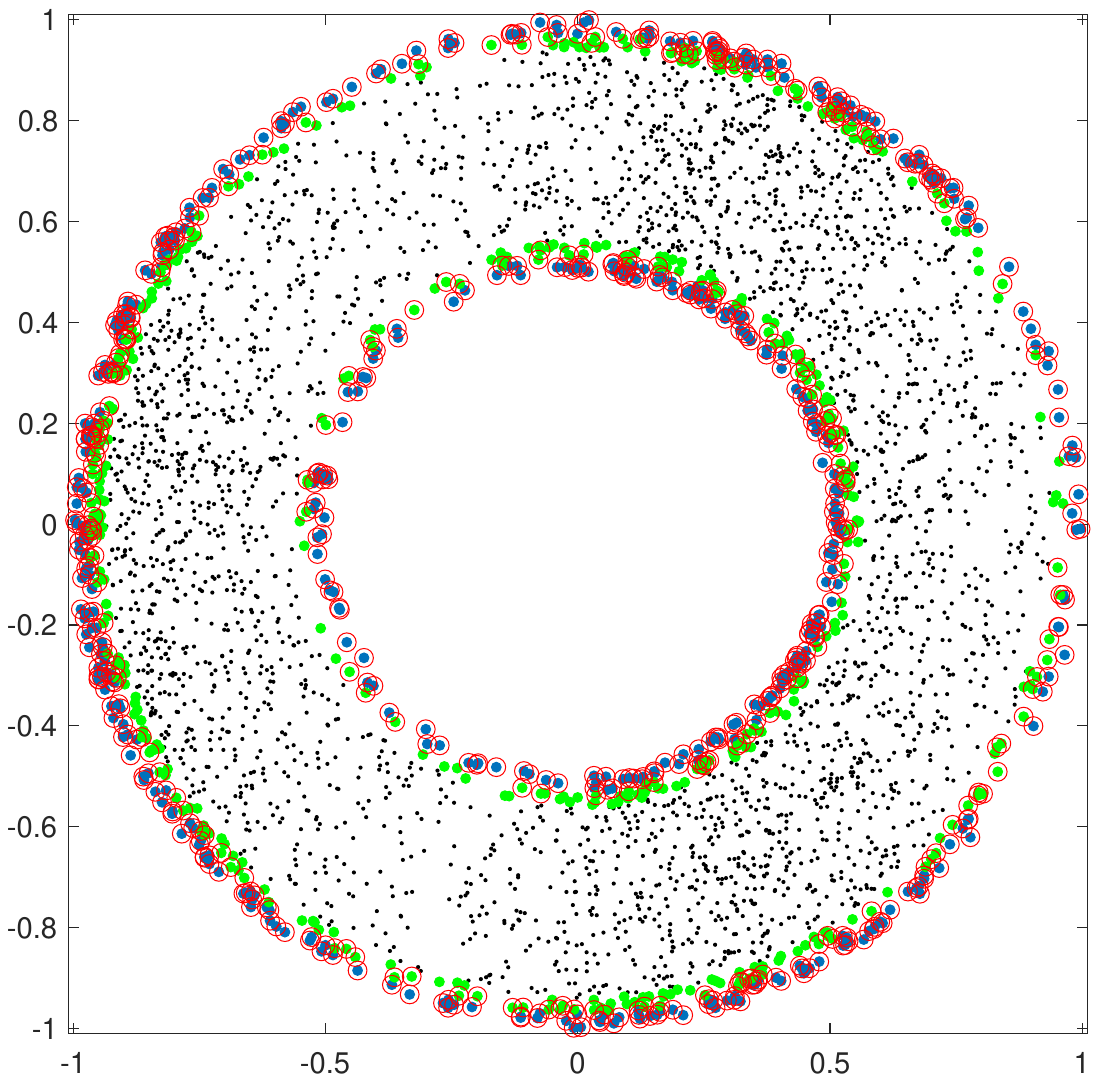}
       \caption{Boundary test on an annulus with inner and outer radii $0.5$ and $0.8$, respectively. $n=2000$ points are drawn from uniform density on the left and sinusoidal density with $L=2$ on the right. The point cloud is represented by black dots, while blue and green dots are the points whose true distance to the boundary are in $[0,\eps)$ and $[\eps,2\eps)$ respectively, for $\eps = 0.03$. The red circles show the points identified by the 2nd order test, with $r=0.18$, as boundary points. Observe that most blue dots are indeed correctly identified, and almost all points identified by the test are either blue or green dots.}
       \label{fig:dist_bdry_pts}
   \end{figure}
    
\textbf{Measuring the test error.} 
 Let $\eps,\rr>0$ be the boundary width and the test radius. 
 Given a test we are considering, let
 the set of \emph{tested boundary points} be the set of points in $\X$ where the test $\widehat{T}_{\eps,\rr}$ defined in \eqref{def:test}. The \emph{tested interior points} is  the complement of the tested $\eps$-boundary points in $\X$.
 Let $P$ be the number of tested boundary points and $N$ the number of tested interior points:
 \[ P = \sharp \{x^i \in \X \::\: \widehat{T}_{\eps,\rr}(x^i) =1\} \quad \te{ and }\quad 
  N = \sharp \{x^i \in \X \::\: \widehat{T}_{\eps,\rr}(x^i)= 0\}. \]

 
 We measure the error rate in a different way than is standard in hypothesis testing. We do it in a way that measures better whether we succeeded in our stated goal to create a test that would identify a large percentage of points near the boundary and would not misidentify as boundary points almost any points deep in the interior. This is important to be able to accurately set boundary conditions for PDE. 
   
      \begin{figure}
       \centering 
       \subfloat[$3$D ball with $R=0.5$, and $n=4000$. Left panel $L=0$ and right panel $L=2$.]{
       \includegraphics[width=0.49\textwidth]{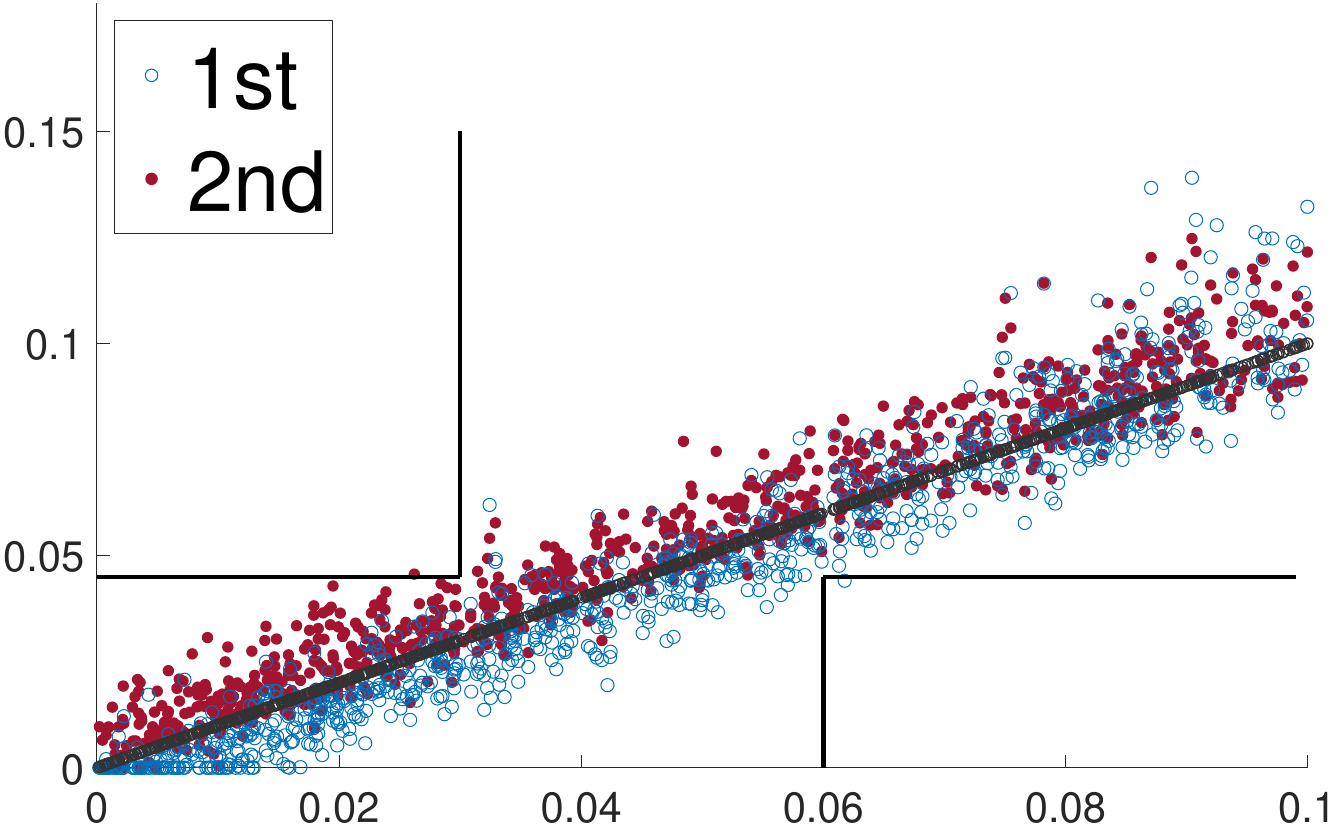}
       \includegraphics[width=0.49\textwidth]{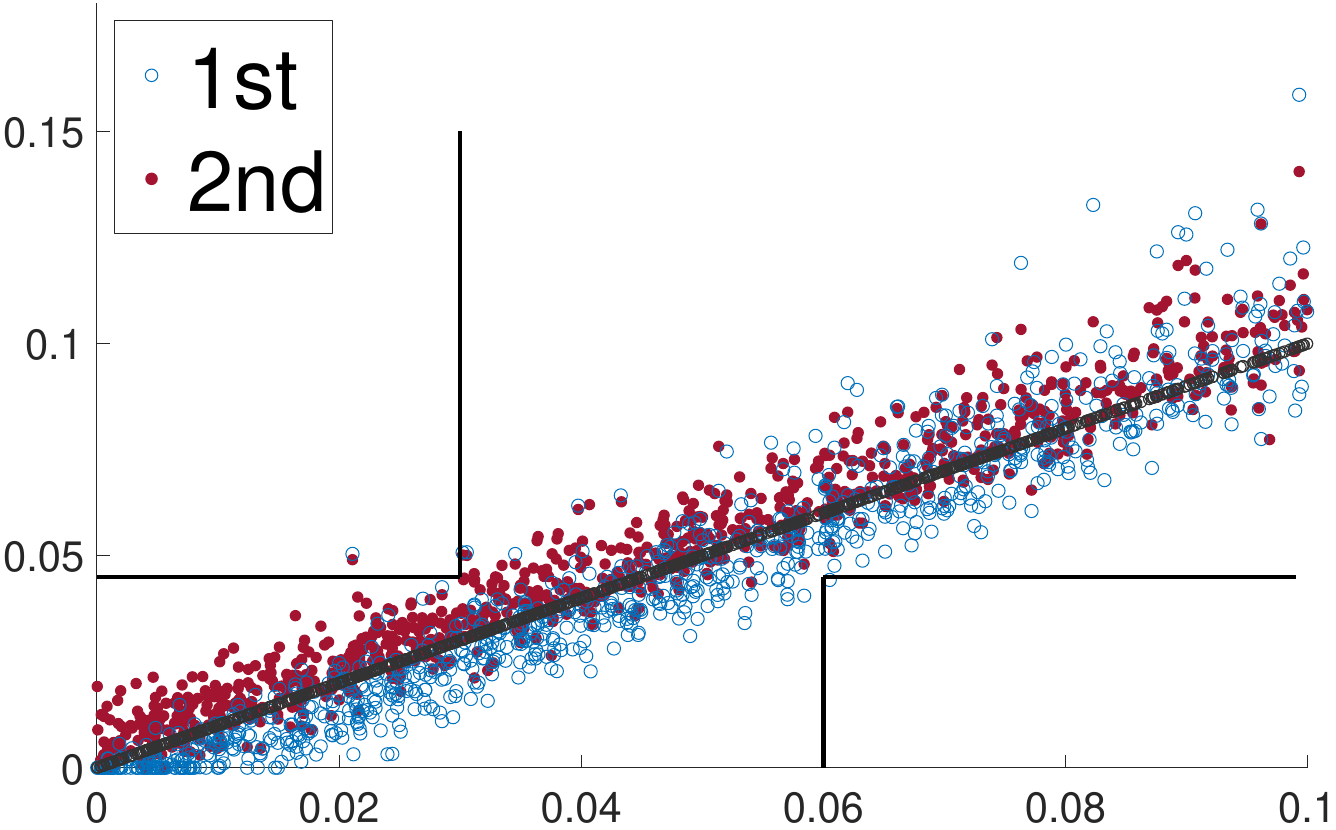}
       }
       
       \subfloat[$3$D annulus with $R_1=0.5, R_2=0.8$, and $n=12000$. Left panel $L=0$ and right panel $L=2$.]{
       \includegraphics[width=0.49\textwidth]{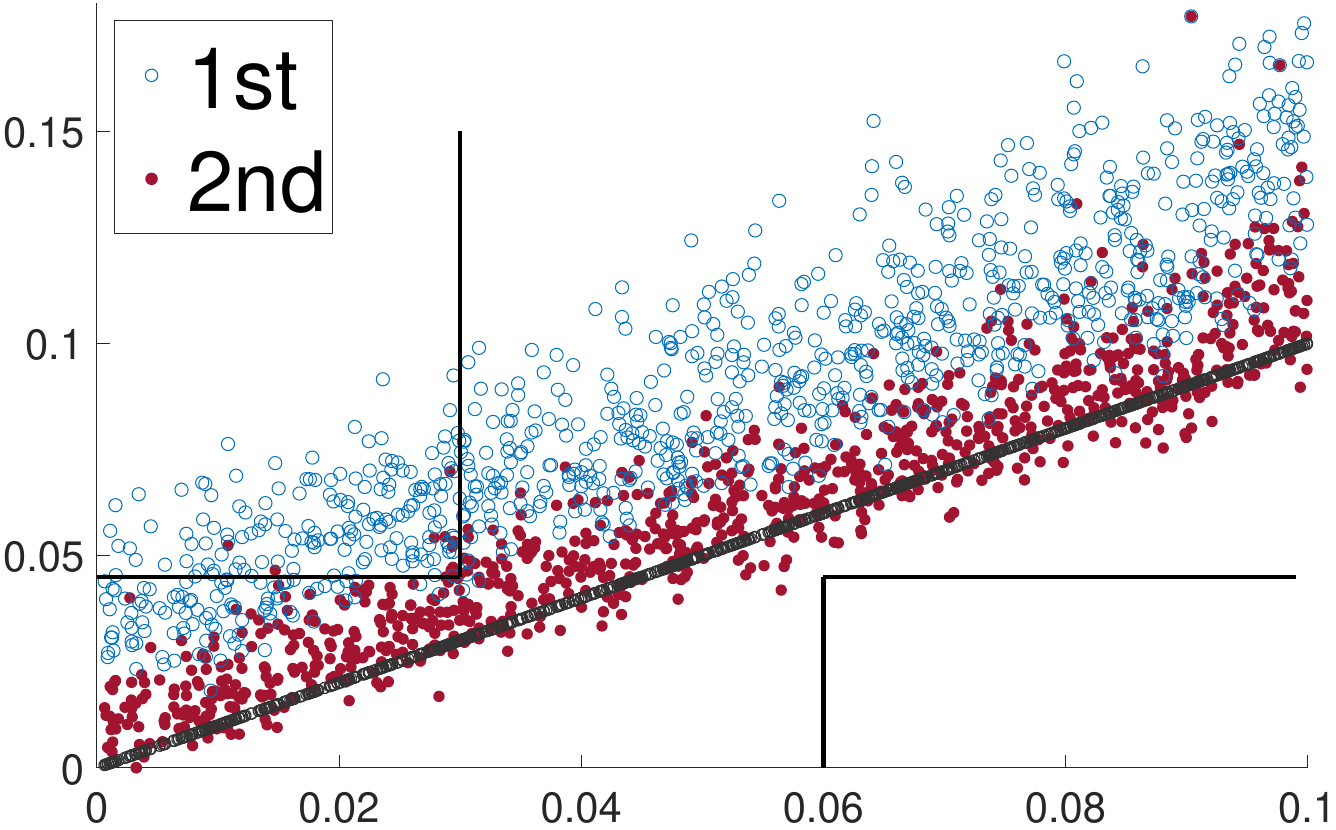}
       \includegraphics[width=0.49\textwidth]{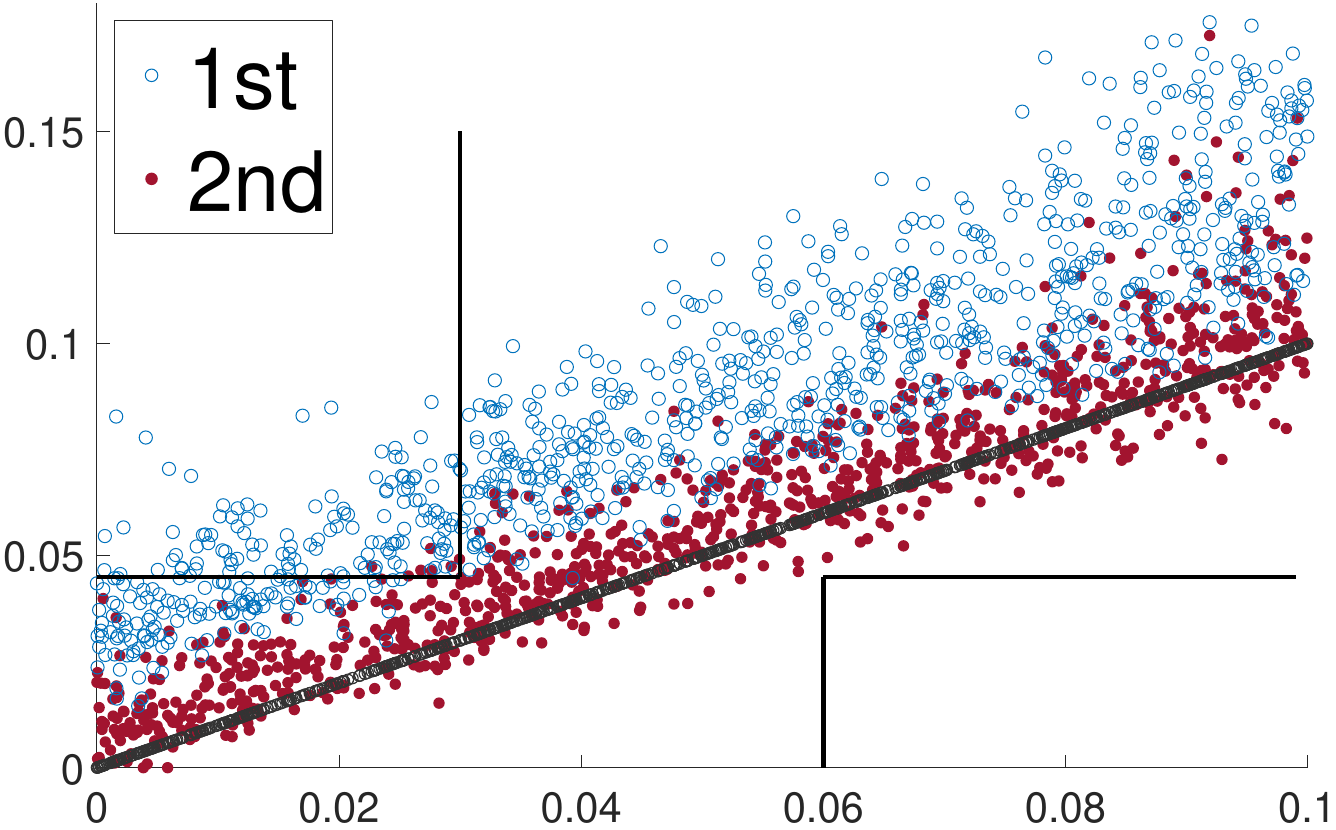}
       }
       \caption{ Plot of distance to boundary with $\eps=0.03$, $\rr=0.18$. $x$- and $y$-axes each represents the true and the estimated distances respectively. 1st and 2nd 
       refer to the order of the algorithm used.
       The boxes in the upper left and lower right corners specify the region for false negatives and false positives respectively. Only 1000 relevant points are plotted for improved visibility. Clear trend of 1st underestimating (resp. overestimating) the distance in a domain of positive (resp. negative) curvature is observed.
       }
       \label{fig:dist_bdry}
   \end{figure}
 
 Thus we  refer to 
 $\partial_{\eps}\Omega=\{x\in \X:\dist\left(x,\Omega\right)\leq\eps\}$ and $ \Omega^\circ_{2\eps}=\{x\in \X:\dist\left(x,\Omega\right)>2\eps\}$
 as true boundary and true interior points, respectively.
 We refer to tested boundary points which lie in $\Omega^\circ_{2\eps}$ as false positives and tested interior  points which lie in $\partial_\eps \Omega$ as false negatives. 
 We denote the number of false positives and false negatives by 
 \[ FP = \sharp \{x \in \X \cap \Omega^\circ_{2\eps} \::\:  \widehat{T}_{\eps,\rr}(x^i) =1 \} \quad  \te{ and } \quad 
 FN = \sharp \{x \in \X \cap \partial_\eps \Omega  \::\:  \widehat{T}_{\eps,\rr}(x^i) =0 \}. \] 
 We denote  by $BP$ the number of true boundary points 
 $ BP = \sharp (\X \cap \partial_\eps \Omega). $
 We define  false negative rate (FNR) and false positive rate (FPR) by
 \[ FNR =\frac{FN}{BP}\quad \te{ and } \quad FPR=\frac{FP}{BP}.\] 
 By the test failure rate (TFR) we mean the sum of FNR and FPR.
 Note the unusual  definition of FPR. From the point of view  hypothesis testing FPR would be the ratio of FP and 
 true interior points. Given the large number of true interior points
 such measure of error would be small even if there is a significant number of points that were misidentified as boundary points. For  our purposes it is important that the impact of false positives is small to the impact of the true positives. Thus we measure the error much more stringently and compare the  number of the false positives to the number of true boundary points.

 
 

    \begin{remark}[Smoothing the estimated normals]\label{rmk:smoothing}
       We observed that it is possible to further improve the accuracy of the estimated normals, thus of the test, if we smooth the normals in a small neighborhood using a suitable kernel. This reduces the variance, and tends to work well in combination with the second-order normal vector estimator \eqref{eq:batvn}, which limits the bias even in the presence of fluctuations in the density. However, when 
       the second derivatives of the density $\rho$ are large there can be a large bias in the estimated normal. In such cases we found  that smoothing may worsen accuracy as errors accumulate. 
    \end{remark}
    
    \begin{figure}[htbp]
        \centering
        \subfloat[Summary for the ball with sinusoidal density. $\eps=0.03$.]{
        \includegraphics[width=0.88\textwidth]{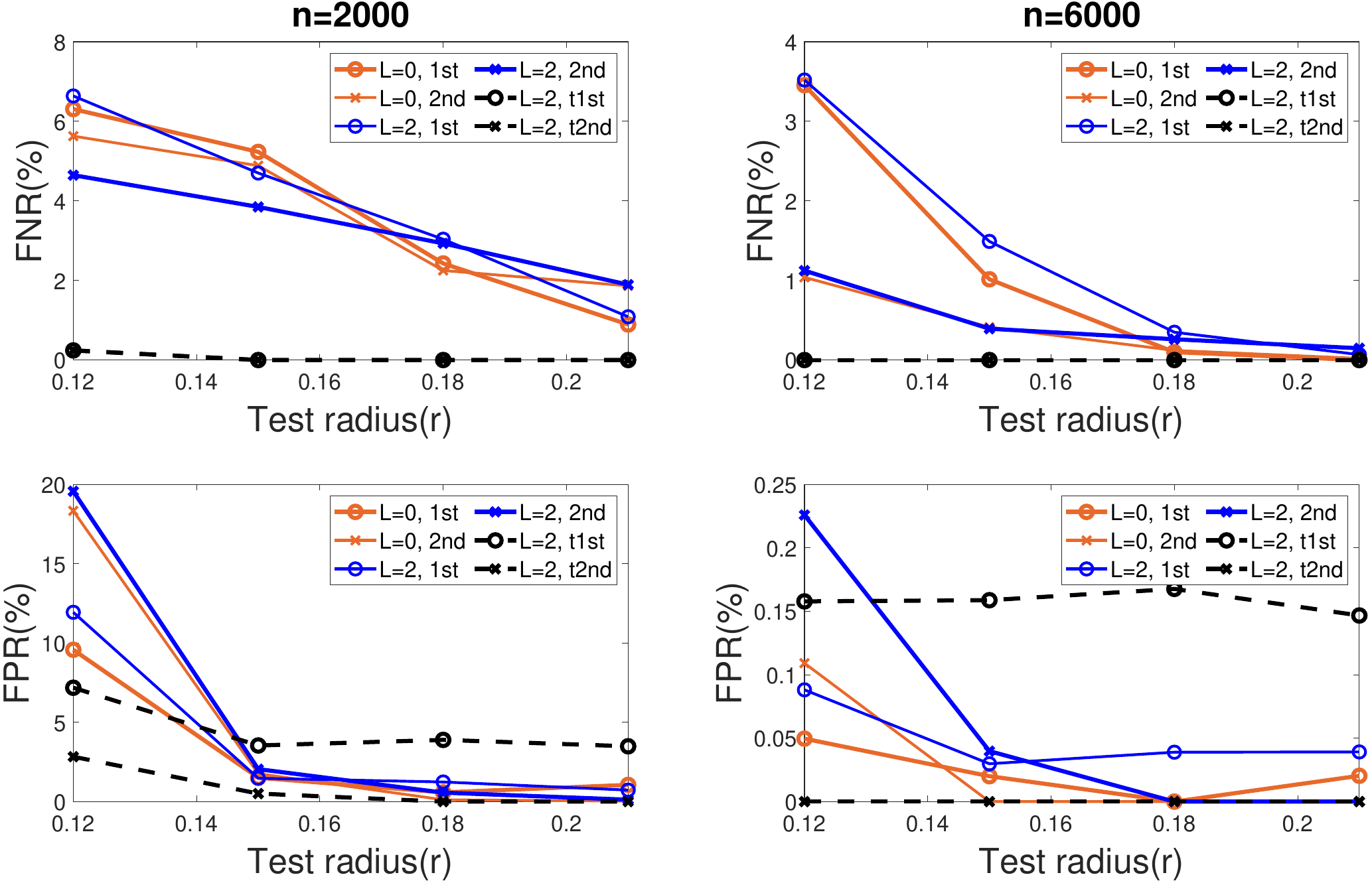} 
        }
    
        
        \subfloat[Summary for the annulus with sinusoidal density. $\eps=0.03$.]{
        \includegraphics[width=0.88\textwidth]{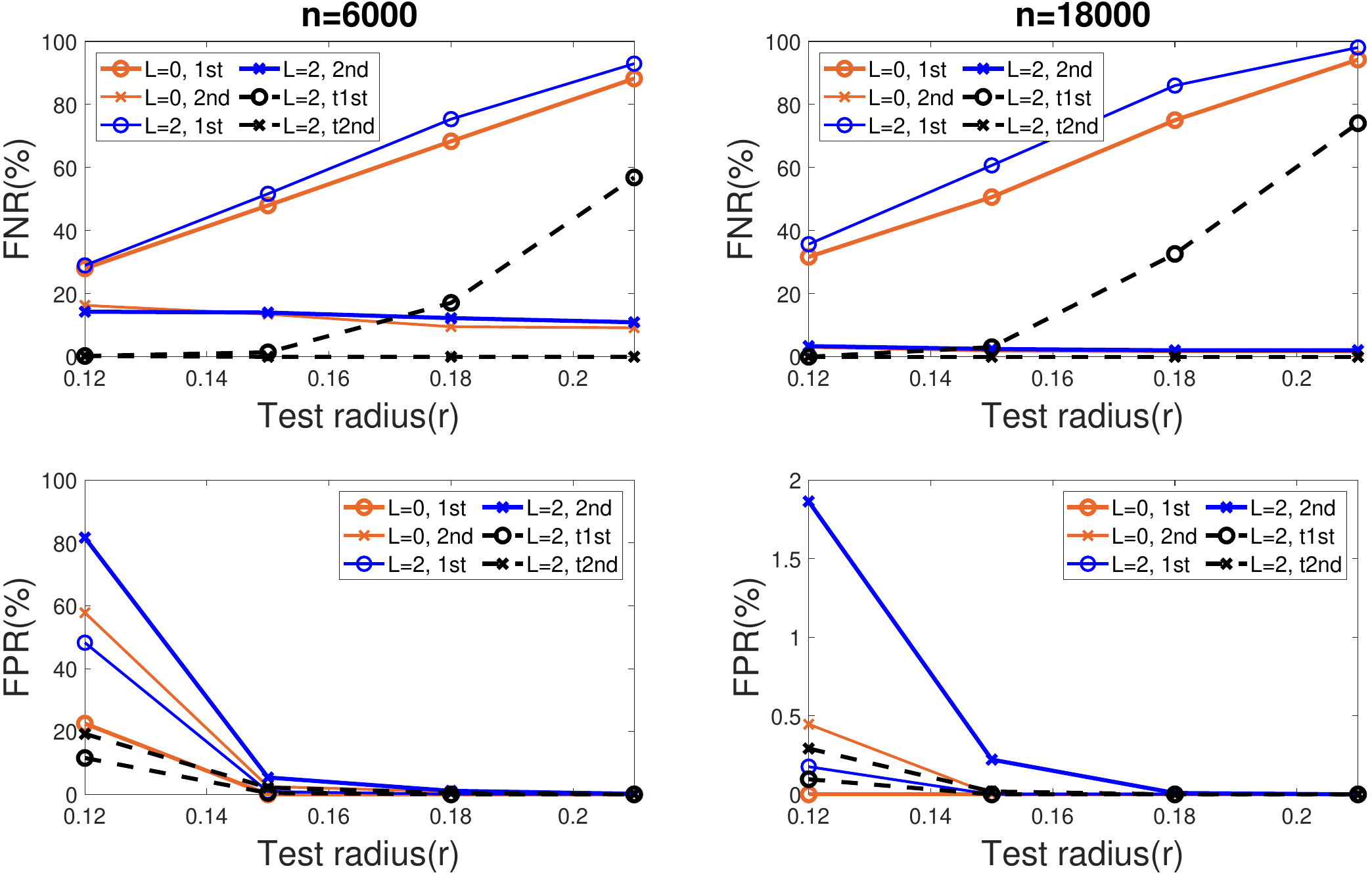} 
        }
        \caption{Test failure rates depending on the test radius, number of points, sign of curvature, and the type of tests. $\eps=0.03$, $R=0.5$. 1st, and 2nd are as in the previous experiments, while t1st and t2nd denote the first and second-order tests using the true normal vectors. Results have been averaged over 10 independent runs.}
        \label{fig:subplots}
    \end{figure}
    
    \begin{figure}[ht]
    \centering
    \includegraphics[width=0.49\textwidth]{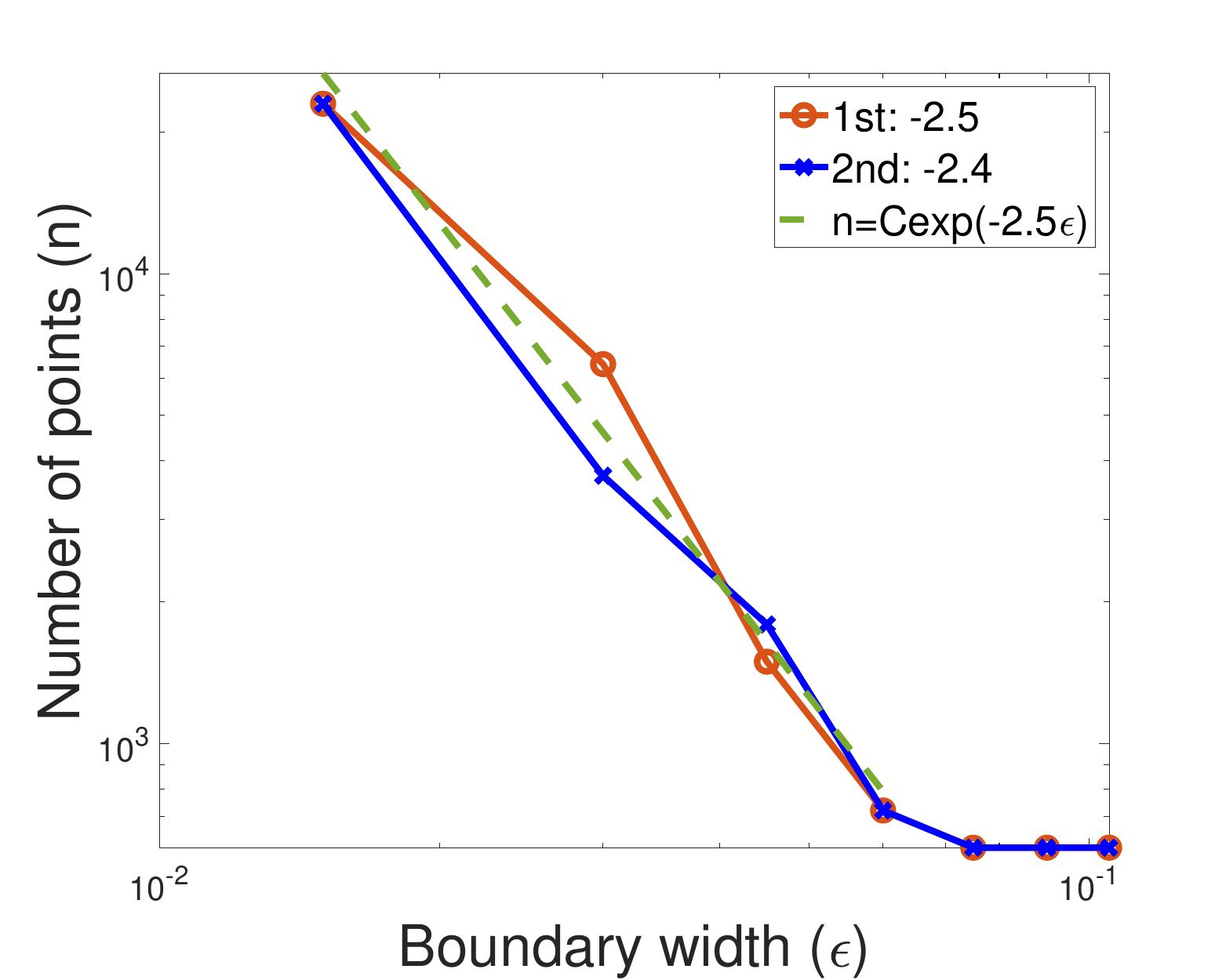}
    \includegraphics[width=0.49\textwidth]{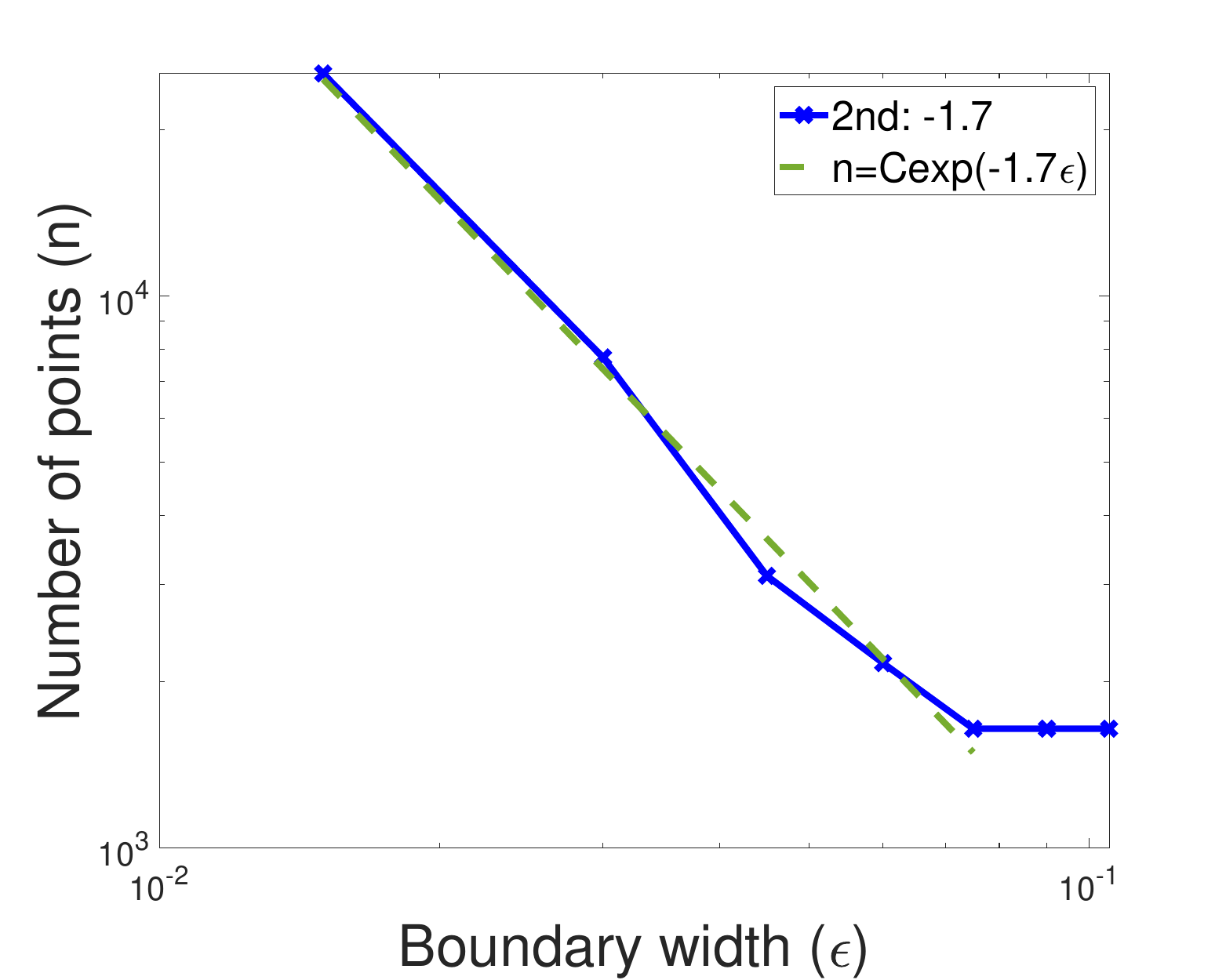} 
    \vspace*{-5pt}
    \caption{
    The plot shows   the smallest number of points $n$ for which TFR$\leq$ threshold, for given boundary width $\eps$.
    (Left) Ball, threshold$=0.5\%$, (Right) Annulus, threshold$=10\%$.
    Maximal $n$ considered was 20000  for the ball and 25000 for the annulus.
    We considered density with $L=2$, 
    and $\rr=\sqrt{\eps}$.  Number in the legend indicate the slope until $n$ becomes stable. 1st order test applied to negatively curved domains have high false negatives, hence the TFR never went below the threshold. Hence the results from the 1st order test is not included for the annulus. 
    Results have been averaged over 10 independent runs.}\label{fig:asymp}
    \end{figure}
    
    In Figure \ref{fig:asymp} we see that the first-order test for the ball shows $n\sim \eps^{-2.5}$, corresponding almost exactly to the optimal theoretical scaling $\eps\sim \rr^2,\,\eps\sim(\log n/n)^{2/(d+2)}$ established in Corollary \eqref{corol:boundary_test}. We see similar trends with the second-order test for the ball. However, the first-order test shows extremely poor performance for the annulus, due to the negative curvature. For it to work, we need $n$ large and $\eps,\rr$ small enough so that the curvature is negligible. On the other hand, the normalized second-order test shows exponential relationship between $n$ and $\eps$, although the exponent is worse than its counterpart for the ball.
    
    \begin{remark}[Choice of parameters $\eps,\rr$]\label{rmk:choice_param}
      
     We have established in Theorem \ref{thm:main} that the optimal scaling for the first-order test is $\rr\sim(\log n/n)^{1/(d+2)}$ and $\eps\sim \rr^2$ as $n\rightarrow\infty$. However, in practical situations, often $n$ is not sufficiently large to guarantee that such scaling is realistic. Then how should we choose $\eps$ and $\rr$?
     
    We observe from Figure \ref{fig:subplots} that the 2nd order test with the true normal vectors (t2nd) gives close to perfect results for both domains. This suggests that the 2nd order test for the most part resolves the challenge posed by curvature, which 1nd order test suffers from, and accurate estimation of normal vectors is key to boosting performance of the boundary test.
     
    There are trade-offs in choosing $\rr$: clearly, when $\rr$ is too small, the estimated normal is inaccurate due to high variance. On the other hand, large $\rr$ leads to larger bias caused by curvature or fluctuations in the density. However, in Section \ref{sec:asym2} we have showed that the normalization by degree in the 2nd order estimator for the normal vector limits the bias to $O(\rr^2)$ even when $\rho$ is non-uniform. Indeed, we see in Figure \ref{fig:subplots} (b) that FNR of 2nd is close to that of t2nd even in the presence of nontrivial fluctuation with $L=2$ and relatively large $\rr$.
    
    Thus, using the 2nd order test, it suffices to choose $\rr$ in a reasonable range, so that $B(x^0,\rr)$ contains sufficiently many points, and $\rr$ is not too close to the reach $R$, when a rough estimate of $R$ is known. When the reach is completely unknown, then we recommend that $\rr$ is taken to be the smallest so that each ball of radius $\rr$ contains sufficient number of points.

    Given $\rr$, $\eps$ should be chosen so that the ratio $\tfrac{|B(x^i,\frac{3\sqrt{2}\eps}{2})|}{|B(x^i,\rr)|}$ of the volume of the balls is no larger than, say, $\frac12$, to limit the number of false positives. The particular coefficient $\tfrac{3}{2}\sqrt{2}$ is is chosen as the threshold of our test is at $\tfrac{3\eps}{2}$, and the $\tfrac{\hat\nu(x^i)+\hat\nu(x^j)}{2}$ can have magnitude as small as $\tfrac{1}{\sqrt{2}}$ when the sharp cutoff function is used. Note that for fixed $\rr,\eps$, the ratio of the volumes decreases in dimension, as volume concentrates near the boundary of the ball in high dimensions. On the other hand, $\eps$ should be large enough so that the strips of height $\frac{\eps}{2}$ and width around $\rr$ contain enough points; this limits the possibility that points $y$ with $d_\Omega(y)$ around $2\eps$ are falsely tested positive. See Figure \ref{fig:false_positive} and Lemma \ref{lem:hatd1_lowbd} for details.
    \end{remark}
    
    \begin{figure}[ht]
    \centering
    
    \includegraphics[width=0.49\textwidth]{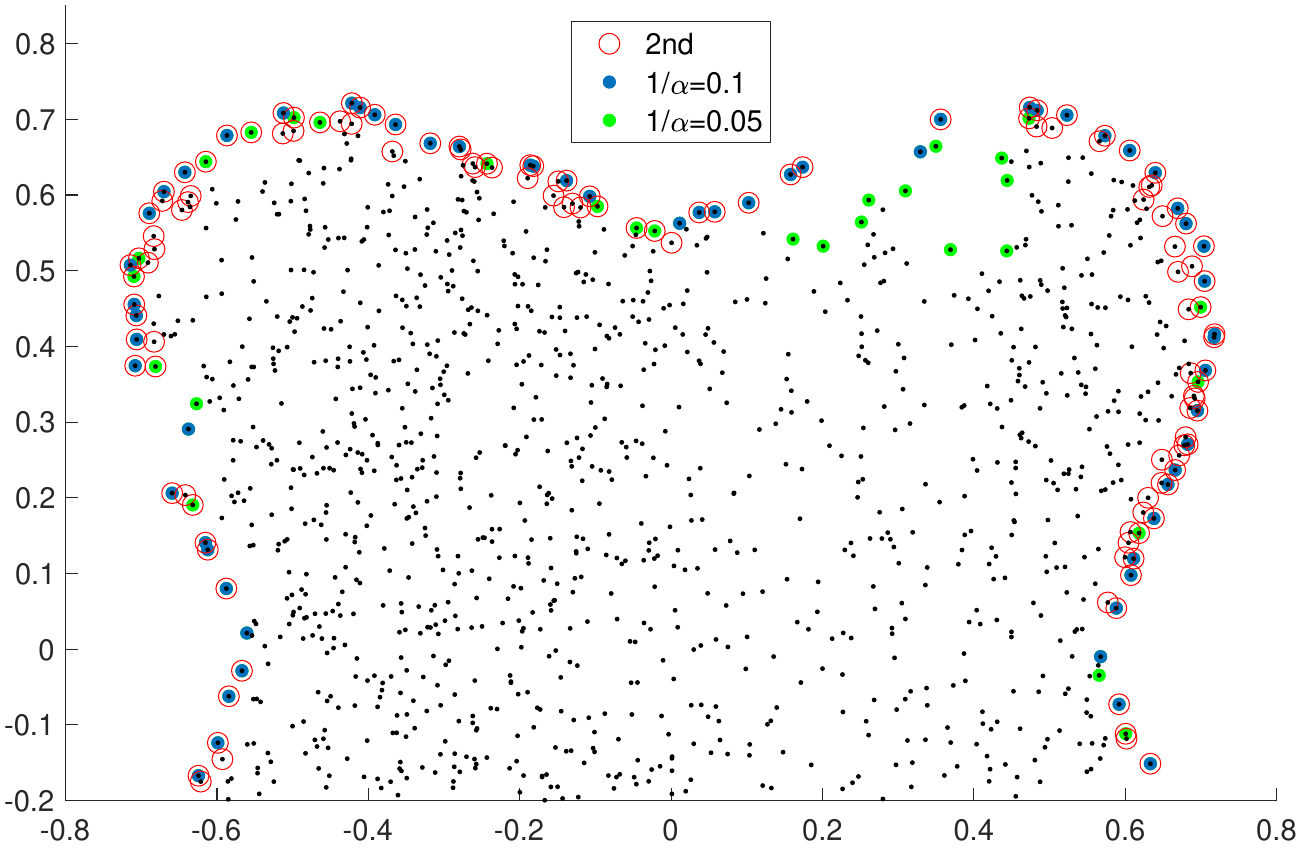}
    \includegraphics[width=0.49\textwidth]{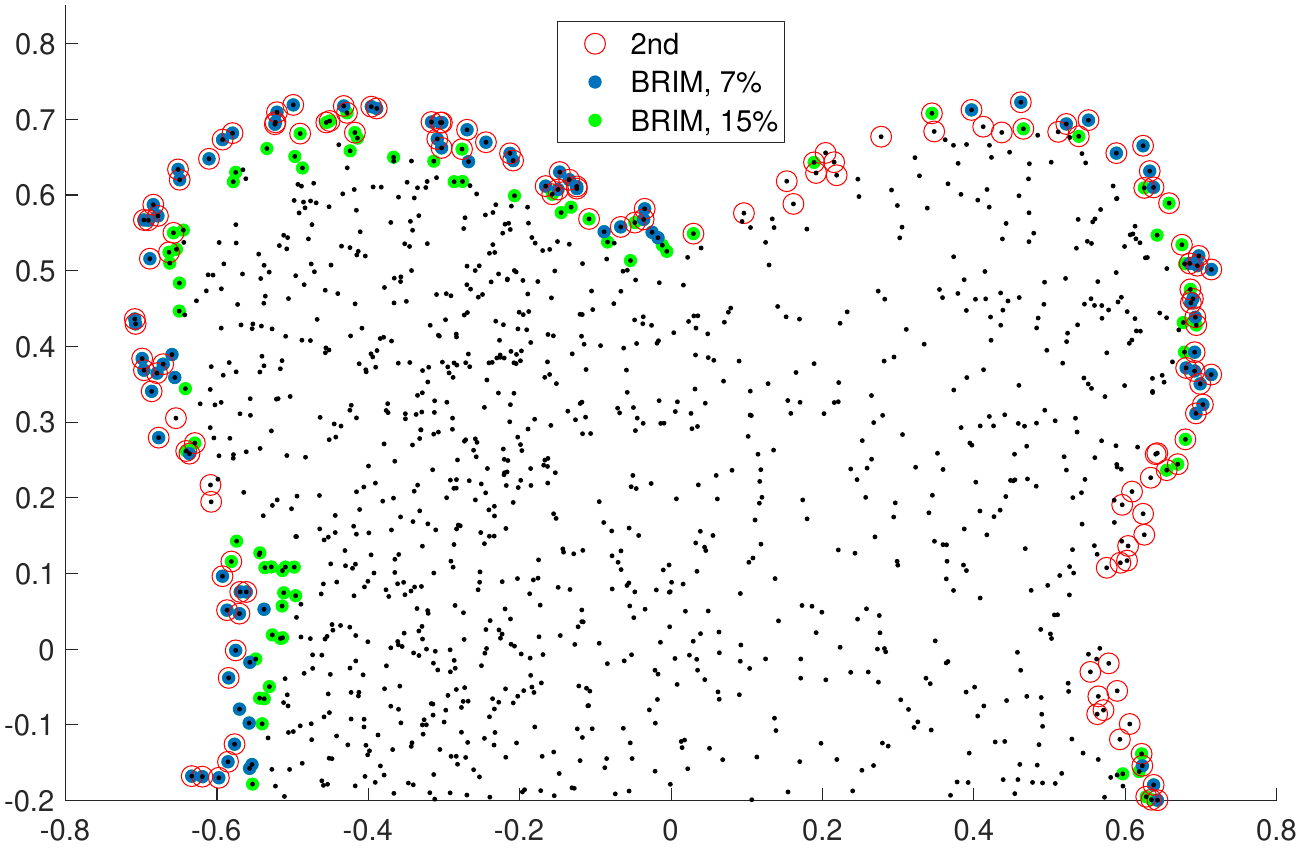}
    
    \includegraphics[width=0.49\textwidth]{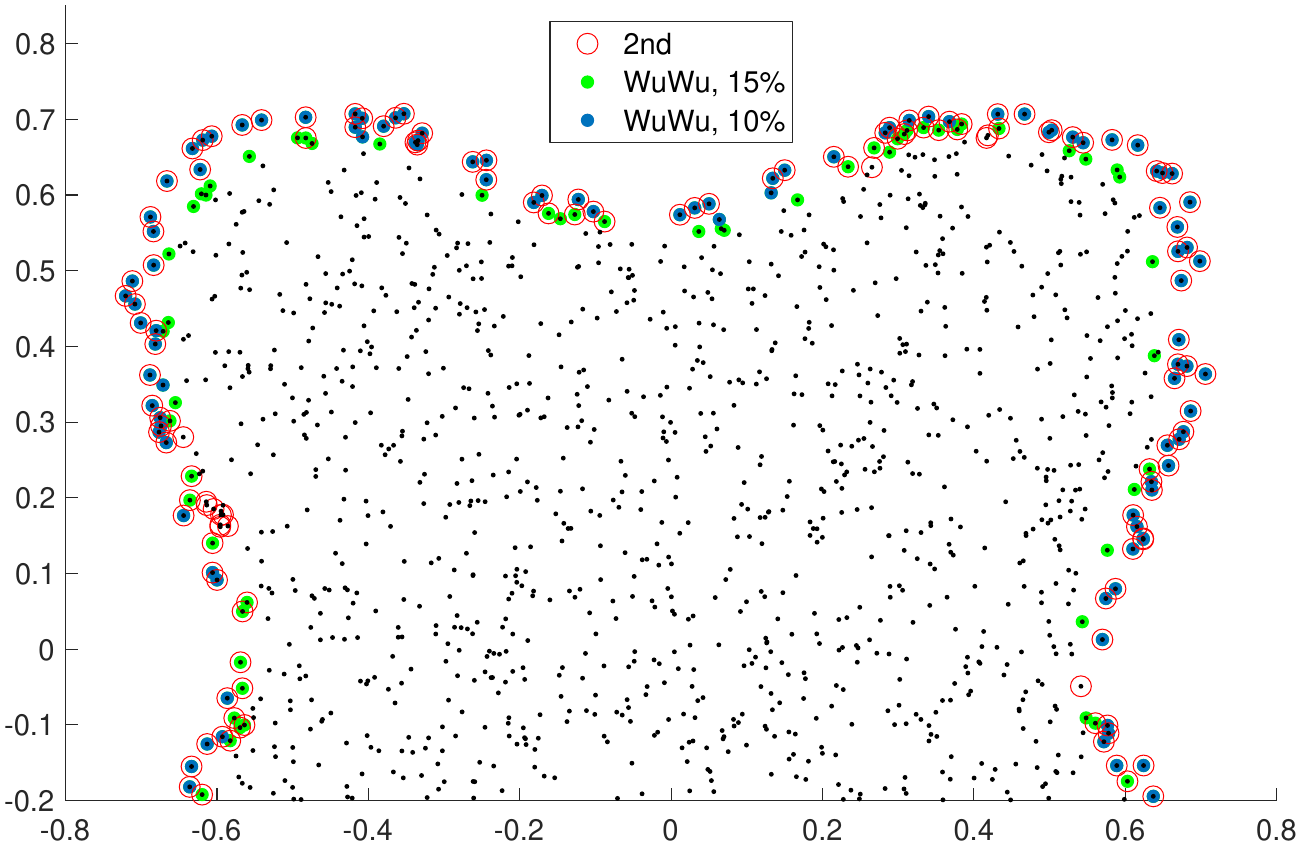}
    \includegraphics[width=0.49\textwidth]{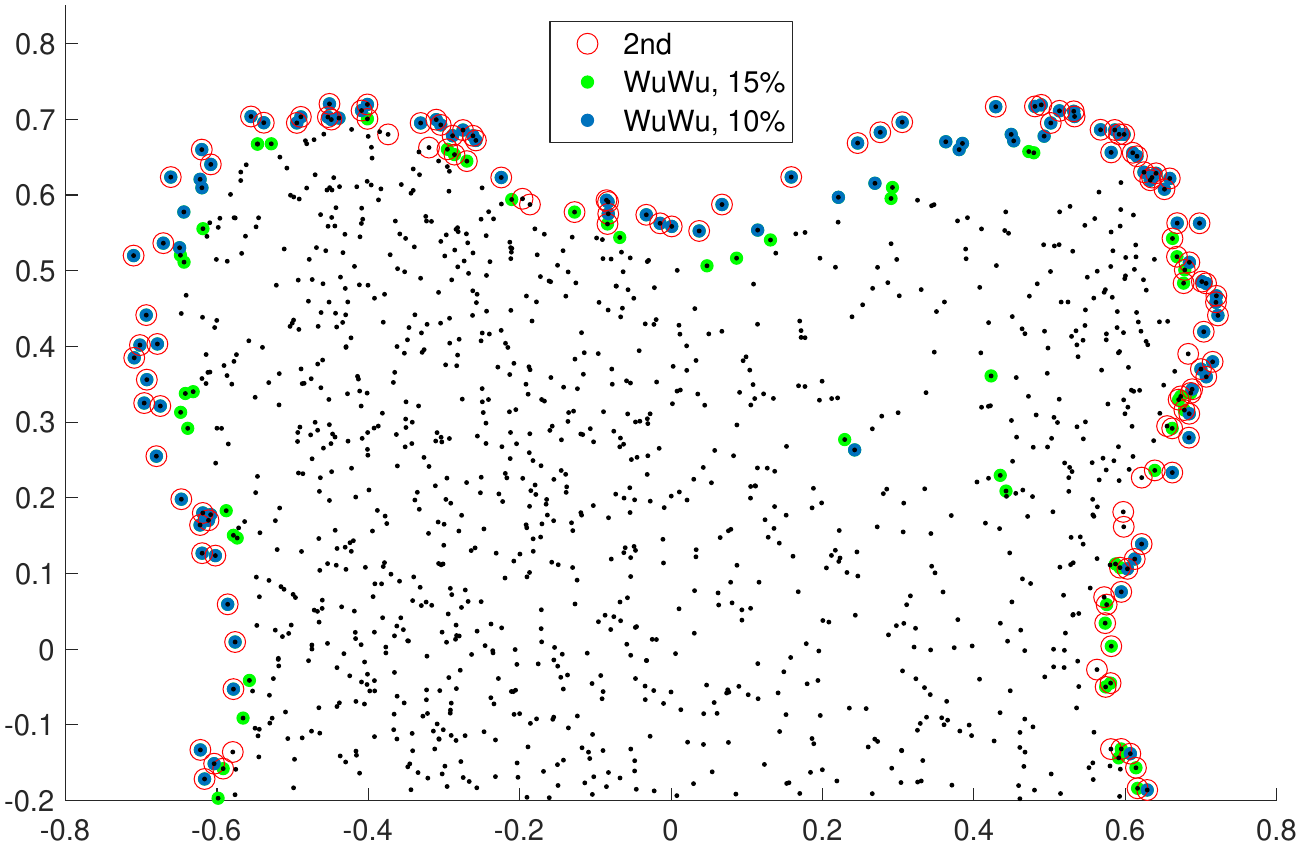}
    \vspace*{-5pt}
    \caption{ Comparison of tests for $n=2000$ points drawn out of density $\rho$ defined in \eqref{eq:rho_L}, with (top, and bottom right) $L=3$, and (bottom left) $L=1$. The second-order test with $\eps=0.03, \rr=0.18$ is compared with (top left) the Devroye-Wise estimator with radius $\alpha^{-1}$, (top right) BRIM, 
    and (bottom) WuWu. For BRIM and WuWu, the colored points are in the indicated top percentile according to the test statistic.
    }\label{fig:2nd-comparison}
    \end{figure}
    
    \subsection{Comparison with other approaches} \label{sec:comparison}
    
    We limit our comparisons with other border detection algorithms to a couple of visual illustrations and remarks. The reason for this is that other algorithms were not designed to identify a boundary layer of desired width, $\epsilon$, that our algorithm is designed for. Furthermore in most cases there  is no straightforward way to adapt other algorithms to do detect a boundary layer of fixed width. 
    
    We compare our 2nd order boundary test with, tests based on the Devroye-Wise estimator \eqref{def:DevWise} (DW), BRIM \cite{BRIM07}, 
    and the statistic of Wu and Wu (WuWu) \cite{WuWu19}. Recall that the Devroye-Wise estimator $\Omega_n$ approximates $\supp\rho$, and by boundary points we mean the points which contribute to the boundary $\partial\Omega_n$ -- i.e. $x^i\in\X$ such that $\overline{B(x^i,\eps)}\cap\partial\Omega_n\neq \emptyset$.
    We note that these are also exactly the data points that lie on the boundary estimator of Casal \cite{Casal07}.
    As discussed in Section \ref{sec:related-works}, such points are precisely the boundary points of the $\alpha$-shape \cite{EdelKirkSeidel83, Edels10}, a generalization of convex hull, with $\alpha=1/\eps$. In dimensions $d=2,3$, efficient algorithms for $\alpha$-shapes exist, and we used the built-in function in MATLAB \cite{aShp_MATLAB} to compute the contributing boundary points. 
    For BRIM and WuWu, we implemented in MATLAB the algorithms described in \cite{BRIM07} and \cite{WuWu19} respectively.
    
    In Figure \ref{fig:2nd-comparison} we see that the Devroye-Wise estimator via $\alpha$-shape effectively finds a thin boundary when a suitable $\alpha$ is used. The choice of appropriate $\alpha$ depends heavily on the density of the set of points considered. 
    Smaller $\alpha^{-1}$ identifies more points, and in particular allows recognizing those where boundary has negative curvature. On the other hand, choosing $\alpha^{-1}$ too small increases the risk of falsely identifying interior points, lying in an area of low density, as boundary points. Indeed, the top plot of Figure \ref{fig:2nd-comparison} exhibits such a trade-off: the test with $\alpha^{-1}=0.1$ misses boundary points around the concave indents, while choosing $\alpha^{-1}=0.05$ results in false positives deep inside the interior. In the context of solving PDEs on graphs, such false positives can be catastrophic.
    As pointed out in Section \ref{sec:related-works}, computing  $\alpha$-shapes becomes expensive when $d>3$. We tested a commonly used alpha shapes package in Python \cite{alphashapetoolbox} on a high performance computer with a 4.5GHz CPU, and found that the computational complexity in dimension for $n=1000$ points independently and uniformly distributed on the unit ball in dimensions $d=2$ up to $d=9$ followed very closely to the exponential complexity $O(n^{0.23 d})$. In terms of raw computational times, the alpha shape for $n=1000$ points in dimension $d=9$ took $110$ minutes, and $d=10$ and $d=11$ would have taken roughly $12$ and $77$ hours, respectively. The memory requirements seem to grow very quickly as well, with $d=8$ taking 13 GB and $d=9$ requiring roughly 45 GB.

    In contrast, BRIM easily  generalizes to dimensions higher than 3. BRIM uses a similar basic idea as our approach: it  approximate the inward normal direction. It does so by identifying the point $x^i\in B(x^0,\rr)$ maximizing $|B(x^i,\rr)\cap\X|$. To detect the boundary it compares the number of points in the normal direction  and those opposite of it. The test is sensitive to variations in the density. Indeed the bottom plot of Figure \ref{fig:2nd-comparison} shows that BRIM identifies significantly more points on the left boundary, near which the density is high, than it does on the  sparsely populated right. 
    
    WuWu also generalizes well to arbitrary dimension. Furthermore, it takes into account the curvature of the boundary by using spectral information of the `sample covariance matrix' (see Section \ref{sec:related-works}). We can see in Figure \ref{fig:2nd-comparison} that  WuWu consistently detects points near negatively curved parts of the boundary. However, it is not as robust under fluctuations in density. Observe WuWu classifies considerably more points on the left side of the boundary, where points are densely distributed, compared to the right. Further, some interior points are in the top 15\% according to the test statistic; this can be resolved by increasing $k$ for kNN, but at the cost of successfully identifying fewer points close to the boundary.
    
    We also ran experiments using the test statistic suggested by Aaron and Cholaquidis \cite{AarCho20}, but it did not perform well, as their statistic is designed to decide whether the manifold has a boundary or not, rather than to identify boundary points. 
    
    We stress again that all the other algorithms we compared were not designed for the task considered. We note that our method is as fast 
    as any of the other methods and provides the best quality boundary for the task considered. Furthermore there is no error analysis that would suggest that any of the other methods are second-order accurate.

    
     \begin{figure}[htb]
     \centering
     \includegraphics[trim={220pt 20pt 220pt 40pt}, clip=true, width=0.45\textwidth]{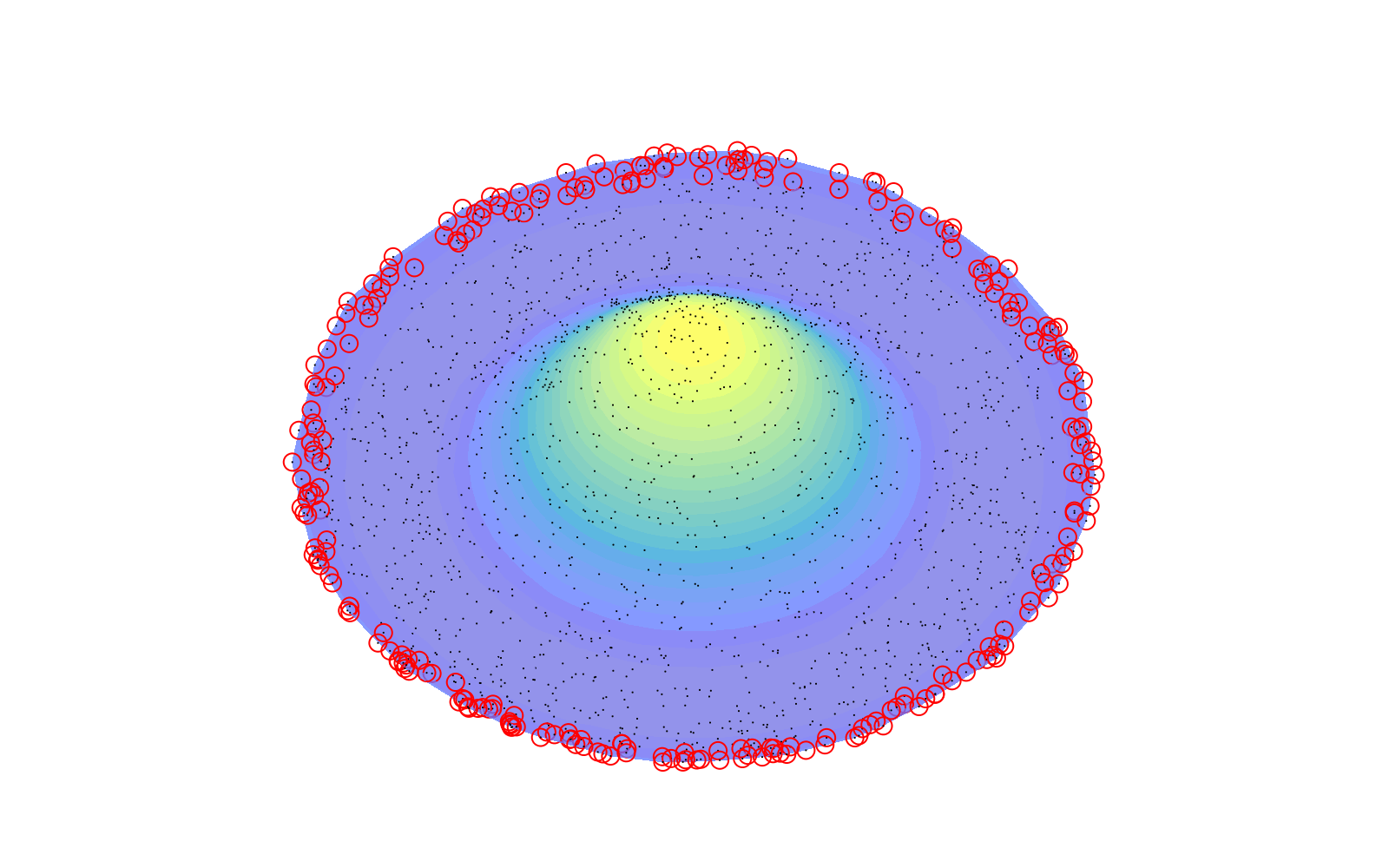}
     \includegraphics[trim={330pt 150pt 330pt 160pt}, clip=true,width=0.45\textwidth]{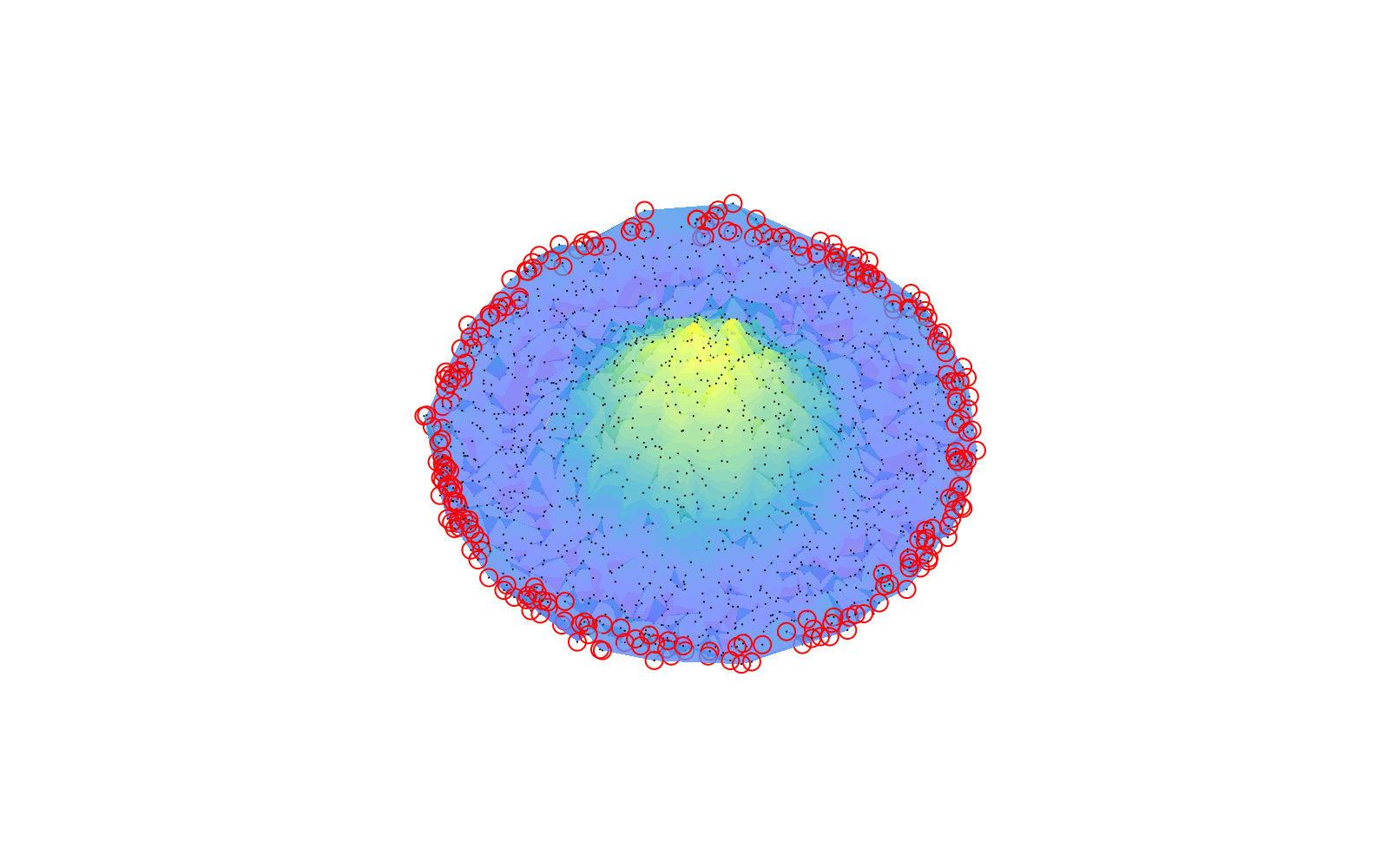}
     
     \vspace*{-15pt}
     
     \includegraphics[trim={180pt 30pt 180pt 0pt}, clip=true, width=0.45\textwidth]{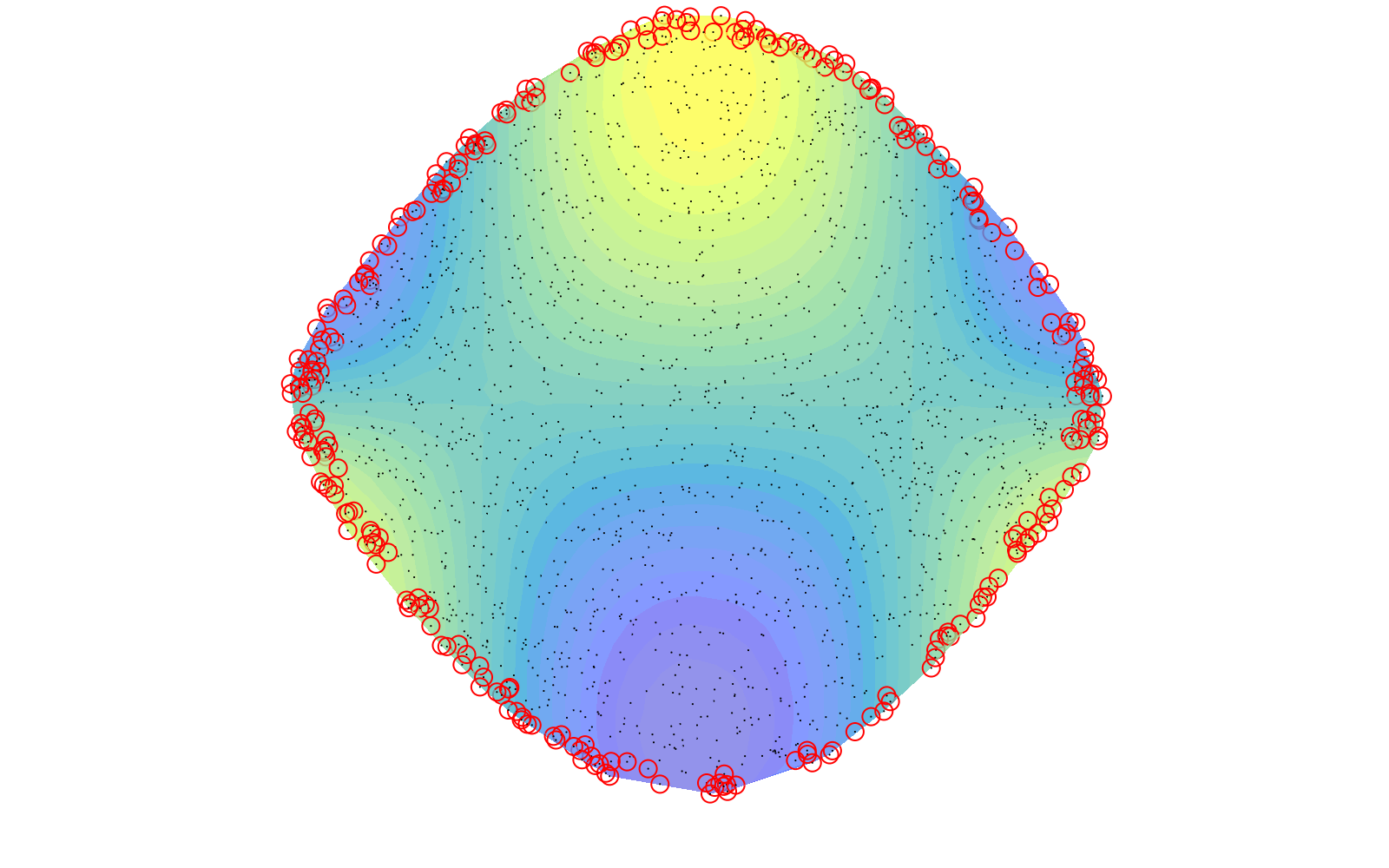}
     \includegraphics[trim={250pt 100pt 250pt 70pt}, clip=true, width=0.45\textwidth]{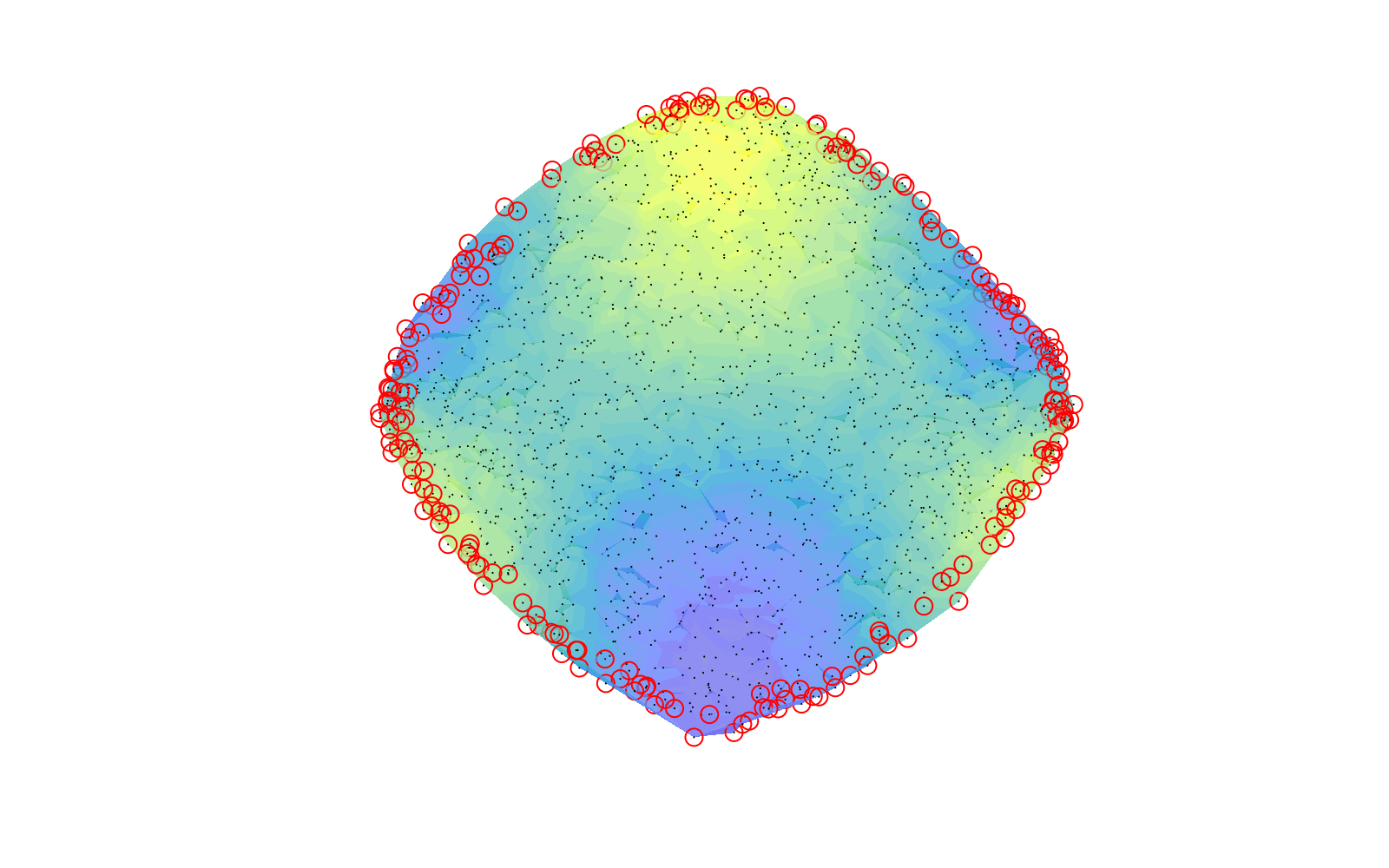}
     
     \vspace*{-12pt}
    
     \caption{Boundary points of point clouds supported on 2-dimensional surfaces, identified using Algorithm \ref{alg:2nd_manifold}. $n=2000,\,\rr=0.21,\,\eps=0.05$. Point clouds are marked in black, and the boundary points are circled in red. (Left) No additive noise. (Right) Additive Gaussian noise with standard deviation set as 1\% of the diameter of the surface. Surfaces appear irregular as they are reconstructed from noisy samples.}
     \label{fig:manifold_illustration}
     \end{figure}

\newpage

\section{Solving PDEs on data clouds}
\label{sec:PDE_graph}

One immediate application of boundary detection is the ability to solve PDEs on point clouds with flexibility in the choice of boundary condition. All of the present approaches to solving PDEs on data clouds, where the boundary is not known in advance,  rely on a variational description of the problem and thus result in natural variational boundary conditions. For the graph Laplacian this always yields homogeneous Neumann boundary conditions (see\cite{CST20} for discussion of the graph Laplacian near the boundary). 
In this section, we show how we can use our boundary detection method, which includes an estimation of the normal vector to the boundary, to solve PDEs on point clouds with various boundary conditions, including Dirichlet, Neumann, oblique, and Robin problems. We then give applications to computing data-depth and medians on real datasets, and present intriguing numerical experiments on MNIST and FashionMNIST.

Throughout this section, we fix some additional notation. For $\eps>0$ we define
\[\partial_\eps\Omega = \{x\in \Omega \, : \, \dist(x,\partial\Omega) \leq  \eps\}\]
and set $\Omega_\eps = \Omega\setminus \partial_\eps\Omega$. We recall that $\X=\{x^1,\dots,x^n\}$ is our point cloud, which is assumed to consist of independent and identically distributed random variables with density $\rho:\Omega\to \R$. We will place various assumptions on $\rho$ throughout the section. We also assume we have an accurate estimation of the points from $\X$ that fall in the boundary tube $\partial_\eps \Omega$. This is provided by our main results on boundary detection in Theorem \ref{thm:main} and Corollary \ref{corol:boundary_test}. In order to make the results in this section as general as possible, we simply assume that we have computed a boundary set $\partial_\eps \X\subset \X$ that satisfies 
\begin{equation}\label{eq:bdy}
\X_\eps\subset \Omega_\eps \ \ \text{ and }\ \ \partial_\eps \X \subset \partial_{2\eps}\Omega,
\end{equation}
where $\X_\eps = \X \setminus \partial_\eps \X$. 

\subsection{The eikonal equation}

First, we consider extending Theorem \ref{thm:main} to estimate the distance function
\begin{equation}\label{eq:dist}
d_\Omega(x):= \dist(x,\partial\Omega)
\end{equation}
on the whole point cloud $\X$. We can do this by solving the graph eikonal equation
\begin{equation}\label{eq:eikonal}
\left.\begin{aligned}
\min_{y\in B_0(x^i,\eps)\cap \X}\left\{ u_\eps(y) - u_\eps(x^i) + |y-x^i| \right\} &= 0,&&\text{if } x^i \in \X_\eps\\
u_\eps(x^i) &=0,&&\text{if }x^i \in \partial_\eps \X,
\end{aligned}\right\}
\end{equation}
where we write $B_0(x,\eps) := B(x,\eps)\setminus \{x\}$ for the punctured ball. The solution $u_\eps$ of the graph eikonal equation \eqref{eq:eikonal} is exactly the distance function on the graph with vertices $\X$ and edge weights $w_{ij}=|x^i-x^j|$ if $|x^i-x^j|\leq \eps$, and $w_{ij}=\infty$ otherwise. When this graph is connected, the solution of \eqref{eq:eikonal} is unique. The solution of \eqref{eq:eikonal} can be computed with Dijkstra's algorithm in $O(nk\log(n))$ time, where $k$ is an upper bound for the number of points in $B(x^i,\eps)\cap \X$ over all $i$. We expect the solution $u_\eps$ converges to the distance function $d_\Omega$ as $\eps\to 0$. Indeed this section is focused on proving this convergence with a quantitative $O(\eps)$ error rate.

For \eqref{eq:eikonal} to be well-defined, we require the set $B_0(x^i,\eps)\cap \X$ to be nonempty for all $x^i \in \X_\eps$.
\begin{proposition}\label{prop:empty}
Let $n\geq 2$. The event that $B_0(x^i,\eps)\cap \X$ is nonempty for all $x^i\in \X_\eps$ has probability at least $1-n\exp\left( -\frac12\omega_d \rho_{min} n\eps^d \right)$.
\end{proposition}
\begin{proof}
By the \emph{i.i.d.}~law, the probability that $B_0(x^i,\eps)\cap \X$ is empty conditioned on $x^i\in \X_\eps$ is
\[\left( 1-\int_{B(x^i,\eps)}\rho(x)\, dx \right)^{n-1}\leq \left( 1 - \rho_{min}\omega_d\eps^d \right)^{n-1} \leq \exp\left( -\rho_{min}\omega_d (n-1)\eps^d \right).\]
The proof is completed by union bounding over $\X$, and using that $n-1\geq \tfrac12 n$ for $n\geq 2$.
\end{proof}

We briefly review some basic properties of the distance function. We recall a function $u:\Omega\to \R$ is \emph{semiconcave} with constant $C$ if $u - C|x|^2$ is concave.  The distance function $d_\Omega$ is $1$-Lipschitz and semiconcave with constant $1/R$ (see, e.g., \cite{cannarsa2004semiconcave}).  By the Alexandrov theorem, a semiconcave function is twice differentiable almost everywhere in $\Omega$.  The distance function also satisfies the \emph{dynamic programming principle}
\[d_{\Omega}(x) = \min_{y\in B(x,\eps)}\left\{ d_{\Omega}(y) + |y-x| \right\}\]
for all balls $B(x,\eps)\subset \Omega$. This can be rearranged into the form
\begin{equation}\label{eq:cont_dpp}
\min_{y\in B(x,\eps)}\left\{ d_{\Omega}(y) - d_{\Omega}(x) + |y-x| \right\}=0.
\end{equation}
Thus, the graph eikonal equation \eqref{eq:eikonal} is merely a discretization of the dynamic programming principle \eqref{eq:cont_dpp} to the point cloud $\X$. At any point $x\in \Omega$ where $d_{\Omega}$ is differentiable, we can Taylor expand $d_{\Omega}$ in \eqref{eq:cont_dpp} and compute the minimum explicitly to find that $|\nabla d_{\Omega}(x)|=1$.  If $\Omega$ is bounded, the distance function $d_\Omega$ always has points of nondifferentiability (for example at its maximum).

The equation $|\nabla u|=1$ is referred to as the \emph{eikonal} equation (more generally $|\nabla u|=f$). The distance function $d_\Omega$ can be interpreted as the unique \emph{viscosity solution} of the eikonal equation. The viscosity solution is a type of weak solution to a partial differential equation (PDE) that allows non-differentiable functions to be solutions of first and second-order PDEs. In the case of the eikonal equation, and other first-order convex Hamilton-Jacobi equations, the viscosity solution coincides with the unique Lipschitz and \emph{semiconcave} function that satisfies the PDE almost everywhere. We use the semiconcave interpretation here and do not discuss viscosity solutions directly. We refer the reader to \cite{calder2020Viscosity,bardi2008optimal}  for more details on viscosity solutions.

We now turn to convergence of the solution of the graph eikonal equation \eqref{eq:eikonal} to the distance function $d_\Omega$.  For this, we require a notion of asymptotic consistency.
\begin{lemma}\label{lem:eikonal_consistency}
Let $0 < t \leq \frac{1}{d}$. The event that
\begin{equation}\label{eq:dbound}
\min_{x\in B_0(x^i,\eps)\cap \X } \left\{ \lambda d_{\Omega}(x) - \lambda d_{\Omega}(x^i) + |x-x^i|\right\}  \leq  t\lambda \epsilon  + \frac{4\lambda\epsilon^2}{R} - (\lambda-1) \epsilon
\end{equation}
holds for all $\lambda\geq 1$ and $x^i\in \X\cap \Omega_\eps$ has probability at least $1-n\exp\left( -\frac{\omega_{d-1}}{4(d+1)}\rho_{min}n\eps^d (2t)^{\frac{d+1}{2}} \right)$.
\end{lemma}
The proof of Lemma \ref{lem:eikonal_consistency} requires some well-known properties of the distance function, which we summarize in the following Proposition, whose proof is postponed to the appendix.
\begin{proposition}\label{prop:dist}
Let $\eps>0$ and  $x^0\in \Omega_\eps$. Let $x_*\in B(x^0,\eps)$ such that
\begin{equation}\label{eq:min}
d_\Omega(x_*) = \min_{B(x^0,\eps)}d_\Omega.
\end{equation}
Then $x_*\in \partial B(x^0,\eps)$, $d_\Omega(x_*) = d_\Omega(x^0) - \eps$,  and for all $x\in \Omega$ we have
\begin{equation}\label{eq:ineq}
d_\Omega(x) - d_\Omega(x_*) \leq p\cdot (x-x_*) + \frac{1}{R}|x-x_*|^2, \ \ \text{ where }p = \frac{x^0-x_*}{\eps}.
\end{equation}
\end{proposition}

\begin{proof}[Proof of Lemma \ref{lem:eikonal_consistency}]
Let $\lambda\geq 1$ and let $x_*^i\in B(x^i,\eps)$ such that $d_{\Omega}(x_*^i)=\min_{B(x^i,\eps)} d_{\Omega}$. For $x^i\in \X\cap \Omega_\eps$ we can apply Proposition \ref{prop:dist} to obtain
\begin{align*} 
\lambda d_{\Omega}(x) - \lambda d_{\Omega}(x^i) + |x-x^i|&= \lambda d_{\Omega}(x) - \lambda d_{\Omega}(x_*^i) - \lambda \epsilon + |x-x^i|\\
&\leq \lambda \,p\cdot(x-x_*^i) + \frac{\lambda}{R}|x-x_*^i|^2 - \lambda \epsilon +  |x-x^i|
\end{align*}
for any $x\in B(x^i,\eps)$, where $p=(x^i-x_*^i)/\eps$. Since $|x-x_*^i| \leq 2\epsilon$ and $|x-x^i|\leq \epsilon$ we obtain
\begin{equation}\label{eq:bound1}
\lambda d_{\Omega}(x) - \lambda d_{\Omega}(x^i) + |x-x^i| \leq \lambda\, p\cdot(x-x_*^i) + \frac{4\lambda\epsilon^2}{R} - (\lambda-1) \epsilon. 
\end{equation}
For $0 \leq t\leq 1$ define the set
\[A^i_t = \left\{ x\in B(x^i,\eps)\, : \, p\cdot(x-x_*^i) \leq t\,\epsilon\right\}.\]
If \eqref{eq:dbound} fails to hold, then it follows from \eqref{eq:bound1} that the set $\X\cap A^i_t$ is empty.  The remainder of the proof is focused on estimating the volume $|A_t^i|$ in order to control the probability that $\X \cap A^i_t$ is empty.

The measure of $A^i_t$ is unchanged by taking $x^i=0$, $x_*^i=\eps e_d$, and $p=-e_d$, which gives
\[|A^i_t| =\left|B(0,\eps)\cap \{x_d \geq (1- t)\epsilon\}\right| = \eps^d \left|B(0,1) \cap \{x_d \geq 1 - t\}\right|.\]
We lower bound the volume of the spherical cap by integrating
\begin{align*}
\left|B(0,1)\cap \{x_d \geq 1-t\}\right|&= \int_{1-t}^{1}\omega_{d-1}(1-x_d^2)^{\frac{d-1}{2}}\, dx_d \\
&\geq \int_{1-t}^{1}\omega_{d-1}(1-x_d^2)^{\frac{d-1}{2}} x_d\, dx_d \\
&= \frac{\omega_{d-1}(2t)^{\frac{d+1}{2}}}{d+1}\left(1 - \tfrac{t}{2}\right)^{\frac{d+1}{2}}.
\end{align*}
Now, since $t\mapsto \left(1 - \tfrac{t}{2}\right)^{\frac{d+1}{2}}$ is convex we have
\[\left(1 - \tfrac{t}{2}\right)^{\frac{d+1}{2}} \geq 1 - \left( \tfrac{d+1}{4}\right)t \geq \frac{1}{2},\]
provided $t \leq \frac{2}{d+1}$, which is satisfied when $t \leq \frac{1}{d}$. This yields
\[|A_t^i| \geq \frac{\omega_{d-1}\eps^d(2t)^{\frac{d+1}{2}}}{2(d+1)}=:\Lambda.\]
Hence, the event that $\X\cap A^i_t$ is empty has probability bounded by
\[(1-\rho_{min}\Lambda)^{n-1}\leq \exp\left( -\rho_{min}(n-1)\Lambda \right) \leq \exp\left( -\frac12 \rho_{min}n\Lambda \right),\]
since $n\geq 2$ so $n-1 \geq \frac12 n$. The proof is completed by union bounding over $\X$.
\end{proof}

We now prove convergence of $u_\eps$ to the distance function $d_{\Omega}$ as $\eps\to 0$ and $n\to \infty$.
\begin{theorem}\label{thm:eikonal}
Assume $\epsilon\leq \frac{R}{8}$ and \eqref{eq:bdy} holds. Let $u_\eps$ solve \eqref{eq:eikonal} and let $0 < t \leq \min\{\frac{1}{d},\frac{1}{2}-\frac{4\epsilon}{R}\}$. Then 
\begin{equation}\label{eq:eikonal_rate}
-2\eps \leq u_\eps - d_\Omega\leq 2d_\Omega\left(t + \frac{4\eps}{R}\right) \ \ \text{on } \X
\end{equation}
holds with probability at least $1-2n\exp\left( -\frac{\omega_{d-1}}{4(d+1)}\rho_{min}n\eps^d (2t)^{\frac{d+1}{2}} \right)$.
\end{theorem}
\begin{proof}
The proof is split into three steps.

1. Let $0< t \leq \frac{1}{d}$ and assume the results of Lemma \ref{lem:eikonal_consistency} hold. Let $\lambda \geq 1$ and let $x^i\in \X$ such that $u_\eps - \lambda d_{\Omega}$ attains its maximum over $\X$ at $x^i$. Then we have that
\[u_{\eps}(x^j) - u_\eps(x^i) \leq \lambda d_{\Omega}(x^j) - \lambda d_{\Omega}(x^i)\]
for all $j$. If $x^i\in \X_\eps$, then since $u_\eps$ satisfies \eqref{eq:eikonal} we have
\[0=\min_{y\in B_0(x^i,\eps)\cap \X}\left\{ u_\eps(y) - u_\eps(x^i) + |y-x^i| \right\} \leq\min_{y\in B_0(x^i,\eps)\cap \X}\left\{ \lambda d_{\Omega}(y) - \lambda d_\Omega(x^i) + |y-x^i| \right\}.\]
By \eqref{eq:bdy} we have $x^i\in \Omega_\eps$, which allows us to apply Lemma \ref{lem:eikonal_consistency} to obtain that
\[ 0 \leq  t\lambda \epsilon  + \frac{4\lambda\epsilon^2}{R} - (\lambda-1) \epsilon.\] 
This cannot hold when when $\lambda > \left( 1 -t-\tfrac{4\epsilon}{R} \right)^{-1}$ and $t + \frac{4\epsilon}{R} < 1$. For any such $\lambda$ we must have $x^i\in \partial_\eps \X$ and so
\[\max_{\X}(u_\eps - \lambda d_\Omega)= \max_{\partial_\eps \X} (u_\eps-\lambda d_\Omega) \leq 0.\]
It follows that $u_\eps - d_\Omega \leq (\lambda-1)d_\Omega$ on $\X$.
Sending $\lambda \to \left( 1 -t-\tfrac{4\epsilon}{R} \right)^{-1}$ we obtain
\[u_\eps - d_\Omega \leq d_\Omega\left[\left( 1 -t-\frac{4\epsilon}{R} \right)^{-1} - 1\right] \ \ \text{on } \X.\]
The proof of this direction is completed by using the inequality
\[(1-x)^{-1} - 1 \leq 2x \ \ \text{for } 0 \leq x \leq \tfrac{1}{2}\]
and imposing the additional restriction that $t + \frac{4\epsilon}{R} \leq \frac{1}{2}$ to simplify the right hand side.

2. For the other direction, let $0 < \lambda < 1$. Since $d_\Omega$ is $1$-Lipschitz we have
\begin{equation}\label{eq:super}
\min_{y\in B_0(x^i,\eps)\cap \X}\left\{ \lambda d_\Omega(y) - \lambda d_\Omega(x^i) + |y-x^i| \right\}\geq (1-\lambda)\min_{y\in B_0(x^i,\eps)\cap \X}\left\{ |y-x^i| \right\} >0,
\end{equation}
provided $B_0(x^i,\eps)\cap \X$ is not empty. Thus, by \eqref{eq:bdy} and Proposition \ref{prop:empty}, \eqref{eq:super} holds for all $x^i \in \X_\eps$ with probability at least $1-n\exp\left( -\frac12\omega_d \rho_{min} n\eps^d \right)$. Let $x^i\in \X$ such that $u_\eps - \lambda d_{\Omega}$ attains its minimum over $\X$ at $x^i$. By an argument similar to the first part of the proof, \eqref{eq:eikonal} and \eqref{eq:super} imply that $x^i\in \partial_\eps \X$. Therefore $u_\eps(x^i)=0$ and by \eqref{eq:bdy} we have $x^i \in \partial_{2\eps}\Omega$. It follows that
\[\min_{x\in \X}(u_\eps(x) - \lambda d_\Omega(x)) = -\lambda d_{\Omega}(x^i) \geq -2\lambda \eps.\]
Sending $\lambda\to 1^-$ completes the proof.

3. Union bounding over the events in steps 1 and 2 above, the results of the theorem hold with probability at least
\[1-n\exp\left( -\frac{\omega_{d-1}}{4(d+1)}\rho_{min}n\eps^d (2t)^{\frac{d+1}{2}} \right) -n\exp\left( -\frac12\omega_d \rho_{min} n\eps^d \right).\]
The first exponential is larger, provided 
\[\frac{\omega_{d-1}}{2(d+1)}(2t)^{\frac{d+1}{2}} \leq \omega_d.\]
Recalling $\omega_{d-1}/\omega_d \leq \sqrt{d}$, this is true when $2t\leq 1$, which is implied by the assumption that $t \leq \frac{1}{d}$ and $d\geq 2$. Therefore, \eqref{eq:eikonal_rate} holds with probability at least $1-2n\exp\left( -\frac{\omega_{d-1}}{4(d+1)}\rho_{min}n\eps^d (2t)^{\frac{d+1}{2}} \right)$.
\end{proof}

\begin{remark}
We now provide an interpretation of the result of Theorem \ref{thm:eikonal}. 
To obtain the conditions under which the error rate is linear in $\eps$ we take $t = \eps$ and obtain
\[-2\eps \leq u_\eps - d_\Omega \leq 2d_\Omega\left( 1 + \frac{4}{R} \right)\eps\]
holds with probability at least $1-2n^{-2}$ provided that the length scale $\eps$ satisfies:
\begin{equation}\label{eq:epsscale}
\eps \geq \left( \frac{6(d+1)\log(n)}{2^{\frac{d+1}{2}}\omega_{d-1}\rho_{min}n} \right)^{\frac{2}{3d+1}}.
\end{equation}
Taking the smallest allowable $\eps$ above, we obtain that $u_\eps$ converges to the distance function $d_\Omega$ at a convergence rate of $\O(n^{-2/(3d+1)})$, up to logarithmic factors. We mention that we have numerically seen convergence rates closer to $\O(\eps^2)$ for $\eps$ much larger than the lower bound in \eqref{eq:epsscale}. This may indicate that, in practice, a sharper convergence rate, as a function of $n$, could be obtained by choosing larger value for $\eps$.

To obtain a sufficient condition for uniform convergence alone we need conditions under which we can take $t_n \to 0$ as $n \to \infty$ and $\eps_n \to0$ for the estimate in Theorem \ref{thm:eikonal} to hold with high probability.
We see that this is possible whenever 
\begin{equation}\label{eq:conn}
\lim_{n\to \infty} \frac{n\eps_n^d}{\log(n)} = \infty.
\end{equation}
Then by the Borel-Cantelli lemma we have that $u_{\eps_n}\to d_\Omega$ uniformly on $\X$ as $n\to \infty$ with probability one.  
\label{rem:length}
\end{remark}

\subsubsection{Numerical results}\label{sssec:1storderPDE}

\begin{figure}[!t]
\centering
\begin{minipage}{0.45\textwidth}
\includegraphics[width=0.45\textwidth]{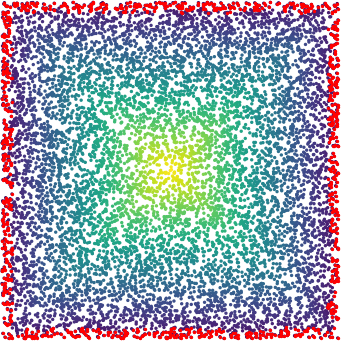} \hspace{3mm}
\includegraphics[width=0.45\textwidth]{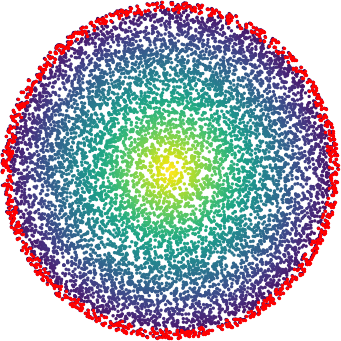}
\vspace{2mm}

\includegraphics[width=0.47\textwidth]{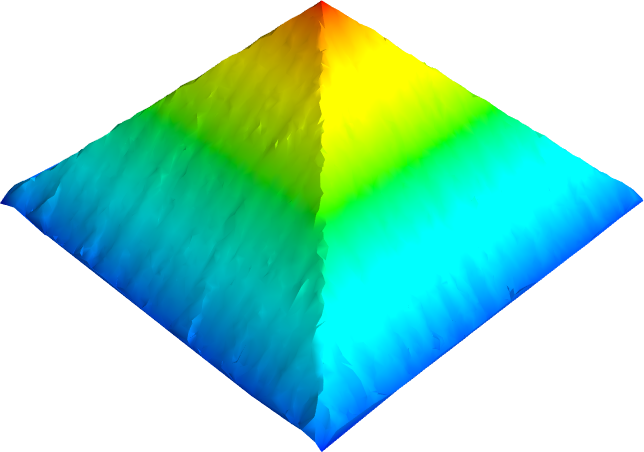} \hspace{3mm}
\includegraphics[width=0.4\textwidth]{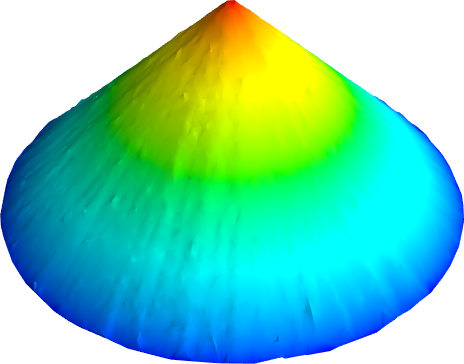}
\end{minipage}
\begin{minipage}{0.54\textwidth}
\subfloat{\includegraphics[clip=true,trim=12 15 13 12, width=\textwidth]{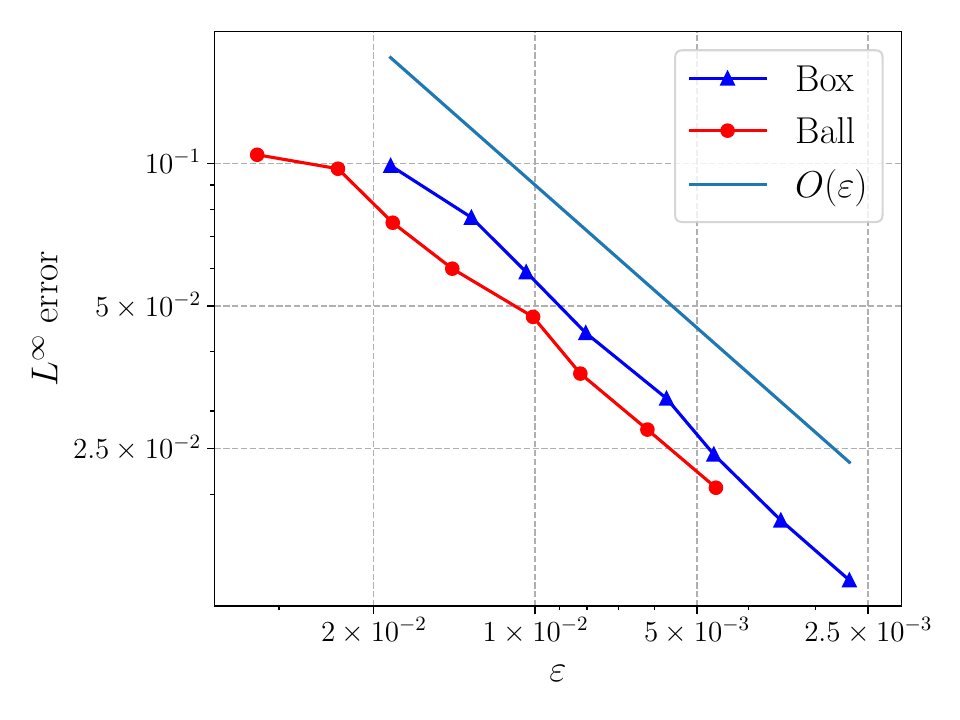}}
\end{minipage}
\caption{Plots of the solution to the graph eikonal equation \eqref{eq:eikonal} for $n=10^4$ for both the box and ball domains, and error plots for varying $\eps$ averaged over $100$ trials. The red points indicate the detected boundary points used in solving \eqref{eq:eikonal}. We see convergence rates better than the linear $O(\eps)$ rate guaranteed by Theorem \ref{thm:eikonal}.  }
\label{fig:eikonal}
\end{figure}

We tested the $O(\eps)$ convergence rate from Theorem \ref{thm:eikonal} on a box $\Omega=[0,1]^2$ and ball $\Omega=B(0,1)$ domain. We used $n=2^{10}$ up to $n=2^{17}=131,072$ \emph{i.i.d.}~random variables uniformly distributed on the domain, and chose $\eps$ adaptively  based on the distance to the $k^{\rm th}$ nearest neighbor, where $k=10 n^{\frac{1}{5}}$. This is equivalent to the scaling $\eps\sim n^{-\frac25}$, since $k \sim n\eps^2$. We detected the boundary by thresholding $\hat{d}_r(x)$ at $\frac{3\eps}{2}$, where $r$ is the distance from $x$ to its $k^{\rm th}$ nearest neighbor, and $\eps$ satisfies $36\pi \rho n\eps^2=k$. In Figure \ref{fig:eikonal} we show the solution of \eqref{eq:eikonal} for $n=10^4$ as both a colored point cloud, and visualized as a surface, computed by constructing a triangulated mesh over the point cloud. In the plot in Figure \ref{fig:eikonal} we show the $L^\infty$ error $|u_\eps-d_\Omega|$ versus $\eps$ averaged over $100$ trials. Both domains track very closely to the theoretical $O(\eps)$ convergence rates.

\subsection{Second-order equations}
\label{sec:2ndOrder}

We now turn to second-order equations on point clouds with general boundary conditions.  In particular, we show how our estimation $\hat{\nu}_\eps$ of the inward unit normal vector $\nu$ can be used to set general boundary conditions involving normal derivatives. We recall that Theorem \ref{thm:normal_vector} shows that $\hat{\nu}_\eps$ is an $O(\eps)$ approximation of $\nu$ with high probability. In order to state the results in the most general setting, we simply assume there exists a constant $C_\nu$ such that
\begin{equation}\label{eq:nu}
|\hat{\nu}_\eps(x^i) - \nu(x^i)| \leq C_\nu \eps
\end{equation}
for all $x^i \in \X\cap \partial_{2\eps}\Omega$. We recall that Theorem \ref{thm:normal_vector} shows that the bound \eqref{eq:nu} holds with high probability as long as $\epsilon \geq C ( \log n/n)^{1/(d+2)}$. This lower bound on $\epsilon$ is also required for all the results in this section to hold with high probability. Indeed, Theorems \ref{thm:PointwiseConsistency} and \ref{thm:2rate} both require $n\epsilon^{d+2} \geq C\log n$ for a sufficently large constant $C$, which amounts to the same lower bound on $\epsilon$ up to constants.

The graph PDEs we solve will involve the graph Laplacian $\L_\eps$, which is defined by
\begin{equation}\label{eq:GL}
\L_\eps u(x^i) = \frac{2}{\sigma_\eta n\eps^{d+2}}\sum_{j=1}^n \eta\left( \frac{|x^i-x^j|}{\eps} \right)(u(x^j)-u(x^i)),
\end{equation}
where $\sigma_\eta = \int_{\R^d}\eta(|z|)z_1^2\, dz$, and $\eta$ is smooth, compactly supported on $[0,1]$, and satisfies $\int_{\R^d}\eta(|z|)\, dz =1$. We define the normal derivative $\nabla_\nu u(x)=\nabla u(x)\cdot \nu$ and the approximate normal derivative $\widehat{\nabla}_\nu$ by
\begin{equation}\label{eq:normalderGL}
\widehat{\nabla}_\nu u(x^i) = \frac{u(p_n(x^i+\eps\hat{\nu}_\eps(x^i))) - u(x^i)}{\eps},
\end{equation}
where $p_n:\Omega \to \X$ is the closest point map. We consider the following graph Poisson equation with Robin-type boundary conditions
\begin{equation}\label{eq:robinGL}
\left.\begin{aligned}
\L_\eps u(x^i) &= f(x^i),&&\text{if } x^i \in \X_\eps\\ 
\gamma u(x^i) - (1-\gamma)\widehat{\nabla}_{\nu}u(x^i) &=g(x^i),&&\text{if } x^i\in \partial_\eps \X.
\end{aligned}\right\}
\end{equation}
Here, $\gamma\in (0,1]$ and $f$ and $g$ are given smooth functions. In this section, we show that the solution of \eqref{eq:robinGL} converges as $n\to \infty$ and $\eps\to 0$ to the solution of the Robin problem 
\begin{equation}\label{eq:robin}
\left\{\begin{aligned}
-\rho^{-1}\div(\rho^2 \nabla u) &= f,&&\text{in }\Omega\\ 
\gamma u -(1-\gamma)\nabla_\nu u &=g,&&\text{on }\partial\Omega.
\end{aligned}\right.
\end{equation}

\begin{remark}\label{rem:extension}
We note that in order to solve the graph PDE \eqref{eq:robinGL} given a \emph{nonconstant} boundary condition $g:\partial \Omega \to \R$, we need a way to define an extension $g_\epsilon :\partial_{2\epsilon}\Omega \to \R$ that is uniformly close to $g$ within the boundary tube $\partial_{2\epsilon}\Omega$. One way to do this is to define the closest point extension $g_\epsilon(x)=g(x_*)$ where $x_* = \text{argmin}_{y\in \partial\Omega}|x-y|$. The closest point $x_*$ is unique for $x\in \partial_{2\epsilon}\Omega$ when $2\epsilon<R$ and if $g$ is Lipschitz then $|g_\epsilon(x)-g(x_*)|\leq C\epsilon$ for $x\in \partial_{2\epsilon} \Omega$.  It is important to note, however, that the closest point extension requires knowledge of the boundary $\partial\Omega$. In applications where the boundary $\partial\Omega$ is not known \emph{a priori}, and is instead estimated from the point cloud, such as in data depth in machine learning, we can only handle constant boundary conditions (i.e., $g=0$ on $\partial\Omega$ for data depth).
\end{remark}

Throughout this section we assume $\partial\Omega$ and $\rho$ are smooth. By elliptic regularity, the solution $u$ of \eqref{eq:robin} is smooth. The constants in this section will be denoted by $C,C_1,C_2,\dots>0$, and may depend on $\gamma,u,d,f,g,\rho,\Omega$ and $\partial\Omega$, and can change from line to line.

The proof of convergence is based on a maximum principle for \eqref{eq:robinGL}. 
\begin{lemma}\label{lem:maxprinc}
If $u$ satisfies
\begin{equation}\label{eq:robinSUB}
\left.\begin{aligned}
-\L_\eps u(x^i) &< 0,&&\text{if } x^i\in X_\eps\\ 
\gamma u(x^i) - (1-\gamma)\widehat{\nabla}_{\nu}u(x^i) &\leq 0,&&\text{if } x^i \in \partial_\eps \X
\end{aligned}\right\}
\end{equation}
then $u \leq 0$ on $\X$.
\end{lemma}
\begin{proof}
Let us write $w_{ij}=\eta\left( \frac{|x^i-x^j|}{\eps} \right)$ and $d_i = \sum_{j=1}^n w_{ij}$. Then by \eqref{eq:robinSUB} we have
\[d_i u(x^i) - \sum_{j=1}^n w_{ij}u(x^j) =\sum_{j=1}^n w_{ij}(u(x^i)-u(x^j)) < 0\]
for all $x^i\in X_\eps$. It follows that $d_i > 0$, and so $u(x^i) < \frac{1}{d_i}\sum_{j=1}^n w_{ij}u(x^j)$. Therefore, $u$ attains its maximum over $\X$ at some $x^i \in \partial_\eps \X$, and so
\[\gamma u(x^i) \leq (1-\gamma)\frac{u(p_n(x^i+\eps\hat{\nu}_\eps(x^i)))-u(x^i)}{\eps} \leq 0.\]
Since $\gamma>0$ we have $u(x^i)\leq 0$.
\end{proof}

The convergence proof also requires pointwise consistency for the graph Laplacian. We refer to \cite[Remark 5.26]{calder2020Calculus} for the following result.
\begin{theorem}\label{thm:PointwiseConsistency}
Let $u\in C^4(\Omega)$, $\eps>0$ and $0<\lambda\leq \eps^{-1}$. Then 
\begin{equation}\label{eq:PC}
\max_{x^i\in \Omega_\eps \cap \X}\left|\L_{\eps}u(x^i) - \rho(x^i)^{-1}\div(\rho^2 \nabla u)\vert_{x_i}\right| \leq C_1\|u\|_{C^4(\Omega)}(\eps^2+\lambda)
\end{equation}
holds with probability at least $1-2 n \exp\left( -C_2 n\eps^{d+2}\lambda^2 \right)$.
\end{theorem}

We now establish our main convergence result in this section. 
\begin{theorem}\label{thm:2rate}
Assume \eqref{eq:bdy} and \eqref{eq:nu}. Let $\eps>0$ and assume $C_\nu \eps \leq 1$. Let $u$ be the solution of \eqref{eq:robin} with $\gamma>0$, and let $u_\eps$ satisfy \eqref{eq:robinGL}.  Then for any $0 < \lambda\leq \eps^{-1}$ and $t>0$, the event that
\begin{equation}\label{eq:2rate}
|u(x^i) - u_\eps(x^i)| \leq C\left(\|\gamma u - (1-\gamma)\nabla_\nu u-g\|_{L^\infty(\partial_{2\eps}\Omega)} + (1-\gamma)(t+C_\nu\eps+\eps)+ \eps^2 + \lambda\right) 
\end{equation}
holds for all $x^i \in \X$ has probability at least $1-n\exp\left( -\frac{1}{6}\omega_d\rho_{min}n\eps^dt^d\right) - 2n\exp\left( -Cn\eps^{d+2}\lambda^2\right)$.
\end{theorem}
\begin{proof}
The proof is split into three steps.

1. Note that $x^i + \eps \nu \in \Omega_\eps$. By \eqref{eq:nu} we have
\[|x^i + \eps\hat{\nu}_\eps(x^i) - (x^i + \eps\nu)| = \eps|\hat{\nu}_\eps - \nu| \leq C_\nu \eps^2.\]
Since $C_\nu \eps \leq 1$ we have $x^i + \eps\hat{\nu}_\eps(x^i) \in \Omega$. Therefore, we can compute
\begin{align*}
\widehat{\nabla}_{\nu} u(x^i) &= \frac{u(p_n(x^i + \eps\hat{\nu}_\eps(x^i))) - u(x^i)}{\eps}\\
&=\frac{u(x^i + \eps\nu(x^i)) - u(x^i)}{\eps} + \O\left( \eps^{-1}|p_n(x^i+\eps\hat{\nu}_\eps(x^i))-(x^i+\eps\hat{\nu}_\eps(x^i))| + C_\nu\eps \right)\\
&=\nabla_\nu u(x^i) + \O\left( \eps^{-1}|p_n(x^i+\eps\hat{\nu}_\eps(x^i))-(x^i+\eps\hat{\nu}_\eps(x^i))| + C_\nu\eps + \eps \right).
\end{align*}
Let $t\geq 0$. If $|p_n(x^i+\eps\hat{\nu}_\eps(x^i))-(x^i+\eps\hat{\nu}_\eps(x^i))|  \geq t\eps$ then the set $B(x^i+\eps\hat{\nu}_\eps(x^i),t\eps)\cap \X$ is empty, which by Lemma \ref{lem:p_bounds} has probability less than $1-\exp\left( -\frac{1}{3}\omega_d\rho_{min}(n-1)\eps^dt^d \right)$. Union bounding over $x^i$ and using that $n-1 \geq \frac{1}{2}n$ for $n\geq 2$, we find that
\[\widehat{\nabla}_{\nu} u(x^i) = \nabla_\nu u(x^i) + \O\left( t + C_\nu \eps + \eps \right)\]
holds for all $x^i \in \partial_\eps \X\subset \partial_{2\eps}\Omega$ with probability at least $1-n\exp\left( -\frac{1}{6}\omega_d\rho_{min}n\eps^dt^d\right)$. A similar computation can be made for $\varphi$, and so we find that
\begin{equation}\label{eq:uBC}
|\widehat{\nabla}_{\nu} \phi(x^i)  - \nabla_\nu \varphi(x^i)|,|\widehat{\nabla}_{\nu} u(x^i)  - \nabla_\nu u(x^i)| \leq C(t + C_\nu\eps + \eps) 
\end{equation}
for all $x^i \in \partial_\eps \X$.

2. Let $0 < \lambda \leq \eps^{-1}$. Let $\phi$ be the solution of
\begin{equation}\label{eq:barrier}
\left.\begin{aligned}
-\rho^{-1}\div(\rho^2\nabla \phi)  &=1 &&\text{in }\Omega\\ 
\gamma \phi -(1-\gamma)\nabla_\nu \phi &=1&&\text{on }\partial\Omega.
\end{aligned}\right\}
\end{equation}
By assumption, $u,\varphi\in C^4(\bar{\Omega})$, and so by Theorem \ref{thm:PointwiseConsistency}, with probability at least $1-2 n \exp\left( -C n\eps^{d+2}\lambda^2 \right)$ we have
\begin{equation}\label{eq:uPC}
|\L_\eps \phi(x^i) - 1\:|\:,|\L_\eps u(x^i) - f(x^i)| \leq C(\eps^2 + \lambda)
\end{equation}
whenever $\dist(x^i,\partial\Omega)\geq \eps$. 

3. Let us now define
\[w(x^i) = u(x^i) - u_\eps(x^i) - K\varphi(x^i),\]
for $K$  to be determined. Then by \eqref{eq:uPC} and \eqref{eq:uBC} we have
\[\L_\eps w(x^i) \leq -K + C(\eps^2 + \lambda)\]
for $x^i \in X_\eps$ and
\[\gamma w(x^i) - (1-\gamma) \widehat{\nabla}_{\nu} w(x^i) \leq -K + \|\gamma u - (1-\gamma)\nabla_\nu u -g\|_{L^\infty(\partial_{2\eps}\Omega)} +C(1-\gamma)(t + C_\nu \eps + \eps)\]
for $x^i \in \partial_\eps \X$. For any choice of $K$ satisfying
\[K > C\left(\|\gamma u - (1-\gamma)\nabla_\nu u-g\|_{L^\infty(\partial_{2\eps}\Omega)} + (1-\gamma)(t+C_\nu\eps+\eps)+ \eps^2 + \lambda\right) \]
we can apply Lemma \ref{lem:maxprinc} to find that $w \leq 0$, and so $u - u_\eps \leq C K\|\varphi\|_{L^\infty(\Omega)}$. The other direction of the proof is similar.
\end{proof}

\begin{remark}
The proof of Theorem \ref{thm:2rate} relies on the maximum principle (Lemma \ref{lem:maxprinc}), which requires $\gamma>0$. Thus, the result does not apply to the pure Neumann case $\gamma=0$. This case would require special attention to ensure the compatibility condition 
\[\int_{\Omega}f \, dx= \int_{\partial \Omega} g\, dS\]
holds at both the continuum and discrete level. 
\label{rem:neumann}
\end{remark}

\begin{remark}
Consider the Dirichlet problem in Theorem \ref{thm:2rate} by setting $\gamma=1$. If we set $\lambda=\eps^2$, then we obtain the rate
\[|u - u_\eps| \leq C(\|u - g\|_{L^\infty(\partial_{\eps}\Omega)} + \eps^2)\]
with probability at least $1-2n\exp\left( -Cn\eps^{d+6} \right)$. If we are able to extend the boundary conditions $g$ to $\Omega$ so that $\|u - g\|_{L^\infty(\partial_{\eps}\Omega)}\leq C\eps^2$, then we obtain a second-order $\O(\eps^2)$ convergence rate in Theorem \ref{thm:2rate}.
\label{rem:secondorder}
\end{remark}

\begin{figure}[!t]
\centering
\subfloat{\hspace{-1mm}
\includegraphics[width=0.14\textwidth]{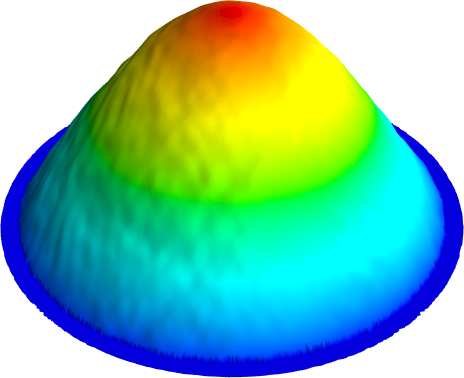} \hspace{-1.5mm}
\includegraphics[width=0.14\textwidth]{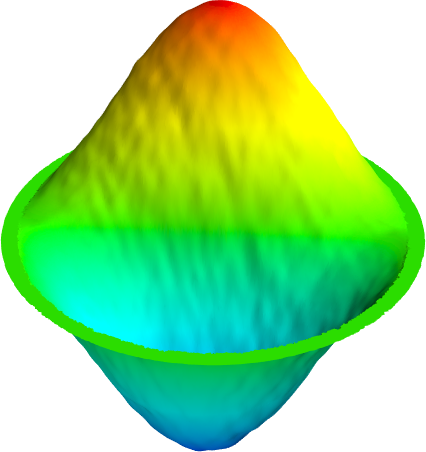} \hspace{-1.5mm}
\includegraphics[width=0.14\textwidth]{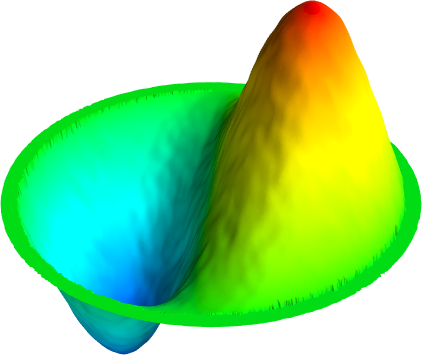} \hspace{-1.5mm}
\includegraphics[width=0.14\textwidth]{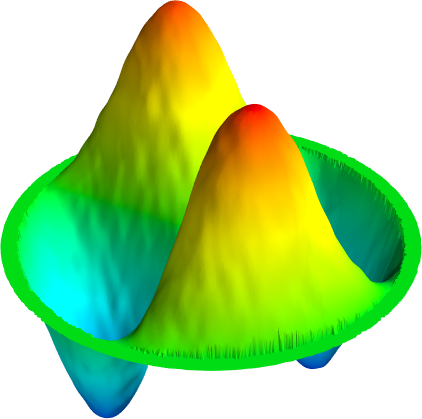} \hspace{-1.5mm}
\includegraphics[width=0.14\textwidth]{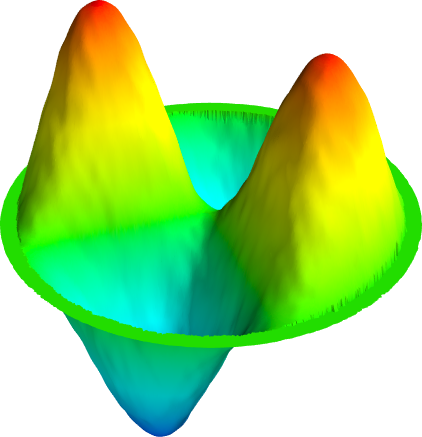} \hspace{-1.5mm}
\includegraphics[width=0.145\textwidth]{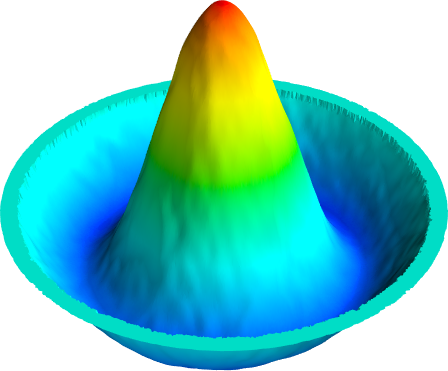}\hspace{-1mm}
\includegraphics[width=0.14\textwidth]{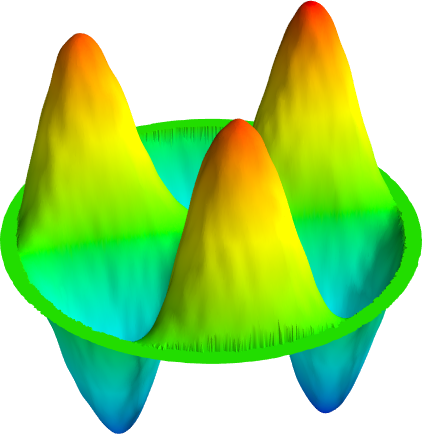}}
\caption{First 7 Laplacian Dirichlet eigenfunctions on the disk computed via approximation with graph Laplacian eigenvectors with $n=10^5$ points.}
\label{fig:eigenfunctions}
\end{figure}

\begin{remark}
Finally, we remark that our boundary detection method allows us to consider Dirichlet eigenfunctions of the Laplacian on the point cloud $\X$ by solving the eigenfunction problem
\begin{equation}\label{eq:eigenGL}
\left.\begin{aligned}
\L_\eps u(x^i) &= \lambda u(x^i),&&\text{if } x^i \in X_\eps\\ 
 u(x^i)&=0,&&\text{if } x^i\in \partial_\eps \X
\end{aligned}\right\}
\end{equation}
The Dirichlet eigenfunctions of $\L_\eps$ would naturally converge to continuum Dirichlet eigenfunction for the weighted Laplacian $-\rho^{-1}\div(\rho^2 \nabla u)$. The proof of this is expected to be more involved than Theorem \ref{thm:2rate}, since we cannot use the maximum principle to obtain strong discrete stability results. We expect discrete to continuum convergence results to hold for the eigenvector problem \eqref{eq:eigenGL} using the combined variational and PDE methods from \cite{GGHS2020,calder2019improved,calder2020Lip}. We show in Figure \ref{fig:eigenfunctions} the first 7 Dirichlet eigenfunctions on the disk computed by solving \eqref{eq:eigenGL} over a graph constructed with $n=10^5$ random variables independent and uniformly distributed on the disk.
\label{rem:eigenfunction}
\end{remark}

\begin{remark}
In the case that $f=0$ and we consider Dirichlet boundary conditions ( $\gamma=1$), we can extend Theorem \ref{thm:2rate} to hold even when $\partial_\eps \X$ is replaced with a thinner boundary $\partial \X_\delta$ for any $\eps^2 \ll \delta \leq \eps$. That is when only the points in a very thin region near the true boundary are identified. 
In this case we can prove the error rate of $O(\eps^2/\delta)$. The proof is a minor adaptation of \cite[Theorem 2.4]{CST20}. We expect the proof would extend to the case of nonzero $f$ as well, though the incorporation of $\gamma<1$ seems more difficult.
\label{rem:delta}
\end{remark}

\subsubsection{Numerical results}\label{sssection:2ndorderPDE}

\begin{figure}[!t]
\centering
\begin{minipage}{0.45\textwidth}
\includegraphics[width=0.45\textwidth]{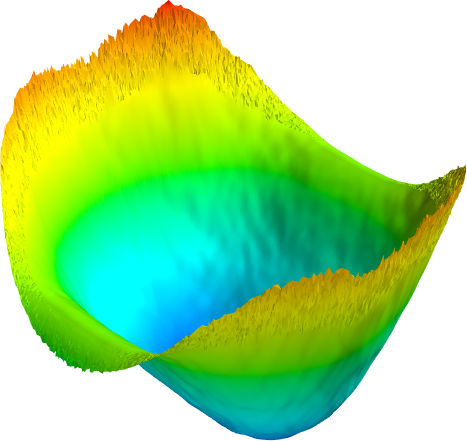} \hspace{3mm}
\includegraphics[width=0.45\textwidth]{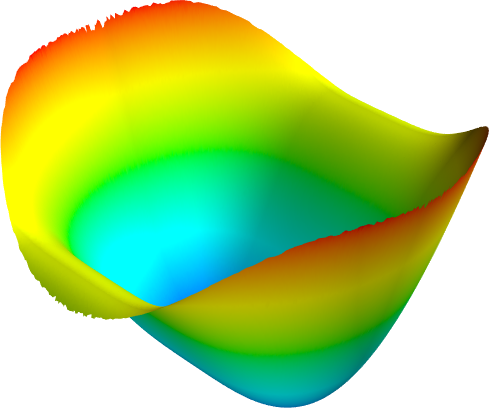}
\vspace{2mm}

\includegraphics[width=0.45\textwidth]{eigen0_graph_1e5.png} \hspace{3mm}
\includegraphics[width=0.45\textwidth]{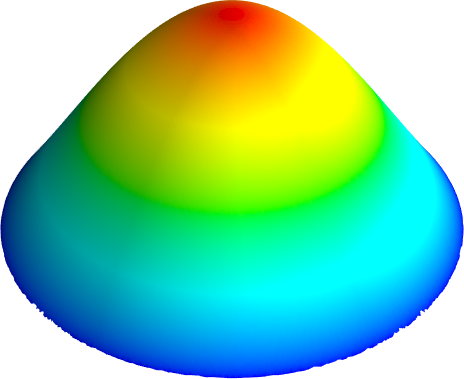}
\end{minipage}
\begin{minipage}{0.54\textwidth}
\subfloat{\includegraphics[clip=true,trim=12 12 10 10, width=\textwidth]{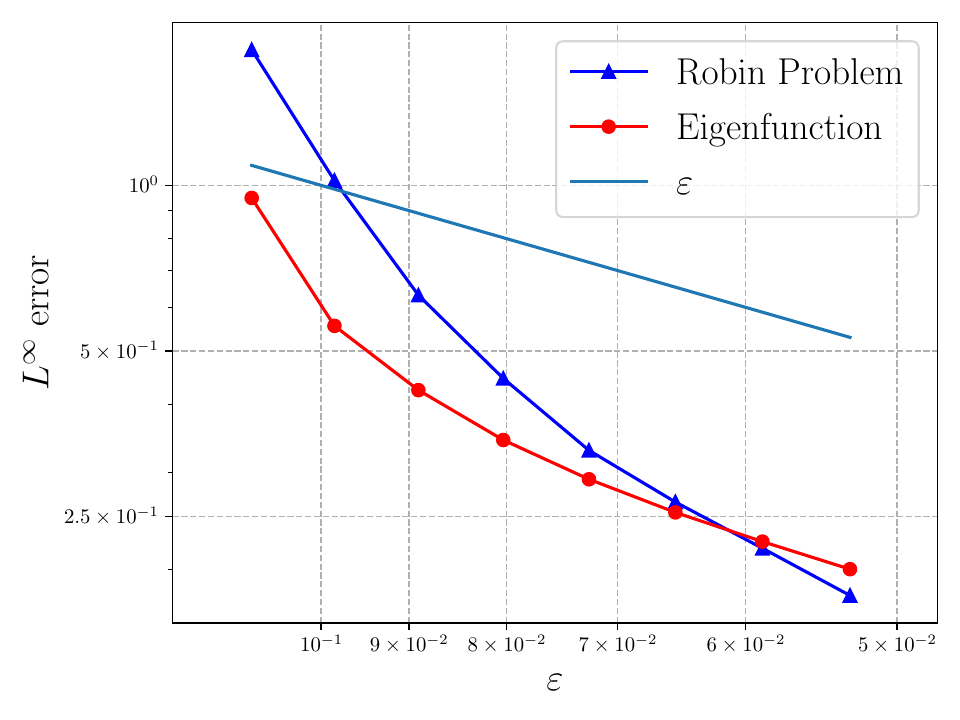}}
\end{minipage}
\caption{On the left, plots of the solution to the Robin problem and principal Dirichlet eigenvector for $n=10^5$ points on the disk, compared to the exact solutions of each problem. On the right we show an error plot for varying $\eps$ averaged over $100$ trials.  }
\label{fig:secondorder}
\end{figure}

We ran several numerical experiments to test the rate of convergence in Theorem \ref{thm:2rate} on the disk $\Omega=B(0,1)\subset \R^2$. In this case, $\rho=1/\pi$. In the first experiment, we set the solution of the Robin problem \eqref{eq:robin} with $\gamma=1/2$ to be
\[u(x) = \sin(2x_1^2) - \cos(2x_1^2)\]
and then set $f = -\frac{1}{\pi}\Delta u$ and $g = \tfrac12 (u - \nabla_\nu u)$, and tested how well the solution of the graph Laplace equation \eqref{eq:robinGL} can reconstruct $u$. In the second problem, we solved \eqref{eq:eigenGL} for the principal Dirichlet eigenfunction, and compared against the true solution $u(x) = J_0(\lambda |x|)$, where $J_0$ is the zeroth order Bessel function of the first kind, and $\lambda$ is the first positive root of $J_0$. In each case we varied the number $n$ of random variables in the point cloud from $n=2^{10}$ up to $n=2^{17}=131,072$ by powers of 2, and set
\[\eps = \frac{1}{4}\left( \frac{\log n}{n} \right)^{\frac{1}{d+4}},\]
where here, $d=2$. We approximated the $\eps$ boundary using $k=2\pi n\eps^2$ nearest neighbors. Figure \ref{fig:secondorder} shows plots of the solutions to each graph-based problem, compared to the true solutions of their corresponding PDEs, and a plot of maximum absolute error versus $\eps$, averaged over 100 trials. In both cases we see better convergence rates than the $O(\eps)$ guaranteed by Theorem \ref{thm:2rate}. Taking the last three data points on each plot, the empirical convergence rates are $\eps^{1.86}$ for the Robin problem and $\eps^{1.13}$ for the Dirichlet eigenfunction.

\subsection{Experiments with real data}
\label{sec:realdata}

We now turn to experiments with real data. We use the MNIST \cite{lecun1998gradient} and FashionMNIST \cite{xiao2017fashion} datasets. MNIST is a standard dataset for handwritten digit recognition, consisting of 70,000 images of handwritten digits $0$--$9$. Each image is a $28\times 28$ grayscale image, which we interpret as a vector in $\R^{784}$. The FashionMNIST dataset is a drop-in replacement for MNIST, with the same number of datapoints and image resolution, except that the 10 classes in FashionMNIST correspond to different items of clothing, with pictures taken from a fashion catalog. In all experiments, we use Euclidean distance between the raw pixel values in $\R^{784}$ to compare images.

We focus our experiments on detecting the boundary images for each class, and then using the discovered boundary to compute a notion of data depth by solving PDEs over the data with Dirichlet boundary conditions. In this way, we also compute a notion of data median, by taking the deepest images in the dataset. To compute the boundary points, we use $k=10$ Euclidean nearest neighbors and compute $\hat d_\eps(x^{i})$ for each image $x^i$ by taking $\eps$ as the Euclidean distance to the $k^{\rm th}$ nearest neighbor. We then set the images with scores $\hat d_\eps(x^{i})$ in the lower 10\% of all images to be boundary points. This is an implicit way to select the desired width of the boundary by instead specifying how many boundary points are desired. Figures \ref{fig:MNIST} and \ref{fig:FashionMNIST} show that top 10 boundary images in each class compared to randomly selected images.

%
%

\begin{figure}[!t]
\centering
\subfloat[Random digits]{\includegraphics[trim=0 0 0 0,   clip=true, width=0.49\textwidth]{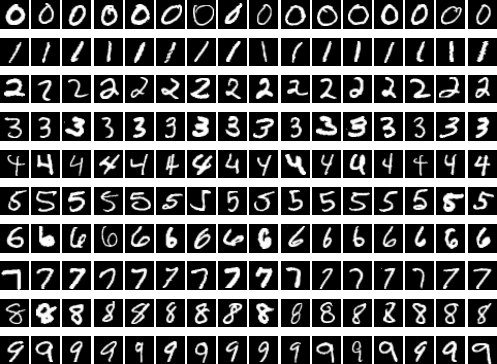}}\hfill
\subfloat[Boundary digits]{\includegraphics[trim=0 0 0 0, clip=true, width=0.49\textwidth]{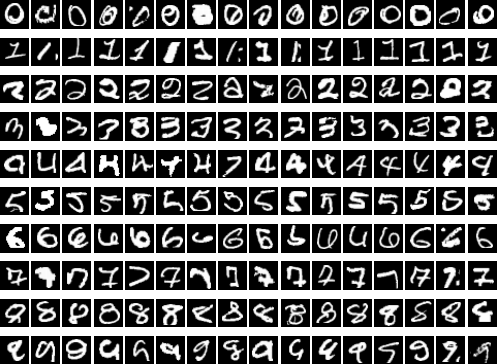}}\\
\subfloat[Eigen Median digits]{\includegraphics[trim=0 0 0 0,   clip=true, width=0.49\textwidth]{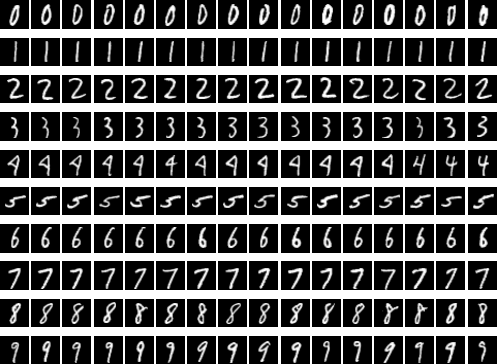}}\hfill
\subfloat[Eikonal Median digits]{\includegraphics[trim=0 0 0 0,   clip=true, width=0.49\textwidth]{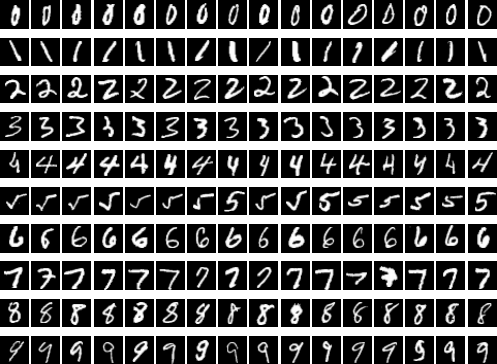}}
\caption{MNIST experiments.}
\label{fig:MNIST}
\end{figure}

\begin{figure}[!t]
\centering
\subfloat[Random images]{\includegraphics[trim=0 0 0 0,   clip=true, width=0.49\textwidth]{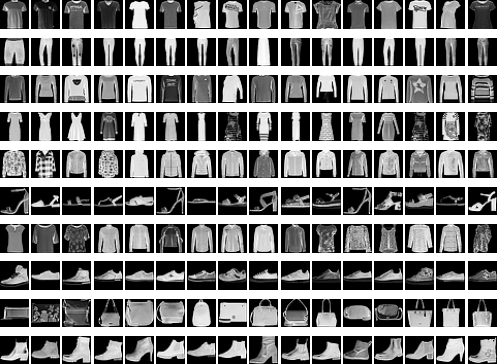}}\hfill
\subfloat[Boundary images]{\includegraphics[trim=0 0 0 0, clip=true, width=0.49\textwidth]{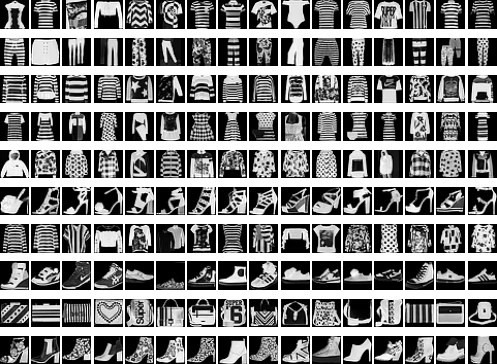}}\\
\subfloat[Eigen Median images]{\includegraphics[trim=0 0 0 0,   clip=true, width=0.49\textwidth]{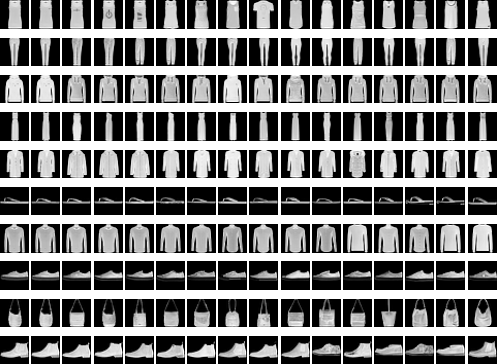}}\hfill
\subfloat[Eikonal Median images]{\includegraphics[trim=0 0 0 0,   clip=true, width=0.49\textwidth]{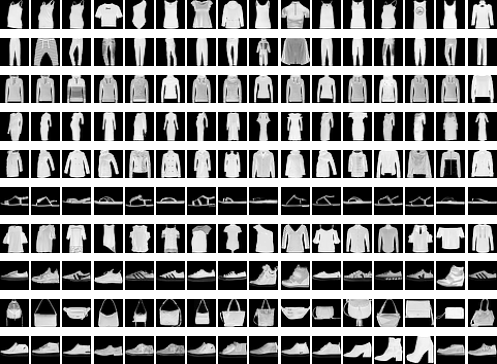}}
\caption{FashionMNIST experiments.}
\label{fig:FashionMNIST}
\end{figure}

Once the boundary points are detected, we construct a $k$ nearest neighbor graph over the data points in each class. We use Gaussian weights given by
\[w_{ij} =\exp\left( -\frac{4|x^i-x^j|^2}{\eps_k(x^i)^2} \right),\]
where $\eps_k(x_i)$ is the distance between $x^i$ and its $k^{\rm th}$ nearest neighbor. We used $k=10$ in all experiments, and the weight matrix was symmetrized by replacing $W$ with $W+W^T$. For a notion of data depth, we compute the principal Dirichlet eigenfunction of the graph Laplacian, i.e., the solution of \eqref{eq:eigenGL} with smallest $\lambda$. We found the symmetric normalization
\[\L u(x) = \sum_{j=1}^n w_{ij}\left( \frac{u(x_i)}{\sqrt{d_i}} - \frac{u(x_j)}{\sqrt{d_j}} \right), \ \ d_i = \sum_{j=1}^n w_{ij}\]
gives slightly more consistent results, and so we report the results with this normalization. The principal Dirichlet eigenfunction has one sign on all of $\X$, and we choose the version that is positive on $\X$. We use $u(x^i)$ as a notion of data depth, and the $x^i$ where $u(x^i)$ is largest can be interpreted as median images for each class.   The median images computed this way are shown in Figures \ref{fig:MNIST} (c) and \ref{fig:FashionMNIST} (c). We also computed the median by solving the eikonal equation \eqref{eq:eikonal}, again using our detected boundary images as Dirichlet boundary conditions. The eikonal median images are shown in Figures \ref{fig:MNIST} (d) and \ref{fig:FashionMNIST} (d). 

We observe that the eigen-median images are all very similar to each other, compared with the eikonal median images, which have much more variation. There is some work showing that the maximum or minimum points of graph Laplacian eigenvectors correspond to nodes in the graph that are unusually well-connected, in the sense that a random walker will take a long time to escape the region (see, e.g., \cite{depavia2020spectral}). These regions then contain groups of highly similar images. In contrast, the eikonal median images are simply those that are furthest from the boundary in the graph geodesic distance, and these images may be scattered around the graph and have far more variability.

We remark that we can also construct a similar notion of data depth by solving the Dirichlet problem \eqref{eq:robinGL} with $f\equiv 1$, $\gamma=1$, and $g\equiv 0$. The solution of this Poisson equation has the interpretation that $u(x^i)$ is the mean exit time for random walkers starting at $x^i$, and exiting at $\partial_\eps \X$.  We almost always obtained the same set of median images, up to some minor differences, using the two graph PDEs, so we only show the results using the Dirichlet eigenfunction. 

\begin{remark}
It is important to point out that our boundary detection method is designed for data sampled from a distribution with a Lebesgue density on a domain  $\Omega\subset \R^d$. That is, our results do not apply to the \emph{manifold assumption}, which is a commonly used modeling assumption in machine learning that assumes the data is sampled from a low dimensional smooth submanifold, possibly with boundary, embedded in $\R^d$. The dimension $m$ of the smooth submanifold is called the \emph{intrinsic dimension} of the data.  While the MNIST dataset has extrinsic dimension $d=784$ (i.e., the number of pixels in each image), it has been estimated that intrinsic dimension of each class of MNIST digits is between $m=12$ and $m=14$ \cite{hein2005intrinsic,costa2006determining}. In the manifold setting, it is possible that our approximation of the unit normal vector $\hat{\nu}_\eps$ will point in the direction normal to the data submanifold in regions of higher curvature.  This would cause interior points to be incorrectly identified as boundary points. This could be addressed by projecting $\hat{\nu}_\eps$ onto the tangent space to the submanifold, but we leave this for future work. Since we see good results for our method on MNIST and FashionMNIST in Figures \ref{fig:MNIST} (b) and \ref{fig:FashionMNIST} (b), this may indicate that curvature is low for both datasets and does not play a large role in boundary detection. 
\label{rem:manifold}
\end{remark}


\bibliographystyle{siam}
\bibliography{be_bib}

\begin{thebibliography}{10}

\bibitem{AAL21}
{\sc E.~Aamari, C.~Aaron, and C.~Levrard}, {\em Minimax boundary estimation and
  estimation with boundary}, arXiv preprint arXiv:2108.03135,  (2021).

\bibitem{AL19}
{\sc E.~Aamari and C.~Levrard}, {\em {Nonasymptotic rates for manifold, tangent
  space and curvature estimation}}, The Annals of Statistics, 47 (2019),
  pp.~177 -- 204.

\bibitem{AarCho20}
{\sc C.~Aaron and A.~Cholaquidis}, {\em On boundary detection}, Ann. Inst.
  Henri Poincar\'{e} Probab. Stat., 56 (2020), pp.~2028--2050.

\bibitem{depavia2020spectral}
{\sc S.~S. Adela~DePavia}, {\em Spectral clustering revisited: Information
  hidden in the {F}iedler vector}, Foundations of Data Science, 3 (2021),
  pp.~225--249.

\bibitem{bardi2008optimal}
{\sc M.~Bardi and I.~Capuzzo-Dolcetta}, {\em Optimal control and viscosity
  solutions of {H}amilton-{J}acobi-{B}ellman equations}, Springer Science \&
  Business Media, 2008.

\bibitem{barnett1976ordering}
{\sc V.~Barnett}, {\em The ordering of multivariate data}, Journal of the Royal
  Statistical Society: Series A (General), 139 (1976), pp.~318--344.

\bibitem{alphashapetoolbox}
{\sc K.~Bellock}, {\em Alpha shape toolbox}, 2021.
\newblock \url{https://github.com/bellockk/alphashape}. (accessed 2021/10/22).

\bibitem{Ben80}
{\sc J.~L. Bentley}, {\em Multidimensional divide-and-conquer}, Communications
  of the ACM, 23 (1980), pp.~214--229.

\bibitem{Vai89}
\leavevmode\vrule height 2pt depth -1.6pt width 23pt, {\em Multidimensional
  divide-and-conquer}, Discrete and Comp. Geom., 4 (1989), p.~101–115.

\bibitem{Github:annoy}
{\sc E.~Bernhardsson}, {\em Annoy: Approximate nearest neighbors in
  c++/python}, 2018.
\newblock \url{https://pypi.org/project/annoy/} (accessed 2020/10/19).

\bibitem{BerrySauer17}
{\sc T.~Berry and T.~Sauer}, {\em Density estimation on manifolds with
  boundary}, Computational Statistics \& Data Analysis, 107 (2017), pp.~1--17.

\bibitem{BD-medial}
{\sc L.~Birbrair and M.~P. Denkowski}, {\em Medial axis and singularities}, J.
  Geom. Anal., 27 (2017), pp.~2339--2380.

\bibitem{bou2021hamilton}
{\sc A.~Bou-Rabee and P.~S. Morfe}, {\em Hamilton-{J}acobi scaling limits of
  pareto peeling in 2d}, arXiv preprint arXiv:2110.06016,  (2021).

\bibitem{boucheron2013concentration}
{\sc S.~Boucheron, G.~Lugosi, and P.~Massart}, {\em Concentration inequalities:
  A nonasymptotic theory of independence}, Oxford university press, 2013.

\bibitem{calder2018game}
{\sc J.~Calder}, {\em The game theoretic p-{L}aplacian and semi-supervised
  learning with few labels}, Nonlinearity, 32 (2018), pp.~301--330.

\bibitem{calder2020Viscosity}
{\sc J.~Calder}, {\em Lecture notes on viscosity solutions}, Online Lecture
  Notes,  (2018).
\newblock
  \url{http://www-users.math.umn.edu/~jwcalder/viscosity_solutions.pdf}.

\bibitem{calder2019lip}
{\sc J.~Calder}, {\em {Consistency of Lipschitz learning with infinite
  unlabeled data and finite labeled data}}, SIAM Journal on Mathematics of Data
  Science, 1 (2019), pp.~780--812.

\bibitem{calder2020Calculus}
{\sc J.~Calder}, {\em The calculus of variations}, Online Lecture Notes,
  (2020).
\newblock
  \url{http://www-users.math.umn.edu/~jwcalder/CalculusOfVariations.pdf}.

\bibitem{Github:GraphLearning}
\leavevmode\vrule height 2pt depth -1.6pt width 23pt, {\em Graph-based
  clustering and semi-supervised learning}, 2020.
\newblock \url{https://github.com/jwcalder/GraphLearning}. (accessed
  2020/10/19).

\bibitem{calder2014}
{\sc J.~Calder, S.~Esedo\=glu, and A.~O. Hero}, {\em A {H}amilton-{J}acobi
  equation for the continuum limit of non-dominated sorting}, SIAM Journal on
  Mathematical Analysis, 46 (2014), pp.~603--638.

\bibitem{calder2021hamilton}
{\sc J.~Calder and M.~Ettehad}, {\em Hamilton-{J}acobi equations on graphs with
  applications to semi-supervised learning and data depth}, In preparation,
  (2021).

\bibitem{calder2019improved}
{\sc J.~Calder and N.~Garc\'ia~Trillos}, {\em Improved spectral convergence
  rates for graph {Laplacians} on $\varepsilon$-graphs and k-{NN} graphs},
  arXiv:1910.13476,  (2019).

\bibitem{calder2020Lip}
{\sc J.~Calder, N.~Garc\'ia~Trillos, and M.~Lewicka}, {\em Lipschitz regularity
  of graph {L}aplacians on random data clouds}, arXiv:2007.06679,  (2020).

\bibitem{CST20}
{\sc J.~Calder, D.~Slep{\v{c}}ev, and M.~Thorpe}, {\em Rates of convergence for
  {L}aplacian semi-supervised learning with low labeling rates},
  arXiv:2006.02765,  (2020).

\bibitem{calder2020convex}
{\sc J.~Calder and C.~K. Smart}, {\em {The limit shape of convex hull
  peeling}}, Duke Mathematical Journal, 169 (2020), pp.~2079--2124.

\bibitem{cannarsa2004semiconcave}
{\sc P.~Cannarsa and C.~Sinestrari}, {\em Semiconcave functions,
  {H}amilton-{J}acobi equations, and optimal control}, vol.~58, Springer
  Science \& Business Media, 2004.

\bibitem{carrizosa1996characterization}
{\sc E.~Carrizosa}, {\em A characterization of halfspace depth}, Journal of
  multivariate analysis, 58 (1996), pp.~21--26.

\bibitem{chen2017meshfree}
{\sc J.-S. Chen, M.~Hillman, and S.-W. Chi}, {\em Meshfree methods: progress
  made after 20 years}, Journal of Engineering Mechanics, 143 (2017),
  p.~04017001.

\bibitem{ChGeWa17}
{\sc Y.-C. Chen, C.~R. Genovese, and L.~Wasserman}, {\em Density level sets:
  asymptotics, inference, and visualization}, J. Amer. Statist. Assoc., 112
  (2017), pp.~1684--1696.

\bibitem{BORDER06}
{\sc {Chenyi Xia}, W.~{Hsu}, M.~L. {Lee}, and B.~C. {Ooi}}, {\em Border:
  efficient computation of boundary points}, IEEE Transactions on Knowledge and
  Data Engineering, 18 (2006), pp.~289--303.

\bibitem{chernozhukov2017monge}
{\sc V.~Chernozhukov, A.~Galichon, M.~Hallin, and M.~Henry}, {\em
  Monge--kantorovich depth, quantiles, ranks and signs}, The Annals of
  Statistics, 45 (2017), pp.~223--256.

\bibitem{costa2006determining}
{\sc J.~A. Costa and A.~O. Hero}, {\em Determining intrinsic dimension and
  entropy of high-dimensional shape spaces}, in Statistics and Analysis of
  Shapes, Springer, 2006, pp.~231--252.

\bibitem{cuevas1997plug}
{\sc A.~Cuevas, R.~Fraiman, et~al.}, {\em A plug-in approach to support
  estimation}, The Annals of Statistics, 25 (1997), pp.~2300--2312.

\bibitem{CuFrGy13}
{\sc A.~Cuevas, R.~Fraiman, and L.~Gy\"{o}rfi}, {\em Towards a universally
  consistent estimator of the {M}inkowski content}, ESAIM Probab. Stat., 17
  (2013), pp.~359--369.

\bibitem{CuFrRo07}
{\sc A.~Cuevas, R.~Fraiman, and A.~Rodr\'{\i}guez-Casal}, {\em A nonparametric
  approach to the estimation of lengths and surface areas}, Ann. Statist., 35
  (2007), pp.~1031--1051.

\bibitem{CueRod04}
{\sc A.~Cuevas and A.~Rodr\'{\i}guez-Casal}, {\em On boundary estimation}, Adv.
  in Appl. Probab., 36 (2004), pp.~340--354.

\bibitem{de2020depth}
{\sc P.~L. de~Micheaux, P.~Mozharovskyi, and M.~Vimond}, {\em Depth for curve
  data and applications}, Journal of the American Statistical Association,
  (2020), pp.~1--17.

\bibitem{DevWis80}
{\sc L.~Devroye and G.~L. Wise}, {\em Detection of abnormal behavior via
  nonparametric estimation of the support}, SIAM J. Appl. Math., 38 (1980),
  pp.~480--488.

\bibitem{DongMosLi11}
{\sc W.~Dong, C.~Moses, and K.~Li}, {\em Efficient k-nearest neighbor graph
  construction for generic similarity measures}, in Proceedings of the 20th
  International Conference on World Wide Web, WWW ’11, New York, NY, USA,
  2011, Association for Computing Machinery, p.~577–586.

\bibitem{Edels10}
{\sc H.~Edelsbrunner}, {\em Alpha shapes—a survey}, Tessellations in the
  Sciences,  (2010).

\bibitem{EdelKirkSeidel83}
{\sc H.~{Edelsbrunner}, D.~{Kirkpatrick}, and R.~{Seidel}}, {\em On the shape
  of a set of points in the plane}, IEEE Transactions on Information Theory, 29
  (1983), pp.~551--559.

\bibitem{EdelsMucke94}
{\sc H.~Edelsbrunner and E.~P. M\"{u}cke}, {\em Three-dimensional alpha
  shapes}, ACM Trans. Graph., 13 (1994), p.~43–72.

\bibitem{finlay2019improved}
{\sc C.~Finlay and A.~Oberman}, {\em Improved accuracy of monotone finite
  difference schemes on point clouds and regular grids}, SIAM Journal on
  Scientific Computing, 41 (2019), pp.~A3097--A3117.

\bibitem{flores2019algorithms}
{\sc M.~Flores, J.~Calder, and G.~Lerman}, {\em {Analysis and algorithms for
  {L}p-based semi-supervised learning on graphs}}, arXiv:1901.05031,  (2019).

\bibitem{flyer2009radial}
{\sc N.~Flyer and G.~B. Wright}, {\em A radial basis function method for the
  shallow water equations on a sphere}, Proceedings of the Royal Society A:
  Mathematical, Physical and Engineering Sciences, 465 (2009), pp.~1949--1976.

\bibitem{foote1984regularity}
{\sc R.~L. Foote}, {\em Regularity of the distance function}, Proceedings of
  the American Mathematical Society, 92 (1984), pp.~153--155.

\bibitem{froese2018meshfree}
{\sc B.~D. Froese}, {\em Meshfree finite difference approximations for
  functions of the eigenvalues of the {H}essian}, Numerische Mathematik, 138
  (2018), pp.~75--99.

\bibitem{fuselier2012scattered}
{\sc E.~Fuselier and G.~B. Wright}, {\em Scattered data interpolation on
  embedded submanifolds with restricted positive definite kernels: Sobolev
  error estimates}, SIAM Journal on Numerical Analysis, 50 (2012),
  pp.~1753--1776.

\bibitem{GGHS2020}
{\sc N.~Garc{\'\i}a~Trillos, M.~Gerlach, M.~Hein, and D.~Slep{\v{c}}ev}, {\em
  Error estimates for spectral convergence of the graph {L}aplacian on random
  geometric graphs toward the {L}aplace--{B}eltrami operator}, Foundations of
  Computational Mathematics, 20 (2020), pp.~827--887.

\bibitem{garcia2020maximum}
{\sc N.~Garc\'ia~Trillos and R.~W. Murray}, {\em A maximum principle argument
  for the uniform convergence of graph {L}aplacian regressors}, SIAM Journal on
  Mathematics of Data Science, 2 (2020), pp.~705--739.

\bibitem{hein2005intrinsic}
{\sc M.~Hein and J.-Y. Audibert}, {\em Intrinsic dimensionality estimation of
  submanifolds in rd}, in Proceedings of the 22nd international conference on
  Machine learning, 2005, pp.~289--296.

\bibitem{LacVeg17}
{\sc R.~Lachi\`eze-Rey and S.~Vega}, {\em Boundary density and {V}oronoi set
  estimation for irregular sets}, Trans. Amer. Math. Soc., 369 (2017),
  pp.~4953--4976.

\bibitem{lai2013local}
{\sc R.~Lai, J.~Liang, and H.-K. Zhao}, {\em A local mesh method for solving
  pdes on point clouds}, Inverse Problems \& Imaging, 7 (2013), p.~737.

\bibitem{lecun1998gradient}
{\sc Y.~LeCun, L.~Bottou, Y.~Bengio, and P.~Haffner}, {\em Gradient-based
  learning applied to document recognition}, Proceedings of the IEEE, 86
  (1998), pp.~2278--2324.

\bibitem{li2017point}
{\sc Z.~Li, Z.~Shi, and J.~Sun}, {\em Point integral method for solving
  poisson-type equations on manifolds from point clouds with convergence
  guarantees}, Communications in Computational Physics, 22 (2017),
  pp.~228--258.

\bibitem{liang2013solving}
{\sc J.~Liang and H.~Zhao}, {\em Solving partial differential equations on
  point clouds}, SIAM Journal on Scientific Computing, 35 (2013),
  pp.~A1461--A1486.

\bibitem{liang2021solving}
{\sc S.~Liang, S.~W. Jiang, J.~Harlim, and H.~Yang}, {\em Solving pdes on
  unknown manifolds with machine learning}, arXiv:2106.06682,  (2021).

\bibitem{liu1999multivariate}
{\sc R.~Y. Liu, J.~M. Parelius, and K.~Singh}, {\em Multivariate analysis by
  data depth: descriptive statistics, graphics and inference,(with discussion
  and a rejoinder by liu and singh)}, The annals of statistics, 27 (1999),
  pp.~783--858.

\bibitem{mcmullen_1970}
{\sc P.~McMullen}, {\em The maximum numbers of faces of a convex polytope},
  Mathematika, 17 (1970), p.~179–184.

\bibitem{molina2021eikonal}
{\sc M.~Molina-Fructuoso and R.~Murray}, {\em Eikonal depth: an optimal control
  approach to statistical depths}, In preparation,  (2021).

\bibitem{molina2021tukey}
\leavevmode\vrule height 2pt depth -1.6pt width 23pt, {\em Tukey depths and
  {H}amilton-{J}acobi differential equations}, arXiv:2104.01648,  (2021).

\bibitem{oberman2008wide}
{\sc A.~M. Oberman}, {\em Wide stencil finite difference schemes for the
  elliptic {M}onge-{A}mpere equation and functions of the eigenvalues of the
  {H}essian}, Discrete \& Continuous Dynamical Systems-B, 10 (2008), p.~221.

\bibitem{piret2012orthogonal}
{\sc C.~Piret}, {\em The orthogonal gradients method: A radial basis functions
  method for solving partial differential equations on arbitrary surfaces},
  Journal of Computational Physics, 231 (2012), pp.~4662--4675.

\bibitem{piret2016fast}
{\sc C.~Piret and J.~Dunn}, {\em Fast rbf ogr for solving pdes on arbitrary
  surfaces}, in AIP Conference Proceedings, vol.~1776, AIP Publishing LLC,
  2016, p.~070005.

\bibitem{QiaPol19}
{\sc W.~Qiao and W.~Polonik}, {\em Nonparametric confidence regions for level
  sets: statistical properties and geometry}, Electron. J. Stat., 13 (2019),
  pp.~985--1030.

\bibitem{BRIM07}
{\sc B.-Z. Qiu, F.~Yue, and J.-Y. Shen}, {\em Brim: An efficient boundary
  points detecting algorithm}, in Advances in Knowledge Discovery and Data
  Mining, Z.-H. Zhou, H.~Li, and Q.~Yang, eds., Berlin, Heidelberg, 2007,
  Springer Berlin Heidelberg, pp.~761--768.

\bibitem{Casal07}
{\sc A.~Rodr{\'\i}guez~Casal}, {\em Set estimation under convexity type
  assumptions}, Annales de l'I.H.P. Probabilit\'es et statistiques, 43 (2007),
  pp.~763--774.

\bibitem{sethian2000fast}
{\sc J.~A. Sethian and A.~Vladimirsky}, {\em Fast methods for the eikonal and
  related {H}amilton--{J}acobi equations on unstructured meshes}, Proceedings
  of the National Academy of Sciences, 97 (2000), pp.~5699--5703.

\bibitem{Shi17}
{\sc Z.~Shi}, {\em Enforce the {D}irichlet boundary condition by volume
  constraint in point integral method}, Commun. Math. Sci., 15 (2017),
  pp.~1743--1769.

\bibitem{small1997multidimensional}
{\sc C.~G. Small}, {\em Multidimensional medians arising from geodesics on
  graphs}, The Annals of Statistics,  (1997), pp.~478--494.

\bibitem{suchde2019fully}
{\sc P.~Suchde and J.~Kuhnert}, {\em A fully lagrangian meshfree framework for
  pdes on evolving surfaces}, Journal of Computational Physics, 395 (2019),
  pp.~38--59.

\bibitem{suchde2019meshfree}
{\sc P.~Suchde and J.~Kuhnert}, {\em A meshfree generalized finite difference
  method for surface pdes}, Computers \& Mathematics with Applications, 78
  (2019), pp.~2789--2805.

\bibitem{aShp_MATLAB}
{\sc {The MathWorks Inc.}}, {\em alphashape: Matlab documentation}.
\newblock \url{https://www.mathworks.com/help/matlab/ref/alphashape.html}.
\newblock Accessed: 2021-10-17.

\bibitem{WuWu19}
{\sc H.~tieng Wu and N.~Wu}, {\em When locally linear embedding hits boundary},
  arXiv:1811.04423,  (2019).

\bibitem{trask2020compatible}
{\sc N.~Trask and P.~Kuberry}, {\em Compatible meshfree discretization of
  surface pdes}, Computational Particle Mechanics, 7 (2020), pp.~271--277.

\bibitem{tukey1975mathematics}
{\sc J.~W. Tukey}, {\em Mathematics and the picturing of data}, in Proceedings
  of the International Congress of Mathematicians, Vancouver, 1975, vol.~2,
  1975, pp.~523--531.

\bibitem{vaughn2019diffusion}
{\sc R.~Vaughn, T.~Berry, and H.~Antil}, {\em Diffusion maps for embedded
  manifolds with boundary with applications to pdes}, arXiv preprint
  arXiv:1912.01391,  (2019).

\bibitem{wang2018modified}
{\sc M.~Wang, S.~Leung, and H.~Zhao}, {\em Modified virtual grid difference for
  discretizing the {L}aplace--{B}eltrami operator on point clouds}, SIAM
  Journal on Scientific Computing, 40 (2018), pp.~A1--A21.

\bibitem{xiao2017fashion}
{\sc H.~Xiao, K.~Rasul, and R.~Vollgraf}, {\em Fashion-{MNIST}: A novel image
  dataset for benchmarking machine learning algorithms}, arXiv:1708.07747,
  (2017).

\bibitem{yuan2020continuum}
{\sc A.~Yuan, J.~Calder, and B.~Osting}, {\em A continuum limit for the
  pagerank algorithm}, European Journal of Applied Mathematics,  (2020),
  pp.~1--33.

\end{thebibliography}


\newpage 

\appendix 
\section{Proof of Lemma \ref{lem:hatd1_lowbd}}\label{appendix:proof_lowbd}

The following lemma will be useful in proving Lemma \ref{lem:hatd1_lowbd}.

\begin{lemma}[Covering with spherical segments]\label{lem:covering}
    Let $\rr\leq 1$ and $0<a<b\leq \rr$. For $u\in \S^{d-1}$ and $0 < a < b\leq r$  define the spherical sector by
    \[S_{a,b}^u = \{x\in B(0,r) \, : \, a \leq x\cdot u\leq b\}.\]
    
    Suppose $\Sigma\subset\S^{d-1}$ is a finite set satisfying the following property:
    \begin{equation}\label{eq:res}
    \text{ for all } u\in\S^{d-1} \text{ there exists } v\in\Sigma \text{ such that } |u-v|\leq \delta.
    \end{equation}
    Then, for any $u\in\S^{d-1}$ we can find $v\in\Sigma$ such that
    \[S^{v}_{a+\delta b,b-\delta b}\subset S^u_{a,b}.\]
\end{lemma}

\begin{proof}
Let $u\in \S^{d-1}$ and fix a $v\in \Sigma$ satisfying \eqref{eq:res}. Suppose that $x\in S_{a+\delta b,b-\delta b}^v$. Then we have
\[a+\delta b \leq x\cdot v \leq b-\delta b.\]
We have
\[|x\cdot v - x\cdot u| =| x\cdot (v-u)| \leq |x| |u-v| \leq \delta |x| \leq \delta b,\]
since $|x|\leq b-\delta \leq b$. Therefore 
\[x \cdot u \leq b-\delta b + \delta b=b  \ \ \text{and}  \ \ x\cdot u \geq a + \delta b - \delta b=a .\]
Therefore $x\in S_{a,b}^u$, which shows that for each $u\in \S^{d-1}$ there exists $v\in \Sigma$ such that
\[S_{a,b}^u \supset S_{a+\delta b,b-\delta b}^v.\]
Hence, the event that $S_{a,b}^u $ is empty for some $u\in \S^{d-1}$ is contained in the event that $S_{a+\delta b,b-\delta b}^v$ is empty for some $v\in \Sigma$---a finite collection of events.

\end{proof}

\begin{remark}[$\eps$-nets and upper bound on $|\Sigma|$]\label{rmk:delta_net}

   Recall that an $\eps$-net of $\S^{d-1}$ is the set of points in $\S^{d-1}$ such that the pairwise distance is at least $\eps$. Then we define a maximal $\eps$-net of the sphere to be an $\eps$-net such that no point on $\S^{d-1}$ can be added while preserving the lower bound for the pairwise distance. 
   
   Then, observe that any maximal $\eps$-net of the unit sphere satisfies the condition of Lemma \ref{lem:covering}. If $\Sigma_\eps=\{x^1,\cdots,x^{N_\eps}\}$ is a maximal $\eps$-net of $\S^{d-1}$, then for each $x\in\S^{d-1}$ there exists $x^i\in\Sigma_{\eps}$ such that $|x-x^i|\leq \eps$. To see this, suppose $|x^\ast - x^i|>\eps$ for all $i=1,\cdots,N_\eps$. Then 
   \[B(x^\ast,\eps/2)\cap B(x^i,\eps/2) =\emptyset \text{ for all } x^i\in\Sigma_\eps.\]
   Thus $\Sigma_{\eps}\cap\{x^\ast\}$ should also be an $\eps$-net, which contradicts the maximality of $\Sigma_\eps$.
   
   Now, let $\Sigma_\delta$ be any $\delta$-net -- i.e. $\eps$-net with $\eps=\delta$. Then $\{B(v^i,\delta/2):\,v^i\in\Sigma_\delta\}$ is a collection of disjoint balls, all contained in $B(0,1+\delta/2)\setminus B(0,1-\delta/2)$. Thus, base on a simple volumetric argument, we can deduce
   \begin{equation}\label{eq:delta_net;cardinality}
       |\Sigma_\delta| \leq 2d\left(1+\frac{2}{\delta}\right)^{d-1},
   \end{equation}
\end{remark}

\begin{proof}[Proof of Lemma \ref{lem:hatd1_lowbd}]
$ $
\begin{enumerate}
    \item Let $\{v_i\}_{i=1}^M=\Sigma\subset\S^{d-1}$ be a maximal $\delta$-net. By Lemma \ref{lem:covering} and Remark \ref{rmk:delta_net}, for any $u\in\S^{d-1}$ we can find $v_k\in \Sigma$ such that
   \[S_{a+b\delta,b-b\delta}^{v_k}\subset S_{a,b}^u.\]
   This means that if all of $S_{a+b\delta,b-b\delta}^{v_i}$ are nonempty, all of $S_{a,b}^u$ is nonempty for $u\in\S^{d-1}$ hence
   \[\hat d_r^1(x^0)\geq a.\]
   Without loss of generality, assume $x^0\in\R^d$ is the origin, and let $\alpha=d_\Omega(x^0)\wedge \frac{\rr}{2}$. Denote by $K_{a,b}^u\subset S_{a,b}^u$ the cone of maximal height sharing the base with $S_{a,b}^u$. Note that $b\leq \alpha$ implies $K_{a,b}^u\subset \overline B(x_0,\rr)\cap\Omega$. On the other hand, we need $a\geq (1-\lambda)\alpha-t$ to deduce the desired lower bound on $\hat d_\rr^1$. Thus choose
   \[a=(1-\lambda)\alpha-t,\,b=\alpha.\]
   Further, we need the height of $S_{a+b\delta,b-b\delta}^{v_i}$ to scale like $t$, in order to lower bound the volume. Thus we need
   \[
   b-b\delta-(a+b\delta)=(1-2\delta)b-\alpha=(1-2\delta)\alpha-(1-\lambda)\alpha-t 
   =(\lambda-2\delta)\alpha+t.
   \]
   As we are interested in $t\lesssim \rr^2\ll \alpha\sim \eps$, we need $\lambda-2\delta\geq0$, hence
   \[\delta\leq\frac{\lambda}{2}.\]
   
   \item Following the discussion in the previous step, let $\Sigma=\{v^1,\cdots,v^{N_\lambda}\}$ be a maximal $\frac{\lambda}{2}$-net of $\S^{d-1}$, and write
   \[S^i = S^{v^i}_{a+b\lambda/2,b-b\lambda/2} \text{ where } a=(1-\lambda) \alpha,\,b=\alpha,\text{ and }. \]
   Thus, to show \eqref{eq:hatd1_lowbd} holds with probability at least $1-n^{-\gamma}$, it suffices to show
   \[\P(\text{ No point in } S^i)\leq (1-\rho_{\min}|S^i\cap \Omega|)^{n}\leq N_\lambda^{-1}n^{-\gamma} \text{ for all } i=1,\cdots,N_\lambda.\]
   
   \item We first compute the lower bound for $|S^i\cap \Omega|$. 
   Temporarily write $a'=a+b\lambda/2,\,b'=b-b\lambda/2$. Let $K_{a',b'}^i$ be the cone of height $b'-a'=t$ sharing the base of $S^i$. Note that $K_{a',b'}^i\subset S^i\cap\Omega$ and its base has radius $\sqrt{r^2-(a')^2}=\rr\sqrt{1-(a'/\rr)^2}$. As the $|K_{a',b'}^i|$ is independent of $i$, we may drop the superscript and deduce
   \begin{align*}
        |S^i\cap\Omega|\geq |K_{a',b'}|=\int_0^t \omega_{d-1}\left(\rr\sqrt{1-(a'/\rr)^2}\frac{s}{t}\right)^{d-1}\,ds
        =\frac1d\omega_{d-1}t\rr^{d-1}(1-(a'/\rr)^2)^{\frac{d-1}{2}}.
   \end{align*}
   As $a'\leq b\leq \alpha\leq \rr/2$, we have $(1-(a'/r)^2)^{(d-1)/2}\geq 2^{-(d-1)/2}$.
   Hence, for each $i=1,\cdots,N_\lambda$
   \[\P(\text{ No point in } S^i)\leq (1-\rho_{\min}|K_{a',b'}|)^{n}\leq \left(1-\frac{\rho_{\min}}{d2^{(d-1)/2}}t\rr^{d-1}\right)^{n}.\]
   The expression on the right is less than $N_\lambda^{-1}n^{-\gamma}$ if
   \[n\log\left(1-\frac{\rho_{\min}}{d2^{(d-1)/2}}t\rr^{d-1}\right)\leq -\gamma\log n-\log N_\lambda,\]
   or equivalently
   \[t\rr^{d-1}\geq \frac{d2^{(d-1)/2}(1-e^{-\frac{\gamma\log n+\log N_\lambda}{n}})}{\rho_{\min}\omega_{d-1}}.\]
   As $1-e^{-x}\leq x$, it suffices for $t,\rr$ to satisfy
   \[t\rr^{d-1}\geq\frac{d2^{(d-1)/2}}{\rho_{\min}\omega_{d-1}}\left(\frac{\gamma\log n+\log N_\lambda}{n}\right).\]
   
   \item We claim that $\log N_\lambda\leq \gamma(d-1)\log n$. By setting $\delta=\frac{\lambda}{2}$ in \eqref{eq:delta_net;cardinality}, we know 
   \[N\leq 2d\left(1+\frac{4}{\lambda}\right)^{d-1}=2d\left(\frac{\lambda+4}{\lambda}\right)^{d-1}.\]
   By hypothesis $n\geq d\vee \frac{\lambda+4}{\lambda}$ and $\gamma>2$, we see
   \[n^{\gamma(d-1)}\geq n^{d-1} n^{d-1}\geq 2d \left(\frac{\lambda+4}{\lambda}\right)^{d-1}\geq N_\lambda.\]
   Thus $\gamma\log n +\log N\leq d\gamma \log n$, and it suffices for $t,\rr$ to satisfy
   \[t\rr^{d-1}\geq\frac{d^2 2^{(d-1)/2}\gamma}{\rho_{\min}\omega_{d-1}}\left(\frac{\log n}{n}\right).\]
   This completes the proof
   
\end{enumerate}
\end{proof}

\section{Proof of Proposition \ref{prop:dist}}
\begin{proof}
The proof is split into several steps.

1. Let $y\in \partial \Omega$ satisfy $d_\Omega(x_*)=|x_*-y|$. Let $z\in \partial B(x^0,\eps)$ be along the line from $x_*$ to $y$. Then we have
\[d_{\Omega}(z) \leq d_\Omega(x_*) - |x_*-z|\]
and so by the property defining $x_*$ we have $x_*=z$; that is $x_*\in \partial B(x^0,\eps)$. Since $d_\Omega$ is 1-Lipschitz, we have $d_\Omega(x_*) \geq d_\Omega(x^0)-\eps$.  By a similar argument as above, we have $d_\Omega(x_*) \leq d_\Omega(x^0) - \eps$, and so
\[d_\Omega(x_*) = d_\Omega(x^0) - \eps.\]
Now, note that the function
\[g(r) = d_\Omega(x_* +  r p)\]
is 1-Lipschitz and satisfies $g(\eps) = d_\Omega(x^0) = g(0) + \eps$. It follows that $g(l) = g(0) + r$ for $0\leq r\leq \eps$, and so
\begin{equation}\label{eq:linear}
d_\Omega(x_* + r p) = d_\Omega(x_*) + r \ \ \text{ for }0 \leq r \leq \eps.
\end{equation}

2. Since $d_\Omega - \frac{1}{R}|x-x_*|^2$ is a concave function, there exists $q\in \R^n$ such that
\[d_\Omega(x) - d_\Omega(x_*) \leq q\cdot(x-x_*)+ \frac{1}{R}|x-x_*|^2. \] 
for all $x\in \Omega$. By \eqref{eq:linear} we have
\[r = d_\Omega(x_* + rp) - d_\Omega(x_*) \leq r q\cdot p + \frac{r^2}{R}\]
for $0 \leq r \leq \eps$. Therefore
\[q\cdot p \geq 1 - \frac{r}{R}.\]
Sending $r\to 0^+$ we find that $p\cdot q \geq 1$. 

3. We now claim that $|q|\leq 1$, which combined with $p\cdot q \geq 1$ from part 2 implies that $p=q$ and completes the proof. To see this, since $B(x^0,\eps)\subset \Omega$, we have $B(x_*,r)\subset \Omega$ for $r>0$ sufficiently small. Now, the dynamic programming principle gives
\[0 = \min_{x\in B(x_*,r)}\left\{d_\Omega(x) - d_\Omega(x_*)+|x-x_*|  \right\} \leq \min_{x\in B(x_*,r)}\left\{q\cdot(x-x_*)+|x-x_*|  \right\} + \frac{r^2}{R}.\]
Setting $x-x_*=-|x-x_*|q/|q|$ we have
\[0 \leq \min_{x\in B(x_*,r)}\left\{ |x-x_*|(1-|q|) \right\} + \frac{r^2}{R} = -r(|q|-1)_+ + \frac{r^2}{R}.\]
Sending $r\to 0^+$ we obtain $|q|\leq 1$, which completes the proof.
\end{proof}

  \section{Concentration inequalities}

 For reference, we state the Chernoff bounds, Hoeffding inequality, and the Bernstein inequality, which are concentration of measure inequalities used to control the variance of our normal and distance estimators. We refer the reader to \cite{boucheron2013concentration} for a general reference on concentration inequalties. Proofs of the exact inequalities below can also be found in \cite[Chapter 5]{calder2020Calculus}. 
 
\begin{theorem}[Chernoff bounds]\label{thm:chernoff}
Let $X_1,X_2\dots,X_n$ be a sequence of \emph{i.i.d.}~Bernoulli random variables with parameter $p\in [0,1]$ (i.e., $\P(X_i=1)=p$ and $\P(X_i=0)=1-p$). Then for any $\eps>0$ we have
\begin{equation}\label{eq:chernoffbound}
\P\left(\sum_{i=1}^n X_i \geq (1+\eps)np\right)\leq \exp\left(-\frac{np\,\eps^2}{2(1+\tfrac13 \eps)} \right),
\end{equation}
and for any $0 \leq \eps < 1$ we have
\begin{equation}\label{eq:chernoffbound2}
\P\left(\sum_{i=1}^n X_i \leq (1-\eps)np\right)\leq \exp\left(-\frac{1}{2}np\,\eps^2 \right),
\end{equation}
\end{theorem}

\begin{theorem}[Hoeffding inequality]\label{thm:hoeffding}
Let $X_1,X_2\dots,X_n$ be a sequence of \emph{i.i.d.}~real-valued random variables with finite expectation $\mu=\E[X_i]$, and write $S_n=\frac{1}{n}\sum_{i=1}^n X_i$. Assume there exists $b>0$ such that $|\X-\mu|\leq b$ almost surely.  Then for any $t>0$ we have
\begin{equation}\label{eq:Hoeffding}
\P(S_n-\mu \geq t)\leq \exp\left(-\frac{nt^2}{2b^2}\right).
\end{equation}
\end{theorem}

\begin{theorem}[Bernstein Inequality]\label{thm:bernstein}
Let $X_1,X_2\dots,X_n$ be a sequence of \emph{i.i.d.}~real-valued random variables with finite expectation $\mu=\E[X_i]$ and variance $\sigma^2=\Var(X_i)$, and write $S_n=\frac{1}{n}\sum_{i=1}^n X_i$. Assume there exists $b>0$ such that $|\X-\mu|\leq b$ almost surely. Then for any $t>0$ we have
\begin{equation}\label{eq:bernstein}
\P(S_n-\mu \geq t)\leq \exp\left( -\frac{nt^2}{2(\sigma^2 + \tfrac13 bt)} \right).
\end{equation}
\end{theorem}

\section{List of Constants}\label{appendixConstants}
    We list the explicit constants that appear in \emph{Sections \ref{sec:prelim} and \ref{sec:main}
   }.  Below $\omega_d$ is the volume of unit ball in $d$ dimensions, and $\gamma>2$ is a parameter of choice related to the error rate in the following way: $\P(\text{ Boundary test fails })=O(n^{-\gamma})$.
  
    \begin{align*}
    C_{x} &= 2\wod[d-1]+\frac{LR\wod}{\rho_{min}},    \\
    C_{y} &= \frac{\omega_{d-1}}{2(d+1)}, \\
    C_\rr &= \frac{1}{R}\max\left[\left(\frac{3\gamma\rho_{\max}d^2\wod R^2}{{C_{x}}^2 \rho_{\min}^2 }\right)^{\frac{1}{d+2}},\left(\frac{4\gamma C_yd^2 2^{(d-1)/2}}{13\rho_{\min}\wod[d-1]C_{x}}\right)^{\frac{1}{d+1}}\right],\\
    \end{align*}

\end{document}